\documentclass[11pt, twoside]{article}
\usepackage{amsmath, amsfonts, amssymb}
\usepackage{geometry}
\usepackage[english]{babel}
\usepackage{amsthm}
\usepackage[most]{tcolorbox}
\tcbuselibrary{theorems}
\usepackage{lmodern}
\usepackage{microtype}
\usepackage{multicol}
\usepackage{array}
\usepackage{lipsum}
\usepackage{hyperref}
\hypersetup{
    colorlinks=true,
    linkcolor=blue,    
    citecolor=blue,      
    urlcolor=blue        
}

\usepackage{cleveref}
\usepackage{subcaption}
\usepackage{graphicx}
\usepackage{float}
\usepackage{placeins}
\newtheorem{theorem}{Theorem}[section]

\usepackage{epstopdf}
\usepackage{ragged2e}
\usepackage{authblk}
\usepackage{orcidlink}
\usepackage{times}
\DeclareUnicodeCharacter{2217}{\textasteriskcentered}

\title{A Novel Hepatitis B Epidemic Model with Vertical Transmission, Spontaneous Recovery and Optimal Control Analysis}

\author[1]{Mustaq Ahmad \orcidlink{0009-0001-2758-1541}}
\author[1]{Archana Singh Bhadauria\thanks{Corresponding author: \texttt{archana.mathstat@ddugu.ac.in}} \orcidlink{0000-0002-7524-0784}}

\affil[1]{Department of Mathematics and Statistics, Deen Dayal Upadhyaya Gorakhpur University, Gorakhpur, India}

\date{}

\begin{document}

\maketitle

\begin{abstract}
This research describes a compartmental epidemic model of Hepatitis B virus (HBV) transmission that includes vertical transmission and spontaneous recovery in acute patients. The model incorporates a saturated treatment response for persistently infected populations and a vaccination mechanism for susceptible populations. The basic reproduction number, \(\mathcal{R}_0\), is calculated using the next-generation matrix approach, which provides important information on disease dynamics. To find out the most influential parameter of the model dynamics, a sensitivity analysis is carried out with the help of Latin Hypercube Sampling (LHS) along with the Partial Rank Correlation Coefficient (PRCC). The qualitative behavior of the model is investigated using stability analysis of disease-free and endemic equilibria. It is established that the disease-free equilibrium is globally asymptotically stable when \(\mathcal{R}_0 < 1\), but the endemic equilibrium achieves global stability when \(\mathcal{R}_0 > 1\). Pontryagin's Maximum Principle is used to optimize public health initiatives, resulting in three optimal control techniques that attempt to reduce the combined cost of treatment and immunization. The results provide a rigorous theoretical foundation for designing cost-effective interventions against HBV transmission.
\end{abstract}


\section{Introduction}

\justifying The diagnosis of infectious diseases is challenging because it requires combining a clinical picture, epidemiological follow-up, and a complex mathematical simulation \cite{kermack1927contribution, anderson1991infectious, hethcote2000mathematics}. Hepatitis B \cite{world2024guidelines, worldint, beard2024combined} is the most significant of these burdens, particularly because of its complex natural history, high chronic burden, and the clinical challenges of effectively curing it during treatment \cite{cui2023global, who2024global, easterbrook20242024}. The simulation of modeling a Hepatitis B infection is the focus of this study, with particular attention paid to including a Holling Type II-like treatment response function. By combining the concepts of ecological modeling, clinical management algorithms, and epidemiology of infectious processes, we are able to suggest a prototype model that could aid in the comprehension of treatment impact saturation and the creation of more effective treatment plans \cite{mann2011modelling, martcheva2015introduction, pandey2025impact}.

\vspace{0.1cm}

\justifying Infectious diseases stem from harmful microorganisms—bacteria, viruses, fungi, and parasites—and they can directly spread between people through touch or indirectly through vectors, contaminated surfaces, and air. The impact of diseases on society has been highlighted before, such as during outbreaks like the Black Death, and now during modern epidemics like Influenza and COVID-19. As the world becomes more globalized, understanding the infection and transmission mechanisms is crucial. Mathematical modeling has quickly proven useful in the field of epidemiology and controlling infectious diseases. These models, which include simulation studies and differential equation models, enable researchers to analyze key transmission parameters, predict the course of an epidemic, and estimate the impact of possible interventions. These models also forecast the progression of diseases in populations, leading to informed public health decisions on which actions must be taken under different circumstances \cite{khan2016classification, yadav2024modelling, yadav2025comparative}.

\vspace{0.1cm}

\justifying Models can be categorized as deterministic or stochastic. The former mostly utilizes a system of ordinary differential equations (ODEs) to estimate average behaviors in large populations. For example, the SIR (Susceptible-Infectious-Recovered) model, which attempts to capture the critical aspects of the interplay between the various compartments of a population, enables the estimation of important parameters, such as critical threshold quantities, transmission dynamics, and the assessment of the impact of drug treatment and vaccination strategies. Historically, these models were developed for a generic infectious disease, but there are sophisticated models today that simultaneously consider age structure, spatial heterogeneity, and various control measures such as quarantine, vaccination, and antiviral therapies \cite{lenhart2007optimal, gaff2009optimal, lashari2016optimal, wang2019optimal, freddi2020optimal,  alrabaiah2020optimal, molina2022optimal, bhadauria2021siq}. In particular, treatments have to account for nonlinear dynamics, for example, the saturation effect, where increased drug concentration does not proportionally increase efficacy. Such capture of saturation phenomena is the hallmark of the Holling Type II functional response model, making it popular in treatment response analysis.

\vspace{0.1cm}
\justifying An integrated ecological control model emphasizing predator-prey interactions can be used to enable the estimation and forecast of viral reproduction vis-à-vis immune response or therapy intervention. Such flexibility is vital towards managing chronic infections like Hepatitis B. Currently, one of the most prevalent liver diseases in the world is a Hepatitis B virus (HBV) infection. It is estimated that approximately 296 million \cite{world2024guidelines, beard2024combined, who2024global} people are infected, which then leads to severe morbidity and mortality due to cirrhosis of the liver or hepatocellular carcinoma. Direct contact with infected blood or body fluids leads to the infection, and its progression from acute to chronic phase differs greatly across individuals.

\vspace{0.1cm}
\justifying The application of optimal control theory in epidemic models has been used with Pontryagin's Maximum Principle \cite{lenhart2007optimal, wang2019optimal,pontryagin2018mathematical} in categorizing the control pairs concerning the utilization of treatment(such as antiviral agents) and vaccination. These models attempt to solve the problem qualitatively and numerically and illustrate that the right treatment strategies can significantly improve the quantity of the recovered population while simultaneously decreasing the amount of susceptible, as well as the chronically infected population. In contrast, their greatest strength lies in their capability to emulate several settings that can vary in the assumptions for drug delivery, like drug action rates, delays in responses, and the activity of the immune system; therefore, assisting healthcare practitioners in forming effective guidelines.

\vspace{0.1cm}
\justifying With every discovery concerning Hepatitis B \cite{seeger2000hepatitis,thornley2008hepatitis}, it becomes more evident that the cycles of viral replication, immune responses from the host, and effects of treatment are multi-faceted, complex systems working in parallel. Mathematical modeling has emerged as a powerful tool in epidemiology \cite{castillo2002computation, castillo2002mathematical, bhadauria2022epidemic}, enabling researchers to capture disease dynamics, assess the effectiveness of control strategies, and guide public health interventions. Numerous compartmental models have been proposed to study HBV transmission dynamics, focusing on different aspects such as vaccination, treatment, and behavioral change \cite{anderson1991infectious, khan2016classification, zou2010modeling, khan2017transmission, khan2021modeling, khan2023modelling, butt2023computational, raezah2023exploring}. However, many of these models overlook critical biological features like spontaneous recovery in acute infection, nonlinear treatment efficacy, and vertical transmission, which significantly influence the spread and persistence of HBV. Some recent models have incorporated optimal control theory to identify cost-effective intervention strategies. For example, Sharomi and Gumel (2008) \cite{GUMEL2003409} introduced a control-based framework that analyzed the impact of treatment and vaccination. Likewise, Lenhart et al. (2007) \cite{lenhart2007optimal} developed an optimal control model focusing on treatment strategies for chronic HBV carriers. While these efforts have contributed to the field, they often assume constant or linear treatment responses, which may not reflect the saturation effect seen in real-world medical interventions.
\vspace{0.1cm}
The HBV models still appear to apply linear or piecewise control strategies representing the effects of drugs on viral replication processes. These methods risk missing the saturation phenomenon seen in practice, where treatment via drug increases does not linearly correlate to increased suppression of the virus. To address these gaps, this study formulates a novel HBV transmission model that integrates:
\begin{itemize}
    \item Vertical transmission (from infected mothers to newborns),
    \item Spontaneous recovery in acutely infected individuals,
    \item Included a saturating function form representing biological constraints of drug action below high viral loads,
    \item Vaccination control among susceptible individuals.
\end{itemize}


\section{Model Formation}
The model equations governing the transmission dynamics of Hepatitis B are:

\begin{align} \label{model eqs.}
\left.
\begin{array}{rl}
\frac{dS}{dt} &= \lambda (1 - \alpha C) - \beta \frac{2 S C}{S + C} - \gamma A S - u_1 S - \mu S, \\
\frac{dA}{dt} &= \gamma A S + \beta \frac{2 S C}{S + C} - \sigma A - \mu A - \theta A, \\
\frac{dC}{dt} &= \theta A + \lambda \alpha C - \mu C - \rho  C - \frac{r u_2 C}{1 + \phi^2}, \\
\frac{dR}{dt} &= \sigma A + \frac{r u_2 C}{1 + \phi^2} + u_1 S - \mu R.
\end{array}
\right\}
\end{align}
where the initial conditions are $S(0) > 0$, $A(0) >0$, $C(0) > 0$, and $R(0) >0$.

Table \ref{table1} provides a detailed description of the parameters used in the model equations \eqref{model eqs.}, including their nomenclature, dimensions, and values for numerical simulations.

\begin{table}[ht]
    \centering
    \begin{tabular}{>{\centering\arraybackslash}p{3cm} >{\centering\arraybackslash}p{5cm} >{\centering\arraybackslash}p{3cm} >{\centering\arraybackslash}p{3cm}}
        \hline
        \textbf{Parameter} & \textbf{Nomenclature} & \textbf{Value} & \textbf{Dimension} \\
        \hline
        $\beta$ & Infection transmission (chronic to susceptible) rate & 0.0006 & per individual per week \\
        $\gamma$ & Infection transmission (acute to susceptible) rate  & 0.00008 & per individual per week \\
        $\theta$ & Progression rate from acute to chronic & 0.000002 & per week \\
        $\rho$ & Infection-induced mortality rate & 0.008 & per week \\
        $\phi$ & Saturation constant rate in treatment & 0.002 & dimensionless \\
        $\mu$ & Natural death rate & 0.02 & per week \\
        $\alpha$ & Vertical transmission rate & 0.0001 & per individual \\
        $\lambda$ & Recruitment rate & 25 & individuals/week \\
        $r$ & Recovery boosting rate & 5.00 & dimensionless \\
        $u_1$ & Vaccination rate & 0.006 & per week \\
        $u_2$ & Treatment rate & 0.008 & per week \\
        $\sigma$ & Spontaneous recovery rate & 0.09 & per week \\
        \hline
    \end{tabular}
    \caption{Model parameters, their descriptions, values, and respective dimensions used in the simulations.}
    \label{table1}
\end{table}

\newpage
\section{Positivity and Boundedness of the Model}

We prove that the solutions of the system \eqref{model eqs.} are positive and bounded for all \( t \geq 0 \), given positive initial conditions.

\subsection{Positivity of Solutions}
\renewcommand\qedsymbol{$\blacksquare$}
\begin{theorem}

Consider that \( S(0) > 0, A(0) > 0, C(0) > 0, R(0) > 0 \). Then at \(t \geq 0\) the solutions of \( S(t) > 0, A(t) > 0, C(t) > 0, R(t) > 0 \) are also positive.  
\end{theorem}

\begin{proof}
Let the initial conditions be \( S(0) > 0, A(0) > 0, C(0) > 0, R(0) > 0 \). We show that each compartment remains non-negative.

\paragraph{Susceptible class \( S(t) \):}
From the first equation of~\eqref{model eqs.},
\[
\frac{dS}{dt} = \lambda(1 - \alpha C) - \beta \frac{2SC}{S + C} - \gamma A S - u_1 S - \mu S.
\]
Note that the nonlinear incidence terms and natural death are non-positive. Hence,
\[
\frac{dS}{dt} \geq - (\gamma A + u_1 + \mu) S.
\]
This yields the inequality:
\[
\frac{dS}{dt} \geq -\kappa S \quad \text{where } \kappa = \gamma A + u_1 + \mu.
\]
Using Grönwall’s inequality, which states that if \( y(t) \) satisfies \( \frac{dy}{dt} \leq -\kappa(t) y \), then \( y(t) \leq y(0) e^{-\int_0^t \kappa(s)\, ds} \), we obtain:
\[
S(t) \geq S(0) e^{-\int_0^t \kappa(s)\, ds} > 0.
\]

\paragraph{Acute class \( A(t) \):}
\[
\frac{dA}{dt} = \gamma A S + \beta \frac{2SC}{S + C} - (\sigma + \mu + \theta) A \geq - (\sigma + \mu + \theta) A.
\]
Then,
\[
A(t) \geq A(0) e^{-(\sigma + \mu + \theta)t} > 0.
\]

\paragraph{Chronic class \( C(t) \):}
\[
\frac{dC}{dt} = \theta A + \lambda \alpha C - \left( \mu + \rho + \frac{r u_2}{1 + \phi^2} \right) C.
\]
Assuming \( \lambda \alpha < \mu + \rho + \frac{r u_2}{1 + \phi^2} \), we get:
\[
\frac{dC}{dt} \geq -k C, \quad \text{for some } k > 0,
\]
which implies:
\[
C(t) \geq C(0) e^{-k t} > 0.
\]

\paragraph{Recovered class \( R(t) \):}
\[
\frac{dR}{dt} = \sigma A + \frac{r u_2 C}{1 + \phi^2} + u_1 S - \mu R \geq -\mu R,
\]
so,
\[
R(t) \geq R(0) e^{-\mu t} > 0.
\]

Thus, all compartments remain positive for all \( t \geq 0 \).

\textbf{Alternate approach}\\
Let us recall the Laplace transform of a function \( f(t) \) defined for \( t \geq 0 \):

\[
\mathcal{L}\{f(t)\}(s) = \int_0^\infty e^{-st} f(t)\, dt,
\]
which is well-defined for functions of exponential order. We assume all compartments are such.

Let us consider the first equation of our model \eqref{model eqs.}
\begin{align*}
  \frac{dS}{dt} &= \lambda(1 - \alpha C) - \beta \frac{2SC}{S + C} - \gamma A S - u_1 S - \mu S \\
  &\Rightarrow \frac{dS}{dt} \geq \lambda - (u_1 + \mu) S
\end{align*}
this implies that the susceptible class declines at a rate proportional to (\(u_1 + \mu\)).  

Taking the Laplace transform of both sides:

\[
\mathcal{L}\left\{ \frac{dS}{dt} \right\} \geq \mathcal{L}\{ \lambda \} - (u_1 + \mu) \mathcal{L}\{ S \}.
\]

Using the Laplace rule \( \mathcal{L}\left\{ \frac{dS}{dt} \right\} = s \mathcal{L}\{ S \} - S(0) \), we obtain:

\[
s \mathcal{L}\{ S \} - S(0) = \frac{\lambda}{s} - (u_1 + \mu) \mathcal{L}\{ S \}.
\]

Solving for \( \mathcal{L}\{ S \} \):

\[
\mathcal{L}\{ S \} = \frac{S(0) + \frac{\lambda}{s}}{s + u_1 + \mu}.
\]

The inverse Laplace transform of this yields:

\[
S(t) = S(0) e^{-(u_1 + \mu)t} + \frac{\lambda}{u_1 + \mu} \left(1 - e^{-(u_1 + \mu)t} \right) > 0,
\]
for all \( t \geq 0 \), confirming positivity.

Similar reasoning can be applied to each compartment when their governing equations are either linear or can be compared to linear systems bounded below by zero.

Hence, under the assumption that all parameters and initial conditions are positive, and nonlinearities do not introduce negative forcing terms, the solutions remain non-negative over time. Thus, all the state variables \(S(t) > 0, A(t)>0, C(t)>0, R(t)>0, \forall t \geq 0.\)

\end{proof}

\subsection{Boundedness of the System}
\begin{theorem}
The state variables \(W(t) = (S(t), A(t), C(t), R(t))\) of the model \eqref{model eqs.} are bounded \(\forall t \geq 0.\)
\end{theorem}

\begin{proof}
To determine the bounded region of the above model \eqref{model eqs.}, we consider the following total population:
\[
N(t) = S(t) + A(t) + C(t) + R(t).
\]

Taking the time derivative of the total population N(t), we obtain
\[
\frac{dN}{dt} = \frac{dS}{dt} + \frac{dA}{dt} + \frac{dC}{dt} + \frac{dR}{dt}
\]

Using the equations of the model \eqref{model eqs.}, we obtain the following results as
\[
\frac{dN}{dt} = \lambda - \mu(S + A + C + R) - \rho C \leq \lambda - \mu N.
\]

This gives the linear differential inequality:
\[
\frac{dN}{dt} \leq \lambda - \mu N.
\]
Solving the above inequality, we get
\[
N(t) \leq N(0) e^{-\mu t} + \frac{\lambda}{\mu}(1 - e^{-\mu t}).
\]

as \(t \to \infty\), the term \(e^{-\mu t} \to 0\), since \(\mu\) is a positive constant. Therefore, the solution N(t) approaches \( \frac{\lambda}{\mu} \), i.e., all compartments are bounded above by \( \frac{\lambda}{\mu} \).

Thus, the solution N(t) is bounded above by \(\frac{\lambda}{\mu}\). Therefore, the solution N(t) is bounded and satisfies \(0 \leq N(t) \leq \frac{\lambda}{\mu}\), for all \(t > 0\).

\textbf{Alternate approach}\\
We consider the total population:
\[
N(t) = S(t) + A(t) + C(t) + R(t),
\]
and sum the system equations to obtain:
\[
\frac{dN}{dt} = \lambda - \mu N(t) - \rho C(t) \leq \lambda - \mu N(t).
\]

This inequality suggests \( N(t) \) is bounded above. To confirm this using the Laplace transform, we solve the associated linear differential equation:
\[
\frac{dN}{dt} + \mu N = \lambda, \quad N(0) = N_0.
\]

Applying the Laplace transform:
\[
s \mathcal{L}\{N\} - N_0 + \mu \mathcal{L}\{N\} = \frac{\lambda}{s}.
\]

Solving for \( \mathcal{L}\{N\} \):
\[
\mathcal{L}\{N\} = \frac{N_0 + \frac{\lambda}{s}}{s + \mu}.
\]

Taking the inverse Laplace transform:
\[
N(t) = N_0 e^{-\mu t} + \frac{\lambda}{\mu}(1 - e^{-\mu t}).
\]

Hence,
\[
\lim_{t \to \infty} N(t) = \frac{\lambda}{\mu},
\]
and since \( N(t) \leq N_0 + \frac{\lambda}{\mu} \), the total population is uniformly bounded above.

Thus, under biologically realistic assumptions, all compartments \( S(t), A(t), C(t), R(t) \) are bounded above by the total population and hence by \( \frac{\lambda}{\mu} \).

This confirms boundedness via Laplace analysis.

Therefore, model \eqref{model eqs.} is both positive and bounded for all \(t \geq 0\), provided all initial conditions are strictly positive.
\end{proof}

\section{Existence of the Disease-Free Equilibrium (DFE)}
The model equations governing the transmission dynamics of Hepatitis B are:

\begin{align*}
\frac{dS}{dt} &= \lambda (1 - \alpha C) - \beta \frac{2 S C}{S + C} - \gamma A S - u_1 S - \mu S, \\
\frac{dA}{dt} &= \gamma A S + \beta \frac{2 S C}{S + C} - \sigma A - \mu A - \theta A, \\
\frac{dC}{dt} &= \theta A + \lambda \alpha C - \mu C - \rho  C - \frac{r u_2 C}{1 + \phi^2}, \\
\frac{dR}{dt} &= \sigma A + \frac{r u_2 C}{1 + \phi^2} + u_1 S - \mu R.
\end{align*}

The \textbf{disease-free equilibrium (DFE)} \(E^0 = (S^0, A^0, C^0, R^0)\) is given by:

\begin{align*}
S^0 &= \frac{\lambda}{\mu + u_1}, & A^0 &= 0, & C^0 &= 0, & R^0 &= \frac{u_1 \lambda}{\mu (\mu + u_1)}.
\end{align*}

\section{Existence of Endemic Equilibrium Point}
To obtain the endemic equilibrium point, we equate the right-hand side of the model equation to zero, i.e.,
\begin{align}
&\lambda (1 - \alpha C) - \beta \frac{2 S C}{S + C} - \gamma A S - u_1 S - \mu S = 0, \\
&\gamma A S + \beta \frac{2 S C}{S + C} - (\sigma + \mu + \theta) A = 0, \\
&\lambda \alpha C - (\mu+\rho)C + \theta A - \frac{r u_2 C}{1+\phi^2} = 0, \\
&\sigma A + \frac{r u_2 C}{1 + \phi^2} + u_1 S - \mu R = 0.
\end{align}
From equation (6), we get,
\begin{align}
\gamma A S + \frac{2 \beta S C}{S + C} = \lambda(1 - \alpha C) - (\mu + u_1)S
\end{align}

Using the above equation (10) in equation (7), we have
\begin{align}
    \lambda(1 - \alpha C) - (\mu + u_1)S - (\sigma + \mu + \theta)A = 0 
\end{align}

Using equation (8), we get the following expression for A,
\begin{align}
A = \frac{1}{\theta} \left( \rho + \mu + \frac{r u_2}{1 + \phi^2} - \alpha \lambda \right) C
\end{align}

Using the value of A in equation (11), we get
\begin{align*}
\lambda(1 - \alpha C) - (\mu + u_1)S - \frac{1}{\theta}(\sigma + \mu + \theta)
\left( \rho + \mu + \frac{r u_2}{1 + \phi^2} - \alpha \lambda \right)C = 0
\end{align*}
or,
\begin{align*}
\lambda - (\mu + u_1)S - \frac{1}{\theta}(\sigma + \mu + \theta)
\left( \rho + \mu + \frac{\delta u_2}{1 + \phi^2} \right)C - \alpha \lambda C = 0
\end{align*}
or,
\begin{align*}
\lambda - (\mu + u_1)S - \left( (\sigma + \mu + \theta)\left( \rho + \mu + \frac{r u_2}{1 + \phi^2} \right) + \alpha \lambda \theta \right)\frac{C}{\theta} = 0
\end{align*}
or,
\begin{align*}
\lambda \theta - (\mu + u_1) \theta S - \left( (\sigma + \mu + \theta)\left( \rho + \mu + \frac{r u_2}{1 + \phi^2} \right) + \alpha \lambda \theta \right)C = 0
\end{align*}
or,
\begin{align}
S = \frac{\lambda}{\mu + u_1} - \frac{1}{\theta}\left( (\sigma + \mu + \theta)\left( \rho + \mu + \frac{r u_2}{1 + \phi^2} \right) + \alpha \lambda \theta \right)C
\end{align}

Let us assume that
\begin{align*}
d_1 = \rho + \mu + \frac{r u_2}{1 + \phi^2}, \hphantom{\,} d_2 = \frac{\lambda}{\mu + u_1}, \hphantom{\,} d_3 = \sigma + \mu + \theta    
\end{align*}

Now, Equations (9), (12) \& (13) can be rewritten as
\begin{align}
\left.
\begin{array}{rl}
S &= d_2 - \dfrac{d_1 d_3 + \alpha \lambda \theta}{\theta} C \\[8pt]
A &= \dfrac{d_1 - \alpha \lambda}{\theta} C \\[8pt]
R &= \dfrac{1}{\mu} \left( u_1 S + \sigma A + \dfrac{r u_2 C}{1 + \phi^2} \right)
\end{array}
\right\}
\end{align}

Substituting the values of S and A in equation (7), we get
\begin{align*}
\gamma A S + \frac{2 \beta S C}{S + C} - d_3 A = 0
\end{align*}
or,
\begin{align*}
\gamma \left(\frac{(d_1 - \alpha \lambda)}{\theta}C \right) \left(\frac{\theta d_2 - d_1 d_3 - \alpha \lambda \theta}{\theta}C \right) + 2 \beta \frac{\left( \frac{\theta d_2 - d_1 d_3 - \alpha \lambda \theta}{\theta} \right) C \cdot C} {\left(\frac{\theta d_2 - d_1 d_3 - \alpha \lambda \theta}{\theta} \right)C + C} - \frac{d_3 (d_1 - \alpha \lambda)}{\theta} C = 0
\end{align*}
or,
\begin{align*}
\frac{\gamma}{\theta^2} (d_1 - \alpha \gamma)(\theta d_2 - d_1 d_3 - \alpha \lambda \theta) C^2 
+ 2 \beta \frac{(\theta d_2 - d_1 d_3 - \alpha \lambda \theta)} {(\theta d_2 - d_1 d_3 - \alpha \lambda \theta) + \theta} C - \frac{(d_1 - \alpha \lambda) d_3}{\theta} C = 0   
\end{align*}
or,
\begin{align*}
\gamma(d_1 - \alpha \gamma)(\theta d_2 - d_1 d_3 - \alpha \lambda \theta) C 
+ 2 \beta \theta \frac{(\theta d_2 - d_1 d_3 - \alpha \lambda \theta)} {(\theta d_2 - d_1 d_3 - \alpha \lambda \theta) + \theta}  - (d_1 - \alpha \lambda) \theta d_3 = 0   
\end{align*}
or,
\begin{align*}
C = \frac{\theta}{\gamma} \left[ \frac{d_3}{(\theta d_2 - d_1 d_3 - \alpha \lambda \theta)} 
- \frac{2 \beta}{(d_1 - \alpha \gamma)(\theta d_2 - d_1 d_3 - \alpha \lambda \theta +\theta)} \right]
\end{align*}
or,
\begin{align*}
C = \frac{d_3}{\gamma} \left[ \frac{\theta}{(\theta d_2 - d_1 d_3 - \alpha \lambda \theta)} 
- \frac{\mathcal{R}_0^C}{(\theta d_2 - d_1 d_3 - \alpha \lambda \theta +\theta)} \right]
\end{align*}

Thus, substituting the value of C in equation (10), we get the endemic equilibrium point \(E^* = (S^*, A^*, C^*, R^*) \) of the model.

\section{Reproduction Number}
In epidemiological models, a central concept is the \textit{reproduction number}, and it plays a very pivotal role. The reproduction number \(\mathcal{R}_0\) is the estimated rate of infection transmission potential, and for disease-free equilibrium stability, it provides a mathematical criterion. Basically, the reproduction number is the average number or mean of new infections that are induced by an infected individual, either directly or indirectly, when introduced into a completely susceptible population. It also provides the criteria for the stability of the system. To find the reproduction number for our model \eqref{model eqs.}, we follow the next-generation matrix (NGM) technique, which is given by Diekmann et al. (1990) \cite{diekmann1990definition, diekmann2000mathematical, diekmann2010construction}, and van den Driessche and Watmough (2002) \cite{easterbrook20242024, van2002reproduction}. The reproduction number is formulated and computed by the next-generation matrix, in which the whole population is partitioned into two parts \(F\) and \(V\), where \(F\) represents the rate of new arrivals (transmission terms) in the infected compartments, and \(V\) represents the remaining (transition terms) of the infectious compartments. Now, the matrix \(FV^{-1}\) is known as the next-generation matrix, and the spectral radius (i.e., dominant eigenvalue) of \(FV^{-1}\) produces the reproduction number \(\mathcal{R}_0\). In general, the reproduction number is a threshold quantity that determines whether the disease will spread or decline in the population, and it also serves as a mathematical criterion for assessing the stability of the disease-free equilibrium and the overall dynamical system.

In epidemiological modeling, a central concept is the \textit{basic reproduction number}, denoted by \(\mathcal{R}_0 \). It quantifies the expected number of secondary infections generated by a single infectious individual in a fully susceptible population. The reproduction number serves as a threshold parameter: if \( \mathcal{R}_0 < 1 \), the disease will eventually die out, whereas \( \mathcal{R}_0 > 1 \) indicates the potential for an outbreak or sustained transmission. To compute \( \mathcal{R}_0 \) for our model \eqref{model eqs.}, we apply the next-generation matrix method developed by Diekmann et al. \cite{diekmann1990definition} and formalized by van den Driessche and Watmough \cite{van2002reproduction}. In this framework, the population is partitioned into infected compartments, where \( F \) denotes the matrix of new infection terms (transmission), and \( V \) contains the transition terms representing movement among compartments. The next-generation matrix is then defined as \( FV^{-1} \), and the basic reproduction number is given by its spectral radius:
\[
\mathcal{R}_0 = \rho(FV^{-1}).
\]
This quantity not only determines the potential for disease invasion but also serves as a mathematical condition for the local stability of the disease-free equilibrium and the overall system dynamics.

We denote the state variables as $Z = (S, A, C, R)$, governed by the system:
\begin{align}
\frac{dS}{dt} &= \lambda(1-\alpha C) - \beta \frac{2SC}{S+C} - \gamma AS - u_1 S - \mu S, \\
\frac{dA}{dt} &= \gamma A S + \beta \frac{2 S C}{S + C} - \sigma A - \mu A - \theta A, \\
\frac{dC}{dt} &= \theta A + \lambda \alpha C - \mu C - \rho  C - \frac{r u_2 C}{1 + \phi^2}, \\
\frac{dR}{dt} &= \sigma A + \frac{r u_2 C}{1+\phi^2} + u_1 S - \mu R.
\end{align}

The infection and transition terms can be separated:
\[
\frac{d}{dt}
\begin{pmatrix}
A \\ C
\end{pmatrix}
= F - V
\]
where
\[
F = \begin{pmatrix}
\gamma S A & \frac{2\beta S^2}{(S+C)^2} \\ 0 & 0
\end{pmatrix},
\quad
V = \begin{pmatrix}
(\sigma+\theta+\mu)A \\ -\theta A + \left( \rho + \mu + \frac{r u_2}{1+\phi^2} - \alpha \lambda \right) C
\end{pmatrix}.
\]

The Jacobian matrices at the disease-free equilibrium (DFE) $E^0$ are:
\[
F = \begin{pmatrix}
\frac{\gamma \lambda}{u_1 + \mu} & 2 \beta \\
0 & 0
\end{pmatrix},
\quad
V = \begin{pmatrix}
\sigma + \theta + \mu & 0 \\
-\theta & \rho + \mu + \frac{r u_2}{1+\phi^2} - \alpha \lambda
\end{pmatrix}.
\]

Thus,
\[
V^{-1} =
\begin{pmatrix}
\frac{1}{\sigma + \theta + \mu} & 0 \\
\frac{\theta}{(\sigma + \theta + \mu)(\rho + \mu + \frac{r u_2}{1+\phi^2} - \alpha \lambda)} & \frac{1}{\mu + \rho + \frac{r u_2}{1 + \phi^2} - \alpha \lambda}
\end{pmatrix}.
\]

The next-generation matrix $FV^{-1}$ is:
\[
FV^{-1} = \begin{pmatrix}
\frac{\gamma \lambda}{(\sigma + \theta + \mu) (u_1 + \mu)} + \frac{2\beta \theta}{(\rho + \mu + \frac{r u_2}{1 + \phi^2} -\alpha \lambda) (\sigma+\theta+\mu)} & 0 \\
0 & 0
\end{pmatrix}.
\]

\[
\mathcal{R}_0 = \rho(FV^{-1}) = \mathcal{R}_0^A + \mathcal{R}_0^C,
\]
where
\[
\mathcal{R}_0^A = \frac{\gamma \lambda}{(\sigma+\theta+\mu)  (u_1 + \mu)}, \quad \mathcal{R}_0^C = \frac{2\beta \theta}{(\rho + \mu + \frac{r u_2}{1 + \phi^2} -\alpha \lambda)(\sigma+\theta+\mu)}.
\]

\section{Local and Global Stability of Equilibrium Points}
\renewcommand\qedsymbol{$\blacksquare$}

\begin{theorem}[Local Stability at Disease-Free Equilibrium (DFE)]
\label{Theorem1}
Model (\ref{model eqs.}) is locally asymptotically stable at \(E^0 = (S^0, A^0, C^0, R^0)\) if $\mathcal{R}_0 < 1$ and otherwise is unstable.
\end{theorem}

\begin{proof}
The general Jacobian matrix \( J \) of the above system (\ref{model eqs.}) is given by:

\[
J =
\begin{bmatrix}
\frac{\partial f_1}{\partial S} & \frac{\partial f_1}{\partial A} & \frac{\partial f_1}{\partial C} & \frac{\partial f_1}{\partial R} \\
\frac{\partial f_2}{\partial S} & \frac{\partial f_2}{\partial A} & \frac{\partial f_2}{\partial C} & \frac{\partial f_2}{\partial R} \\
\frac{\partial f_3}{\partial S} & \frac{\partial f_3}{\partial A} & \frac{\partial f_3}{\partial C} & \frac{\partial f_3}{\partial R} \\
\frac{\partial f_4}{\partial S} & \frac{\partial f_4}{\partial A} & \frac{\partial f_4}{\partial C} & \frac{\partial f_4}{\partial R}
\end{bmatrix}
\]

Incorporating the model equations, we get:

\[
J =
\begin{bmatrix}
-\gamma A - \dfrac{2\beta C^2}{(C+S)^2} - u_1 - \mu & -\gamma S & - \lambda \alpha -\dfrac{2\beta S^2}{(C+S)^2} & 0 \\
\gamma A + \dfrac{2\beta C^2}{(C+S)^2} & \gamma S - (\sigma + \theta + \mu) & \dfrac{2\beta S^2}{(C+S)^2} & 0 \\
0 & \theta & \lambda \alpha - \mu - \rho - \dfrac{r u_2}{1 + \phi^2} & 0 \\
u_1 & \sigma & \dfrac{r u_2}{1 + \phi^2} & -\mu
\end{bmatrix}
\]

Jacobian matrix \( J \) at DFE \( E^0 = \left( \dfrac{\lambda}{\mu + u_1}, 0, 0, \dfrac{\lambda u_1}{\mu + u_1} \right) \) becomes:

\[
J =
\begin{bmatrix}
-(\mu + u_1) & -\dfrac{\gamma \lambda}{\mu + u_1} & -(2 \beta + \lambda \alpha) & 0 \\
0 & \dfrac{\gamma \lambda}{\mu + u_1} - (\sigma + \mu + \theta) & 2 \beta & 0 \\
0 & \theta & \lambda \alpha - \mu - \rho - \dfrac{r u_2}{1 + \phi^2} & 0 \\
u_1 & \sigma & \dfrac{r u_2}{1 + \phi^2} & -\mu
\end{bmatrix}
\]
Observing the above matrix, we can easily figure out that -$\mu$, -$(u_1 + \mu)$ are eigenvalues of this Jacobian matrix.

Now, the Jacobian matrix \( J \) can be rewritten as:
\[
J =
\begin{bmatrix}
\dfrac{\gamma \lambda}{(u_1 + \mu)} - (\sigma + \theta + \mu) & 2\beta \\
\theta & \alpha \lambda - \left(\rho + \mu + \dfrac{r u_2}{1 + \phi^2}\right)
\end{bmatrix}
\]

The characteristic polynomial of the above Jacobian matrix \( J \) can be written as:

\[
x^2 + a_1 x
+ a_2 = 0
\]
where \begin{align*} 
 a_1 &= \left((\sigma + \theta + \mu) + (\rho + \mu + \dfrac{r u_2}{1 + \phi^2}) - \dfrac{\lambda}{(\mu + u_1)} - \alpha \lambda \right) \\ 
 a_2 &= \left\{ \left( \dfrac{\gamma \lambda}{(\mu + u_1)} - (\sigma + \theta + \mu) \right) \left(\alpha \lambda - (\rho + \mu + \dfrac{r u_2}{1 + \phi^2}) \right) - 2\beta \theta \right\} \\ 
 &=(\sigma + \theta + \mu) \left(\rho + \mu + \dfrac{r u_2}{1 + \phi^2} - \alpha \lambda \right) (1 - \mathcal{R}_0) 
 \end{align*}



\newtcolorbox{theoremBox}[1][]{
  colback=blue!5!white,
  colframe=blue!75!black,
  fonttitle=\bfseries,
  title=Conditions for Positivity of Coefficients,
  #1
}

\begin{theoremBox}
Let \( a_1 \) and \( a_2 \) be expressions defined in the above characteristic polynomial in \(x\). Then:

\begin{equation} \label{positivity of a_i's}
a_i > 0 \quad \text{if} \quad
\begin{cases}
(\sigma + \theta + \mu) + \left( \rho + \mu + \dfrac{r u_2}{1 + \phi^2} \right) 
> \dfrac{\lambda}{\mu + u_1} + \alpha \lambda, & \text{for } i = 1, \\[10pt]

\mathcal{R}_0 < 1 \quad \text{and} \quad 
\rho + \mu + \dfrac{r u_2}{1 + \phi^2} > \alpha \lambda, & \text{for } i = 2.
\end{cases}
\end{equation}
\end{theoremBox}

Routh-Hurwitz criteria \cite{hurwitz1964conditions, dejesus1987routh, patil2021routh, lancaster1985theory} for the two-degree equation \(x^2 + a_1 x + a_2 = 0\) to have negative roots are:
\[a_1 > 0 \quad \text{and} \quad a_2 > 0.\]
Thus, according to the Routh-Hurwitz criteria, the model (\ref{model eqs.}) is locally asymptotically stable at $E^0$ if (\ref{positivity of a_i's}) is satisfied.

\end{proof}

\begin{theorem}[Global Stability at Disease-Free Equilibrium (DFE)] Model (\ref{model eqs.}) is globally asymptotically stable at \(E^0 = (S^0, A^0, C^0, R^0)\) if $\mathcal{R}_0 < 1$, and unstable if $\mathcal{R}_0 > 1$.
\end{theorem}

\begin{proof}
Let us consider the Lyapunov function \cite{korobeinikov2006lyapunov, zarin2021fractional} L as
\[L(S, A, C, R) = S + A + \left(C - C^0 \ln C \right) + R\]

Then, differentiating L with respect to the time t, we get

\begin{align*}
\frac{dL}{dt} &= \frac{dS}{dt} + \frac{dA}{dt} + \left(1 - \dfrac{C^0}{C} \right) \frac{dC}{dt} + \frac{dR}{dt} \\
\frac{dL}{dt} &= \lambda - \mu N - \rho C - \frac{C^0}{C} \left(\alpha \lambda C - \rho C - \mu C - \dfrac{r u_2 C^0}{1 + \phi^2} + \theta A \right) \\
\frac{dL}{dt} &= \lambda - \mu (S + A + R) - \rho C \left( 1 - \dfrac{C^0}{C} \right) - \mu C \left(1 - \dfrac{C^0}{C} \right) - \alpha \lambda C^0 - \frac{\theta A C^0}{C} + \dfrac{r u_2 C^0}{1 + \phi^2} \\
\frac{dL}{dt} &= -\left( \mu N^* + \rho C (1 - \dfrac{C^0}{C}) + \mu C (1 - \dfrac{C^0}{C}) + \alpha \lambda C^0 + \frac{\theta A C^0}{C} - \dfrac{r u_2 C^0}{1 + \phi^2} - \lambda \right) \\
\frac{dL}{dt} &< 0
\end{align*}

Therefore, an application of the LaSalle’s invariance principle \cite{la1976stability, lasalle1976stability} yields that the disease-free equilibrium $E^0$ is globally asymptotically stable.
\end{proof}

\begin{theorem}[Global Stability at Disease-Free Equilibrium (DFE)]
\label{Theorem2}
Let us consider the system defined by the dynamics of susceptible \( S \), acute \( A \), and chronic \( C \) compartments. In order to establish the global stability of the disease-free equilibrium, we construct the Lyapunov function:
\begin{align}
L(S,A,C) = d_1 (S - S^0) + d_2 A + d_3 C,
\end{align}
which is linear, positive definite, and radially unbounded. 

Then, the disease-free equilibrium \[E^0 = (S^0, A^0, C^0, R^0)\]
is globally asymptotically stable if \( \mathcal{R}_0 < 1 \).
\end{theorem}

\begin{proof}
Consider the time derivative of the Lyapunov function:

\[
\frac{dL}{dt} = d_1 \frac{dS}{dt} + d_2 \frac{dA}{dt} + d_3 \frac{dC}{dt}
\]

Substituting the first three equations of model (\ref{model eqs.}), we get:

\begin{align*}
\frac{dL}{dt} &= d_1 \left( \lambda(1 - \alpha C) - \frac{2 \beta S C}{S + C} - \gamma A S - u_1 S - \mu S \right) \\
&\quad + d_2 \left( \frac{2 \beta S C}{S + C} - \gamma A S- (\sigma + \theta + \mu) A \right) \\
&\quad + d_3 \left( \theta A + \alpha \lambda C - \rho C - \frac{r u_2 C}{1 + \phi^2} - \mu C \right)
\end{align*}

Let all parameters \(d_1, d_2, d_3, d_4\) be positive and biologically feasible. We define the constants as following:
\[
d_1 = d_2 = u_1 + \mu, \quad
d_3 = \sigma + \theta + \mu, \quad
d_4 = \rho + \mu + \frac{r u_2}{1 + \phi^2} - \alpha \lambda
\]

\begin{align*}
\frac{dL}{dt} 
&= (u_1 + \mu) \lambda - (u_1 + \mu)^2 S - (u_1 + \mu) \alpha \lambda C - \gamma \lambda A + \gamma \lambda A - (u_1 + \mu)(\sigma + \theta + \mu) A \\
&\quad + \theta (\sigma + \theta + \mu) A 
+ (\sigma + \theta + \mu) \left(\alpha \lambda - \rho - \mu - \frac{r u_2}{1 + \phi^2} \right) C + 2 \beta \theta C - 2 \beta \theta
\end{align*}

\begin{align*}
\frac{dL}{dt} 
&= (u_1 + \mu) \lambda - (u_1 + \mu)^2 S - (u_1 + \mu) \alpha \lambda C +\left(\frac{\gamma \lambda}{(u_1 + \mu)(\sigma + \theta + \mu)} - 1\right) (u_1 + \mu)(\sigma + \theta + \mu) A - \gamma \lambda A + \theta (\sigma + \theta + \mu) A \\
&\quad + \left(\frac{2 \beta \theta}{\left( \rho + \mu + \frac{r u_2}{1 + \phi^2} - \alpha \lambda \right)(\sigma + \theta + \mu)} - 1\right) \left( \rho + \mu + \frac{r u_2}{1 + \phi^2} - \alpha \lambda \right)(\sigma + \theta + \mu) C - 2 \beta \theta C
\end{align*}

Since \(S^0 = \frac{\lambda}{u_1 + \mu} \Rightarrow \lambda = (u_1 + \mu) S^0\), then substituting the values $\lambda$ and $d_i's$ where $i = 1,2,3,4$, we get: 

\begin{align*}
\frac{dL}{dt} = d_1^2 S_0 - d_1^2 S - \alpha \lambda d_1 C + (\mathcal{R}_0^A - 1) d_1 d_3 A +(\mathcal{R}_0^C - 1) d_3 d_4 C + \theta d_3 A - 2 \beta \theta C 
\end{align*}

\begin{align}
\label{Lyapunov1}
\frac{dL}{dt} = - (S - S^0) d_1^2 - \alpha \lambda d_1 C + (\mathcal{R}_0^A - 1) d_1 d_3 A +(\mathcal{R}_0^C - 1) d_3 d_4 C + \theta d_3 A - 2 \beta \theta C
\end{align}

Here, \( \mathcal{R}_0^A \) and \( \mathcal{R}_0^C \) denote the reproduction numbers corresponding to the acute and chronic infection transmission, respectively. 

If \( \mathcal{R}_0 < 1\) $\Rightarrow$ \(0 < \mathcal{R}_0^A < 1 \) and \(0 < \mathcal{R}_0^C < 1 \), then \( \mathcal{R}_0^A - 1 < 0 \) and \( \mathcal{R}_0^C - 1 < 0 \).

It is evident from equation (\ref{Lyapunov1}) that \(\frac{dL}{dt} < 0\) and \(\frac{dL}{dt} = 0\) if and only if \(S = S^0, A = A^0, C = C^0, R = R^0 \).

Hence, the largest invariant set where \( \frac{dL}{dt} = 0 \) corresponds to the disease-free equilibrium $E^0$. Therefore, by LaSalle’s Invariance Principle \cite{la1976stability, lasalle1976stability}, the disease-free equilibrium 
\(E^0 = (S^0, A^0, C^0, R^0)\) is globally asymptotically stable.
\end{proof}

\begin{theorem}[Local Stability at Endemic Equilibrium (EE)]
\label{Theorem3}
Model (\ref{model eqs.}) is locally asymptotically stable at \(E^* = (S^*, A^*, C^*, R^*)\) if $\mathcal{R}_0 > 1$ and, otherwise, is unstable.
\end{theorem}

\begin{proof}
The Jacobian matrix of the system (\ref{model eqs.}) at the endemic equilibrium point \(E^* = (S^*, A^*, C^*, R^*)\) is as follows:
\[
J = \begin{bmatrix}
-\gamma A^* - \frac{2 \beta (C^*)^2}{(S^* + C^*)^2} - u_1 - u & -\gamma S^* & -\alpha \lambda - \frac{2 \beta (S^*)^2}{(S^* + C^*)^2} \\
\gamma A^* + \frac{2 \beta (C^*)^2}{(S^* + C^*)^2} & \gamma S^* - (\sigma + \mu +\theta)) & \frac{2 \beta (S^*)^2}{(S^* + C^*)^2} \\
0 & \theta & \alpha \lambda - \rho - \mu - \frac{r u_2}{1 + \phi^2}
\end{bmatrix}
\]

Let us define the following substitutions:
\begin{align*}
l_1 &= \gamma A^* + \frac{2 \beta (C^*)^2}{(S^* + C^*)^2}, \quad l_2 = \gamma S^*, \quad l_3 = \frac{2 \beta (S^*)^2}{(S^* + C^*)^2} \\
d_1 &= \rho + \mu + \frac{r u_2}{1 + \phi^2}, \quad d_3 = \sigma + \theta + \mu 
\end{align*}

The above matrix can be rewritten as:
\[
J = \begin{bmatrix}
-(l_1 + u_1 + \mu) & - l_2 & -(\alpha \lambda + l_3) \\
l_1 & l_2 - d_3 & l_3 \\
0 & \theta & \alpha \lambda - d_1
\end{bmatrix}
\]

Further, let:
\begin{align*}
L_1 &= (l_1 + u_1 + \mu) \\
L_2 &= l_2 - d_3 \\
L_3 &= \alpha \lambda + l_3 \\
L_4 &=\alpha \lambda - d_1
\end{align*}

Then:
\[
J = \begin{bmatrix}
-L_1 & - l_2 & -L_3 \\
l_1 & L_2 & l_3 \\
0 & \theta & L_4
\end{bmatrix}
\]

Applying the row operation $R_2 \rightarrow R_2 + \frac{l_1}{L_1}R_1$:

\[
J = \begin{bmatrix}
-L_1 & - l_2 & -L_3 \\
0 & L_1 L_2 - l_1 l_2 & l_3 L_1 - l_1 L_3 \\
0 & \theta & L_4
\end{bmatrix}
\]

Finally, defining:
\begin{align*}
L_5 &= l_1 l_2 - L_1 L_2 \\
L_6 &= l_1 L_3 - l_3 L_1
\end{align*}

The reduced matrix becomes:
\[
J = \begin{bmatrix}
-L_1 & - l_2 & -L_3 \\
0 & -L_5 & -L_6 \\
0 & \theta & L_4
\end{bmatrix}
\]

Applying the row operation $R_3 \rightarrow R_3 + \frac{\theta}{L_5}R_2$:

\[
J = \begin{bmatrix}
-L_1 & - l_2 & -L_3 \\
0 & -L_5 & -L_6 \\
0 & 0 & -L_7
\end{bmatrix}
\]
where \[L_7 = \theta L_6 - L_4 L_5 \]

Using the upper triangular form of the matrix, we can directly read off the eigenvalues as
\[
\lambda_1 = -L_1, \quad \lambda_2 = -L_5, \quad \text{and} \quad \lambda_3 = -L_7.
\]
For the endemic equilibrium to be stable, all of these eigenvalues must have negative real parts. This leads to a simple condition: each of the terms $L_1$, $L_5$, and $L_7$ must be greater than zero. In other words, the stability of the endemic point requires that
\[
L_i > 0 \quad \text{for} \quad i = 1,\,5,\,7.
\]


Here 
\begin{align*}
L_1 &= (l_1 + u_1 + \mu) > 0 \\
L_5 &= l_1 l_2 - L_1 L_2 > 0 \quad if \quad \mathcal{R}_0^A < 1 \quad \text{and} \quad \mathcal{R}_0^C > 1\\
L_7 &= \theta L_6 - L_4 L_5 > 0 \quad if \quad \mathcal{R}_0^A < 1 \quad \text{and} \quad d_3 > l_2.\\
\end{align*}

\end{proof}

\newtcbtheorem{mytheo}{}%
{colback=red!5!white,
 colframe=yellow!75!black,
 fonttitle=\bfseries,
 coltitle=black,
 enhanced,
 sharp corners,
 breakable,
 before skip=10pt, after skip=10pt}%
{th}

\begin{mytheo}{Positivity for \(L_5\)}{} 
Let
\[
l_1 = \gamma A^* + \frac{2 \beta C^{*2}}{(C^* + S^*)^2}, \quad l_2 = \gamma S^*
\]
and
\[
L_1 = l_1 + u_1 + \mu, \quad L_2 = l_2 - (\sigma + \theta + \mu).
\]
Then
\begin{align*}
L_5 &= l_1 l_2 - L_1 L_2 \\
    &= \left( \gamma A^* + \frac{2 \beta C^{*2}}{(C^* + S^*)^2} \right) \gamma S^* - \left( \lambda_1 + u_1 + \mu \right) \left( \lambda_2 - (\sigma + \theta + \mu) \right) \\
    &= \lambda_1 \lambda_2 - \left[ \lambda_1 \lambda_2 - \lambda_1 (\sigma + \theta + \mu) + (u_1 + \mu)(\sigma + \theta + \mu) - (u_1 + \mu) \lambda_2 \right] \\
    &= - (u_1 + \mu) \lambda_2 + \lambda_1 (\sigma + \theta + \mu) + (u_1 + \mu)(\sigma + \theta + \mu) \\
    &= - (u_1 + \mu) \gamma S^* + \left( \gamma A^* + \frac{2 \beta C^{*2}}{(C^* + S^*)^2} \right)(\sigma + \theta + \mu) + (u_1 + \mu)(\sigma + \theta + \mu)
\end{align*}

Now, use the equilibrium relation:
\[
\gamma S^* = \frac{\lambda}{u_1 + \mu} - \frac{1}{\theta} \left( (\sigma + \theta + \mu) \left( \rho + \mu + \frac{r u_2}{1 + \phi^2} \right) + \alpha \gamma \theta \right) C^*
\]

Substituting into \( L_5 \):
\begin{align*}
L_5 &= - (u_1 + \mu) \cdot \left[ \frac{\lambda}{u_1 + \mu} - \frac{1}{\theta} \left( (\sigma + \theta + \mu) \left( \rho + \mu + \frac{r u_2}{1 + \phi^2} \right) + \alpha \gamma \theta \right) C^* \right] \\
& \quad + \lambda_1 (\sigma + \theta + \mu) + (u_1 + \mu)(\sigma + \theta + \mu) \\
&= - \lambda + \frac{u_1 + \mu}{\theta} \left( (\sigma + \theta + \mu) \left( \rho + \mu + \frac{r u_2}{1 + \phi^2} \right) + \alpha \gamma \theta \right) C^* \\
& \quad + \lambda_1 (\sigma + \theta + \mu) + (u_1 + \mu)(\sigma + \theta + \mu)
\end{align*}

Recall:
\[
R_0^A = \frac{\lambda}{(u_1 + \mu)(\sigma + \theta + \mu)} = \frac{\lambda}{(u_1 + \mu) d_3}
\]

So, we finally get:
\begin{align*}
L_5 &= (u_1 + \mu) d_3 (1 - R_0^A) + \frac{u_1 + \mu}{\theta} \left( (\theta + d_3) \alpha \lambda + 2 \beta \theta \right) - \frac{u_1 + \mu}{\theta} (R_0^C - 1) d_1 d_3
\end{align*}

\[
\therefore \quad L_5 = \lambda_1 \lambda_2 - L_1 L_2 > 0 \quad \text{if} \quad R_0^A < 1 \quad \text{and} \quad R_0^C > 1.
\]
\end{mytheo}

\begin{tcolorbox}[
  colframe=yellow!75!black,
  colback=red!5,
  boxrule=1pt,
  arc=4pt,
  title={2: Calculation for the $L_7$}
]

\[
\begin{aligned}
L_7 &= \theta L_6 - L_4 L_5 \\
    &= \theta(\lambda_1 L_3 - l_3 L_1) - (\alpha \lambda - d_1)(l_1 l_2 - L_1 L_2) \\
    &= \theta(l_1 (\alpha \lambda + l_3) - l_3(l_1 + u_1 + \mu)) - (\alpha \lambda - d_1)\left(l_1 l_2 - (l_1 + u_1 + \mu)(l_2 - d_3)\right) \\
    &= \theta(\alpha \lambda l_1 - (u_1 + \mu) l_3) - (\alpha \lambda - d_1)\left(l_1 d_3 + (u_1 + \mu)(d_3 - l_2)\right) \\
    &= \theta\left(\alpha \lambda l_1 - \gamma \lambda + \gamma \lambda - (u_1+\mu)(\sigma + \theta + \mu)\right) + (d_1 - \alpha \lambda)\left(l_1 d_3 + (u_1 + \mu)(d_3 - l_2)\right) \\
    &= \theta \lambda (\alpha l_1 - \gamma) - \theta (1 - R_0^A)(u_1+\mu)(\sigma+\theta+\mu) + (d_1 - \alpha \lambda)\left(l_1 d_3 + (u_1 + \mu)(d_3 - l_2)\right) \\
\end{aligned}
\]

\(
\therefore L_7 = \theta L_6 - L_4 L_5 > 0 \quad \text{if} \quad R_0^A < 1 \quad \text{and} \quad d_3 > l_2.
\)

\end{tcolorbox}

\begin{theorem}[Global Stability at Endemic Equilibrium (EE)]
\label{Theorem4}
Model (\ref{model eqs.}) is globally asymptotically stable at \(E^* = (S^*, A^*, C^*, R^*)\) if $\mathcal{R}_0 > 1$, and unstable if $\mathcal{R}_0 < 1$.
\end{theorem}

\begin{proof}
Consider the system describing the dynamics of the susceptible \( S \), acute \( A \), chronic \( C \), and recovered \( R \) compartments. To establish the global stability of the endemic equilibrium point, we construct the Lyapunov function:
\[
L(S,A,C,R) = (S - S^* - S^* \ln \tfrac{S}{S^*}) + (A - A^* - A^* \ln \tfrac{A}{A^*}) + (C - C^* - C^* \ln \tfrac{C}{C^*}) + (R - R^* - R^* \ln \tfrac{R}{R^*}).
\]

This function is non-negative for all \(S, A, C, R > 0\) and equals zero if and only if \(S = S^*, A = A^*, C = C^*, R = R^*\). Hence, \(L\) is positive definite.

Taking the time derivative of the Lyapunov function \(L\) along the solution trajectories of the system:
\[
\frac{dL}{dt} = \left( 1 - \frac{S^*}{S} \right) \frac{dS}{dt}
+ \left( 1 - \frac{A^*}{A} \right) \frac{dA}{dt}
+ \left( 1 - \frac{C^*}{C} \right) \frac{dC}{dt}
+ \left( 1 - \frac{R^*}{R} \right) \frac{dR}{dt}.
\]

Let us denote the time derivative of the Lyapunov function as:
\[
\frac{dL}{dt} = I_S + I_A + I_C + I_R,
\]
where each term corresponds to one compartmental contribution.

\paragraph{Susceptible Component:}
\[
I_S = \left( 1 - \frac{S^*}{S} \right) \left[ \lambda (1 - \alpha C) - \beta \frac{2SC}{S + C} - \gamma A S - u_1 S - \mu S \right]
\]

\paragraph{Acute Component:}
\[
I_A = \left( 1 - \frac{A^*}{A} \right) \left[ \gamma A S + \beta \frac{2SC}{S + C} - (\sigma + \theta + \mu) A \right]
\]

\paragraph{Chronic Component:}
\[
I_C = \left( 1 - \frac{C^*}{C} \right) \left[ \theta A + \lambda \alpha C - \left( \rho + \mu + \frac{r u_2}{1 + \phi^2} \right) C \right]
\]

\paragraph{Recovered Component:}
\[
I_R = \left( 1 - \frac{R^*}{R} \right) \left[u_1 S + \sigma A + \frac{r u_2 C}{1 + \phi^2}  - \mu R \right]
\]

These terms reflect the difference between each compartment's current state and its endemic equilibrium value, multiplied by the net flow into or out of that compartment. Each contributes to the total sign of $\frac{dL}{dt}$ in a Lyapunov stability analysis.

At the endemic equilibrium \(E^*\), we have:
\begin{align*}
\frac{dS}{dt} &= \lambda(1 - \alpha C^*) - \beta \frac{2S^*C^*}{S^* + C^*} - \gamma A^* S^* - u_1 S^* - \mu S^* = 0, \\
\frac{dA}{dt} &= \gamma A^* S^* + \beta \frac{2S^*C^*}{S^* + C^*} - (\sigma + \mu + \theta) A^* = 0, \\
\frac{dC}{dt} &= \theta A^* + \lambda \alpha C^* - \left( \mu + \rho + \frac{r u_2}{1 + \phi^2} \right) C^* = 0, \\
\frac{dR}{dt} &= \sigma A^* + \frac{r u_2 C^*}{1 + \phi^2} + u_1 S^* - \mu R^* = 0.
\end{align*}

Consider the derivative of the susceptible compartment:

\[
\frac{dS}{dt} = \lambda(1 - \alpha C) - \beta \frac{2SC}{S + C} - \gamma A S - u_1 S - \mu S.
\]

At the endemic equilibrium \( E^* = (S^*, A^*, C^*, R^*) \), we have:

\[
\left. \frac{dS}{dt} \right|_{E^*} = \lambda(1 - \alpha C^*) - \beta \frac{2S^* C^*}{S^* + C^*} - \gamma A^* S^* - u_1 S^* - \mu S^* = 0.
\]

Now consider the Lyapunov term:

\[
\left(1 - \frac{S^*}{S}\right)
\left[ \lambda(1 - \alpha C) - \beta \frac{2SC}{S + C} - \gamma A S - u_1 S - \mu S \right].
\]

To simplify the above expression, we subtract the equilibrium derivative of S and rearrange the terms to get:

\[
= \left(1 - \frac{S^*}{S}\right) 
\left[ \left( \frac{dS}{dt} - \left. \frac{dS}{dt} \right|_{E^*} \right) \right].
\]

This form is useful in proving that the Lyapunov derivative \( \frac{dL}{dt} \leq 0 \), since the difference reflects deviation from the equilibrium, and the prefactor \( \left(1 - \frac{S^*}{S} \right) \) ensures negativity under convexity or Jensen-type arguments.

We apply the same argument to all compartments. For any compartment \( x \in \{S, A, C, R\} \), let
\[
f_x(x) = \text{RHS of } \frac{dx}{dt}, \text{so that } f_x(x^*) = 0 \text{ at equilibrium}.
\]

Then, the total derivative of the Lyapunov function can be compactly written as:
\[
\frac{dL}{dt} = \sum_{x \in \{S, A, C, R\}} \left( 1 - \frac{x^*}{x} \right) \left( f_x(x) - f_x(x^*) \right).
\]

Here the expression \(f_x(x) - f_x(x^*)\) captures how the nonlinear terms deviate from equilibrium, and the scalar factor \(1 - \frac{x^*}{x}\) is positive when \(x > x^*\) and negative \(x < x^*\). So, the overall product \(\left( 1 - \frac{x^*}{x} \right) \left( f_x(x) - f_x(x^*) \right)\) becomes non-positive due to convexity and Jensen-type inequalities.

This allows us to rewrite the \(\frac{dL}{dt}\):
\[
\frac{dL}{dt} = \left( 1 - \frac{S^*}{S} \right)(f_S(S) - f_S(S^*))
+ \left( 1 - \frac{A^*}{A} \right)(f_A(A) - f_A(A^*))
+ \left( 1 - \frac{C^*}{C} \right)(f_C(C) - f_C(C^*)) + \left( 1 - \frac{R^*}{R} \right)(f_R(R) - f_R(R^*)) \leq 0.
\]
with equality if and only if \(x = x^*\) for all x.

Thus, by \textbf{LaSalle's Invariance Principle} \cite{la1976stability, lasalle1976stability}, the endemic equilibrium \(E^*\) is \textbf{globally asymptotically stable}.
\end{proof}

\section{Optimal Control Problem}
Control of infectious diseases is typically effective when crucial parameters affecting the basic reproduction number \(\mathcal{R}_0\) can be identified, and action taken to push it below the threshold of one \cite{molina2022optimal, pontryagin2018mathematical}. There are two control-related parameters \(u_1\), which is the effort for vaccination or behavioral intervention, and \(u_2\), the effort for treatment of chronic infections included in our model. Sensitivity analysis shows that both \( u_1 \) and \( u_2 \) are inversely related to \( \mathcal{R}_0 \), which illustrates that changes in these two controls will contribute to the significant inhibition of infection transmission. This indicates that a combination of vaccination campaigns and treatment interventions is able to steer the disease dynamics to lower values than the epidemic threshold. The threshold condition \( \mathcal{R}_0 < 1 \) has an important policy interpretation: public health policies must strive to reduce \( \mathcal{R}_0 \) through measures targeted toward different subgroups in the population until this inequality is reached. If attained, the model suggests that the disease will not survive in the population and will instead be eliminated. Hence, the timely and effective implementation of control measures can shift the system from an endemic to a disease-free regime, aligning theoretical thresholds with practical decision-making.

The full optimal control problem consists of solving the state system, the adjoint system, and the control characterization.

To minimize the infected population and the cost of intervention, we define the objective functional:
\begin{align}
J(Z, U) = \int_0^{T_f} \left( k_1 A + k_2 C + \frac{1}{2} w_1 u_1^2 + \frac{1}{2} w_2 u_2^2 \right) dt.
\end{align}

Applying Pontryagin's Maximum Principle \cite{pontryagin2018mathematical}, we define the Hamiltonian:
\begin{align}
\label{Hamiltonian}
H = k_1 A + k_2 C + \frac{1}{2}w_1 u_1^2 + \frac{1}{2}w_2 u_2^2 + \sum_{j=1}^4 \lambda_j g_j(Z,U).
\end{align}

The adjoint equations are:
\begin{align}
\frac{d\lambda_j}{dt} = -\frac{\partial H}{\partial Z_j}, \quad j = 1,2,3,4.
\end{align}

Explicitly, these are:
\begin{align}
\frac{d\lambda_1}{dt} &= \left( \lambda_1 + \lambda_4 \right) (u_1 + \mu) - \lambda_1 \left( \frac{2\beta C(S+C) - 2\beta SC}{(S+C)^2} + \gamma A \right) + \lambda_2 \left( \frac{2\beta C(S+C) - 2\beta SC}{(S+C)^2} + \gamma A \right), \\
\frac{d\lambda_2}{dt} &= -k_1 + \lambda_2 (\sigma + \mu + \theta) - \lambda_3 \theta - \lambda_4 \sigma, \\
\frac{d\lambda_3}{dt} &= -k_2 + \lambda_1 \lambda \alpha + \lambda_3 \left( \mu + \rho + \frac{r u_2}{1+\phi^2} \right) - \lambda_4 \frac{r u_2}{1+\phi^2}, \\
\frac{d\lambda_4}{dt} &= \mu \lambda_4.
\end{align}

The optimal controls are characterized by solving:
\begin{align}
\frac{\partial H}{\partial u_1} &= 0, \quad \frac{\partial H}{\partial u_2} = 0.
\end{align}

Thus, the expressions for optimal controls are:
\begin{align}
u_1^*(t) &= \min\left(\max\left(0, \frac{(\lambda_4 - \lambda_1)S}{w_1}\right), 1\right), \\
u_2^*(t) &= \min\left(\max\left(0, \frac{\phi (\lambda_3 - \lambda_4)C}{w_2(1+\phi^2)}\right), 1\right).
\end{align}

\section{Numerical Simulations and Optimal Control}
Numerical simulations are crucial for understanding the impact of modifications of parameters on the overall dynamics of HBV infection. In this section, to bolster the theoretical findings in the preceding section, we perform numerical simulations on the proposed epidemiological model \eqref{model eqs.} based on the outcomes of the theoretical research.  To perform these numerical simulations, MATLAB\textsuperscript{\textregistered} software was used. Sensitivity analysis was carried out using LHS along with PRCC \cite{gomero2012latin}. We used the ode15s solver function to find the solution of our model and demonstrate the local and global stability of our proposed model.
Using backward and forward finite difference schemes, we have graphically demonstrated the potential of optimal control in reducing infected individuals, reducing the susceptibility to infection, and increasing the recovered population.

\subsection{Sensitivity Analysis}
Sensitivity analysis is a vital component of mathematical modeling \cite{gomero2012latin}. Most models comprise numerous parameters, but only a subset significantly influences the qualitative behavior of the system. Identifying these influential parameters is crucial for model calibration, control design, and effective intervention strategies. Several techniques exist for conducting a sensitivity analysis. Among them, the most commonly used are normalized forward sensitivity indices, Partial Rank Correlation Coefficients (PRCC), Sobol indices, etc. In this study, we employed Latin Hypercube Sampling (LHS) along with PRCC \cite{qian2020sensitivity} to calculate the sensitivity of parameters involved in the reproduction number \( \mathcal{R}_0 \). This technique allowed us to quantify how changes in parameter values affect \( \mathcal{R}_0 \) and the overall dynamics of the model, and it reveals that the infection transmission rates \( \beta, \gamma\) and \(\alpha\) are positively correlated with \( \mathcal{R}_0 \), indicating that any increment in these parameters tends to raise the reproduction number, whereas parameters such as \( \sigma, \theta, \mu, u_2 \) and treatment rate \( r \) showed a negative correlation with \( \mathcal{R}_0 \), suggesting that increased efforts in vaccination and treatment can reduce the transmission potential of the disease. The sensitivity analysis results are summarized in Figure \ref{fig:LHS/PRCC}, where subfigures \ref{fig:bargraph} and \ref{fig:line-chart} show the PRCC bar plot and the corresponding line plot, respectively. Bar plots \ref{fig:bargraph} and line charts \ref{fig:line-chart} of PRCC values for each parameter, along with the magnitude and direction of their influence on \(\mathcal{R}_0\), have been used to illustrate the results. For the statistical significance of each parameter, a comprehensive table \ref{tab:PRCC_table2} containing the calculated PRCC values and corresponding p-values is also provided. This analysis aims to identify the most influential parameters and provide direction for focused public health interventions and policy design.

\vspace{1cm}
\begin{table}[ht]
    \centering
    \begin{tabular}{>{\centering\arraybackslash}p{4cm} >{\centering\arraybackslash}p{4cm} >{\centering\arraybackslash}p{4cm}}
        \hline
        \textbf{Parameter} & \textbf{PRCC} & \textbf{p-value} \\
        \hline
        $\beta$ & 0.4627 & 0.0000 \\ \hline
        $\sigma$ & -0.1245 & 0.0053 \\ \hline
        $\gamma$ & 0.2254 & 0.0000 \\ \hline
        $\theta$ & -0.1918 & 0.0000 \\ \hline
        $\phi$ & -0.0077 & 0.8629 \\ \hline
        $\mu$ & -0.4316 & 0.0000 \\ \hline
        $\alpha$ & 0.1327 & 0.0030 \\ \hline
        $\lambda$ & 0.3770 & 0.0000 \\ \hline
        r & -0.4878 & 0.0000 \\ \hline
        $u_1$ & -0.0210 & 0.6394 \\ \hline
        $u_2$ & -0.4097 & 0.0000 \\ \hline
        $\rho$ & -0.0685 & 0.1263 \\ \hline
        \hline
    \end{tabular}
    \caption{PRCC values and corresponding p-values for various parameters.}
    \label{tab:PRCC_table2}
\end{table}

\begin{figure}[htbp]
  \centering
  \begin{subfigure}{0.65\textwidth}
    \includegraphics[width=\linewidth]{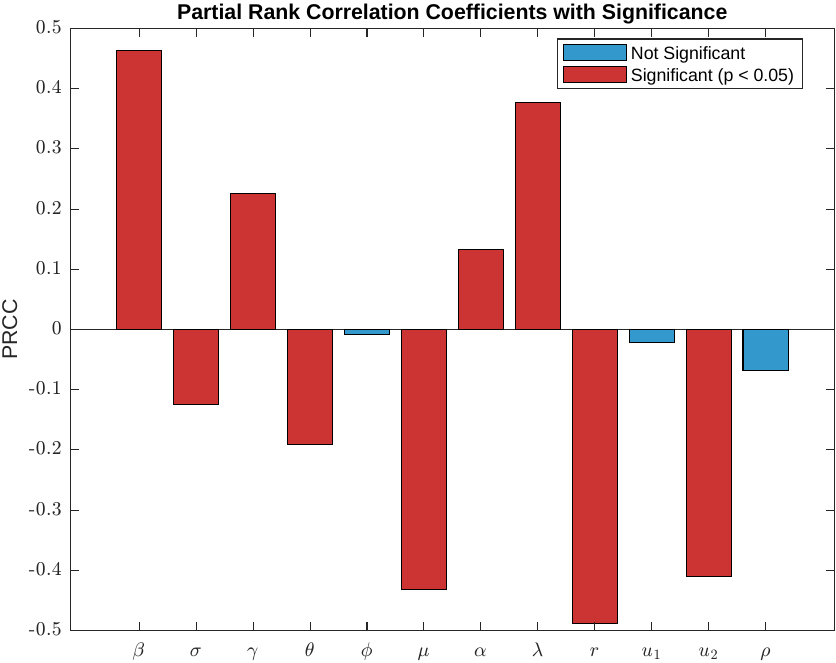}
    \caption{Bar graph of parameters involved in \(\mathcal{R}_0\) and their significance}
    \label{fig:bargraph}
  \end{subfigure}

  \vspace{0.3cm}
  \begin{subfigure}{0.85\textwidth}
    \centering
    \includegraphics[width=\linewidth]{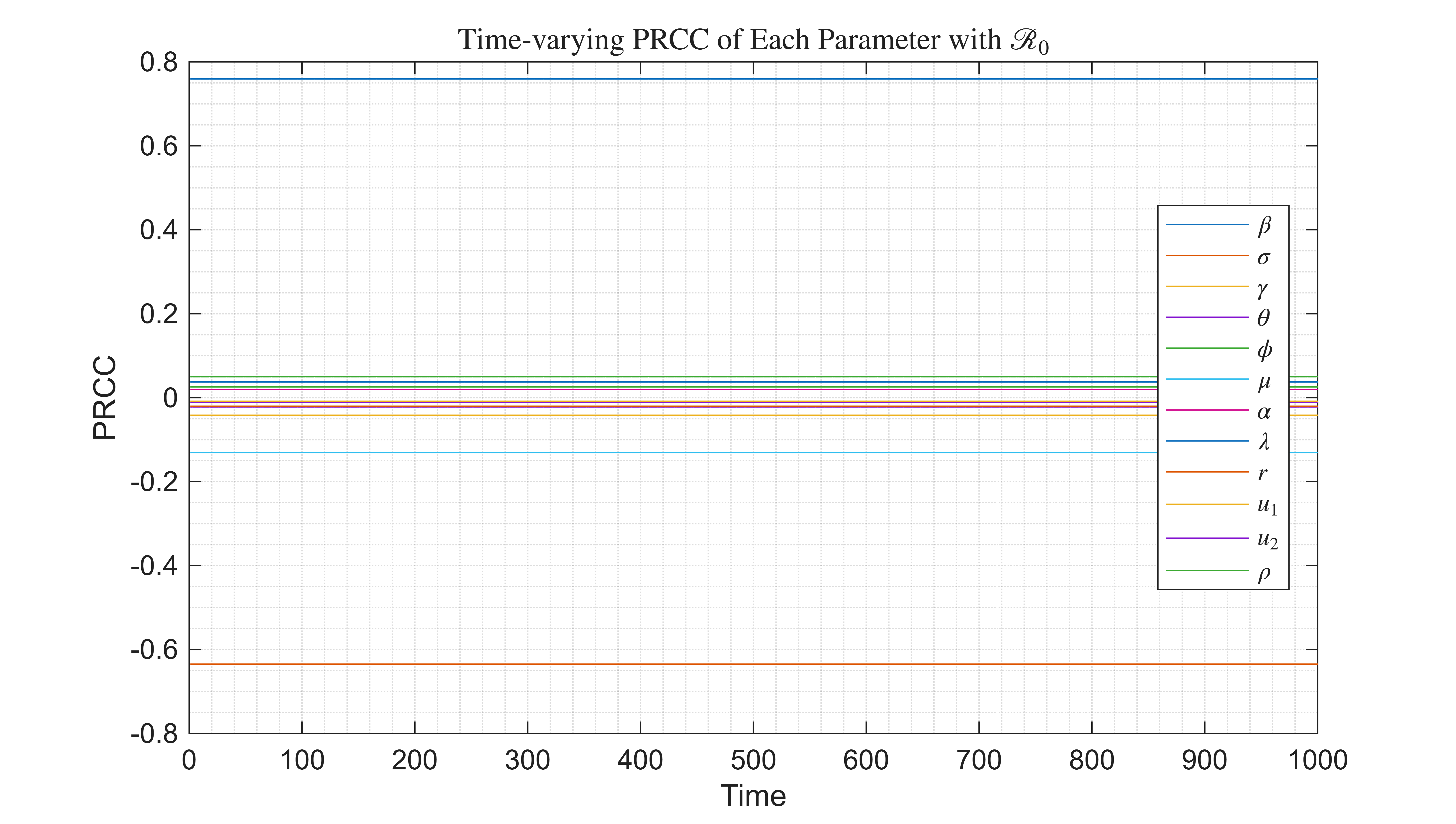}
    \caption{Line chart of parameters involved in \(\mathcal{R}_0\) and their significance}
    \label{fig:line-chart}
  \end{subfigure}
  \caption{Partial Rank Correlation Coefficients (PRCCs) illustrating the sensitivity of the basic reproduction number \( \mathcal{R}_0 \) concerning the model parameters. Positive PRCC values indicate parameters that increase \( \mathcal{R}_0 \), while negative values reflect those that decrease it. Parameters such as the transmission rates \( \beta \& \gamma \) and vertical transmission rate \( \lambda \) show strong positive influence, whereas control-related parameters such as \( u_2 \) (treatment) and \( r \) (recovery rate) are negatively correlated with \( \mathcal{R}_0 \).}
  \label{fig:LHS/PRCC}
\end{figure}

\newpage
Now, we have plotted a 3d surface graph and the corresponding heatmap of pairs of significant parameters against \(\mathcal{R}_0\) on the basis of the findings of the above sensitivity analysis, see Figs. \cref{Sens_1,Sens_2,Sens_3,Sens_4,Sens_5,Sens_6,Sens_7,Sens_8,Sens_9,Sens_10,Sens_11,Sens_12}. From each pair of 3d surface graphs and heatmaps, we can easily observe how the value of \(\mathcal{R}_0\) changes when we vary the values of parameters. Figures \ref{fig:com1} and \ref{fig:com2} illustrate the effect of the variation of different parameters on the recovered population. Parameters like \(\beta, \gamma, \mu, \theta\) and \(u_2\) show greater variation. 

\begin{figure}[htbp]
  \centering
  \begin{minipage}{0.40\textwidth}
    \includegraphics[width=\linewidth]{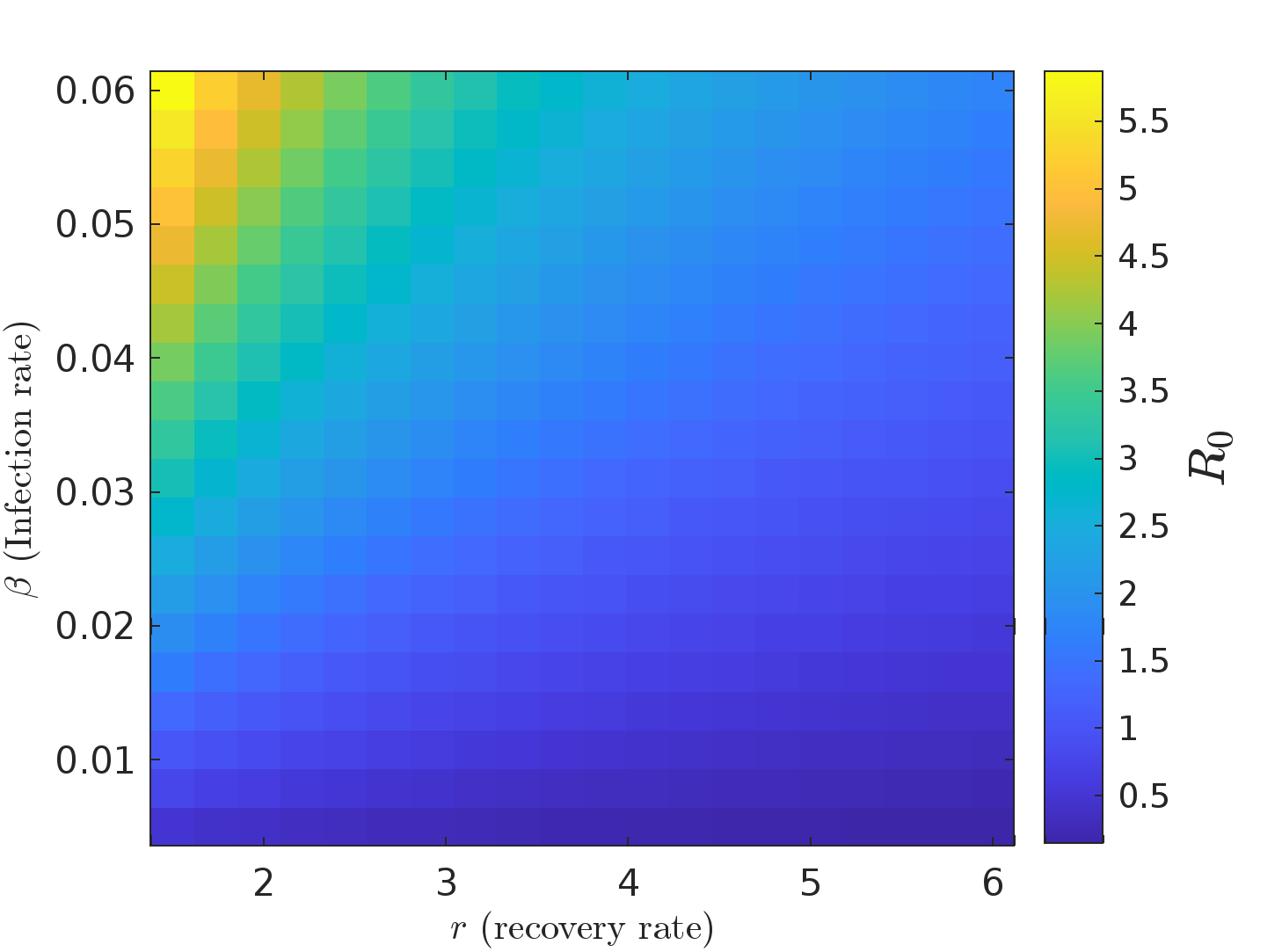}
  \end{minipage}
  \hfill
  \begin{minipage}{0.50\textwidth}
    \includegraphics[width=\linewidth]{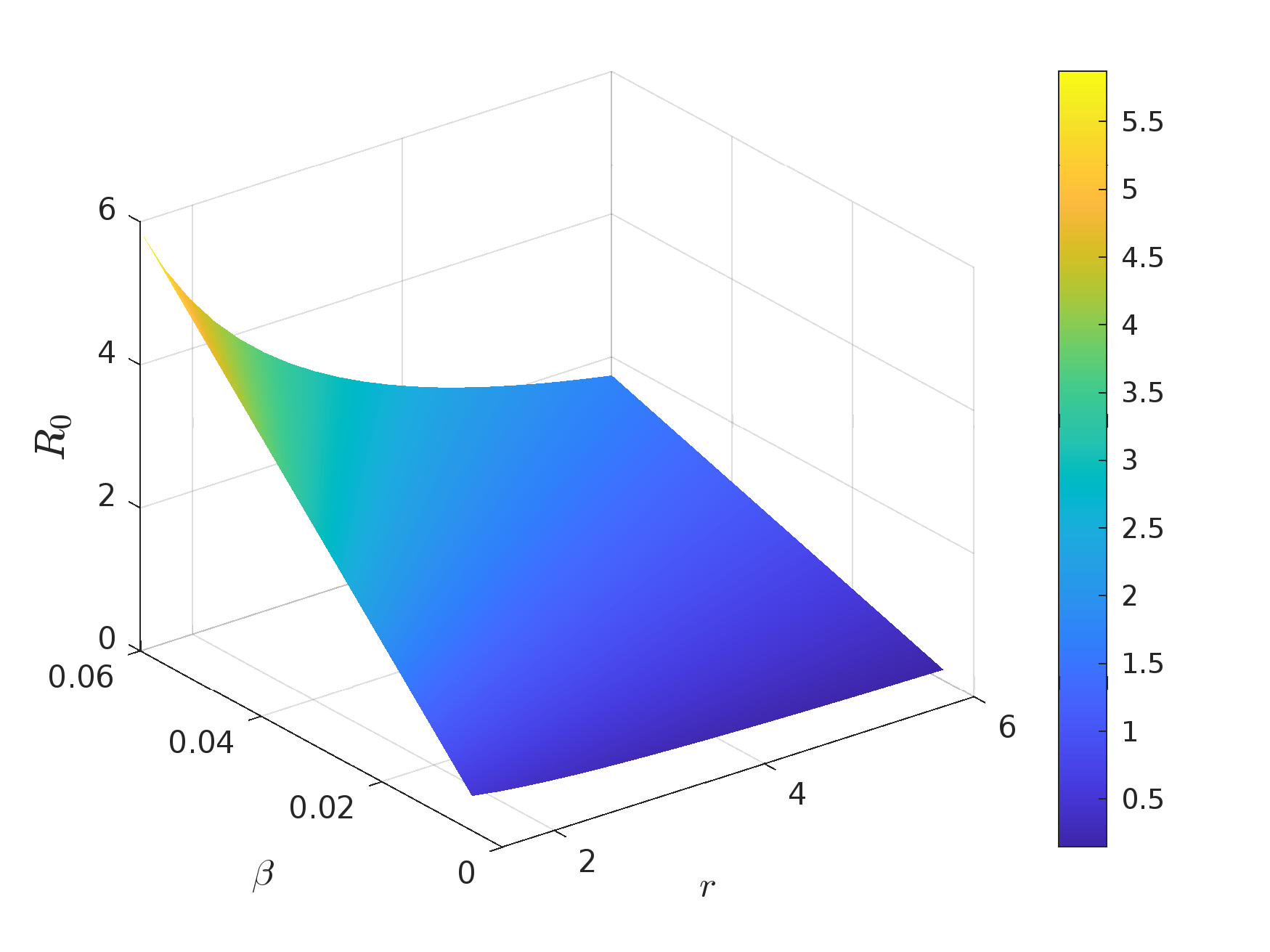}
  \end{minipage}
  \caption{Heatmap and Surface plot of  \(\mathcal{R}_0\) when varying $r$ and $\beta$.}
  \label{Sens_1}
\end{figure}

\begin{figure}[htbp]
  \centering
  \begin{minipage}{0.40\textwidth}
    \includegraphics[width=\linewidth]{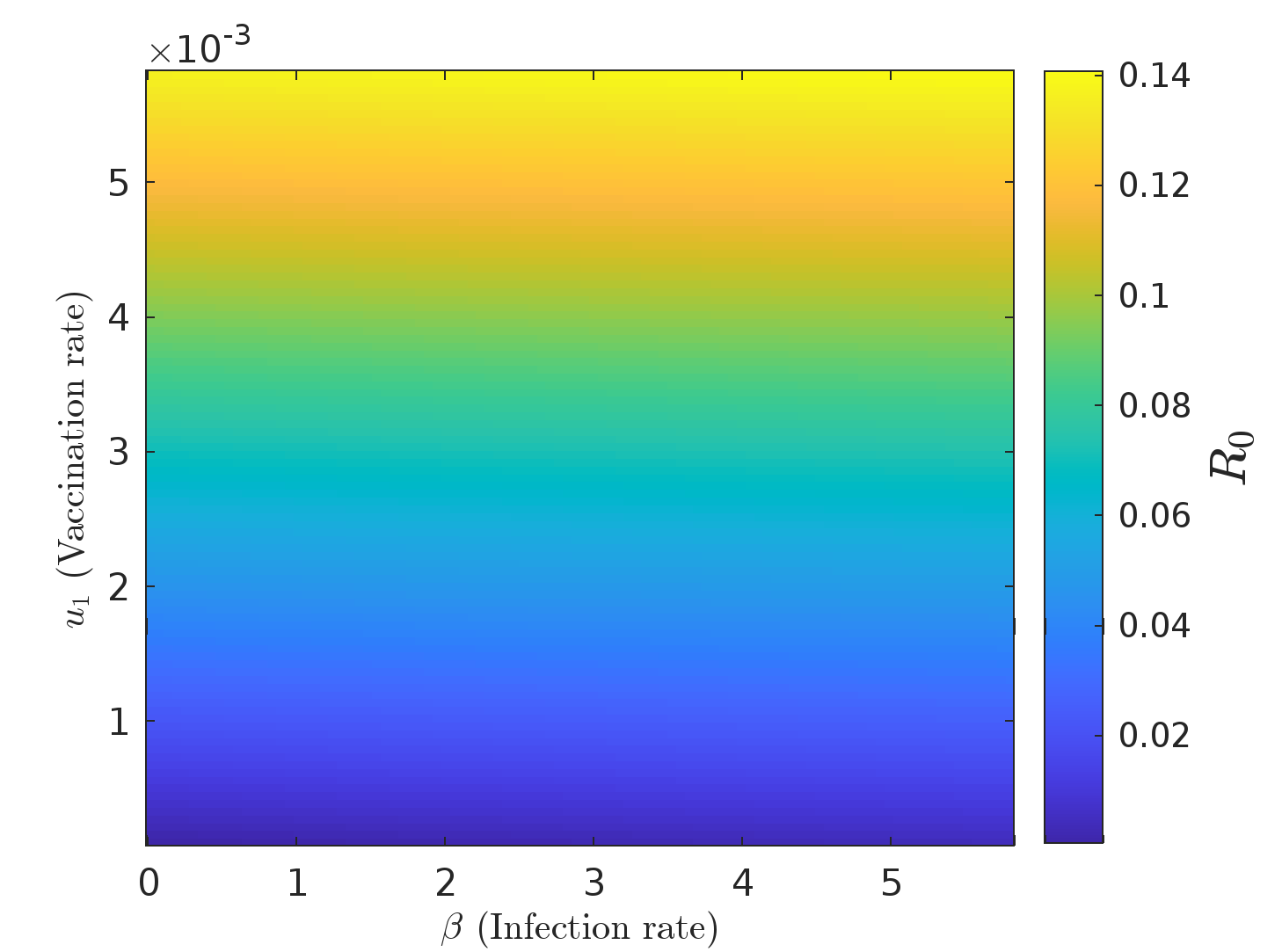}
  \end{minipage}
  \hfill
  \begin{minipage}{0.50\textwidth}
    \includegraphics[width=\linewidth]{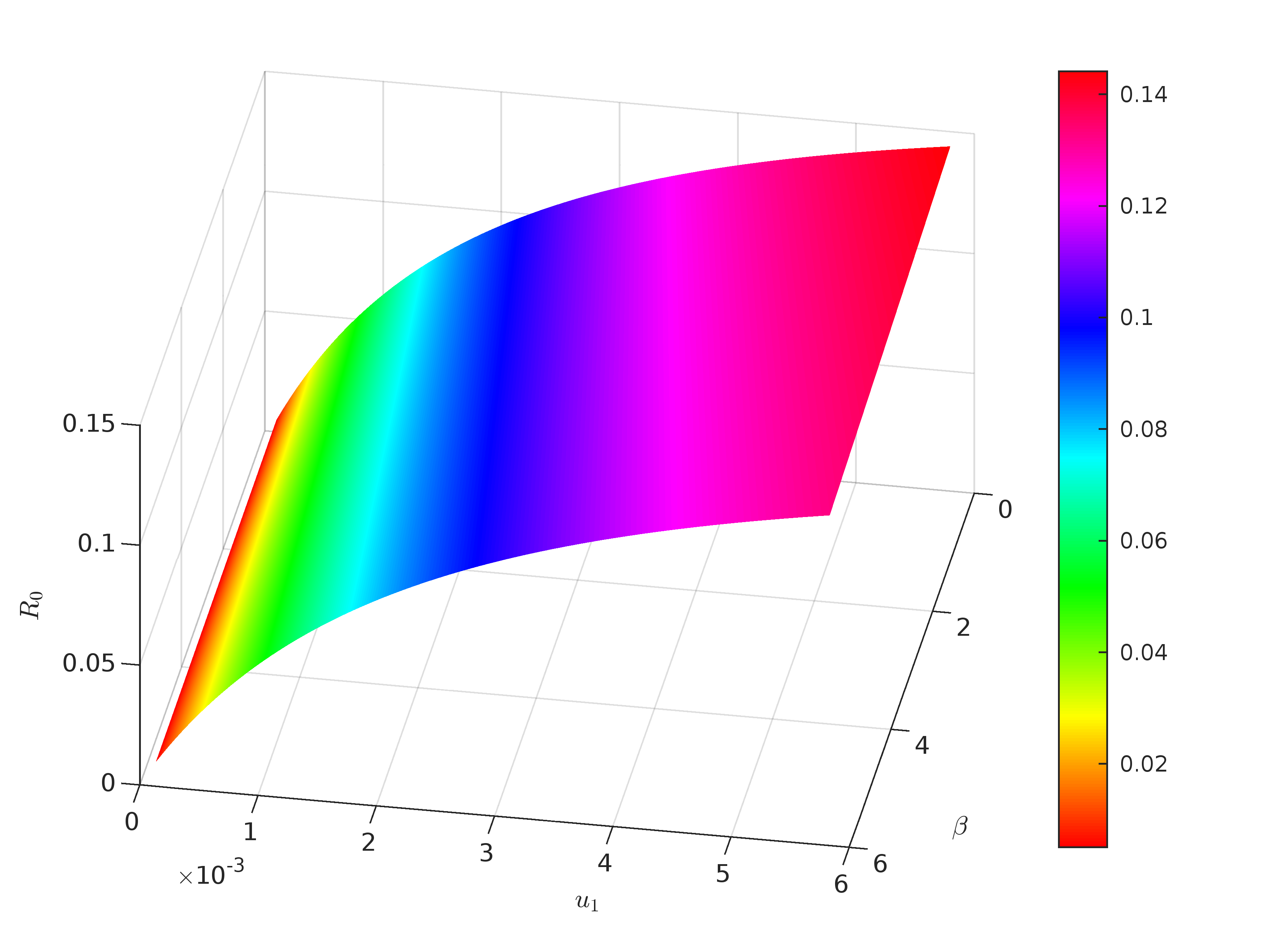}
  \end{minipage}
  \caption{Heatmap and Surface plot of $R_0$ when varying $\beta$ and $u_1$.}
  \label{Sens_2}
\end{figure}

\begin{figure}[htbp]
  \centering
  \begin{minipage}{0.40\textwidth}
    \includegraphics[width=\linewidth]{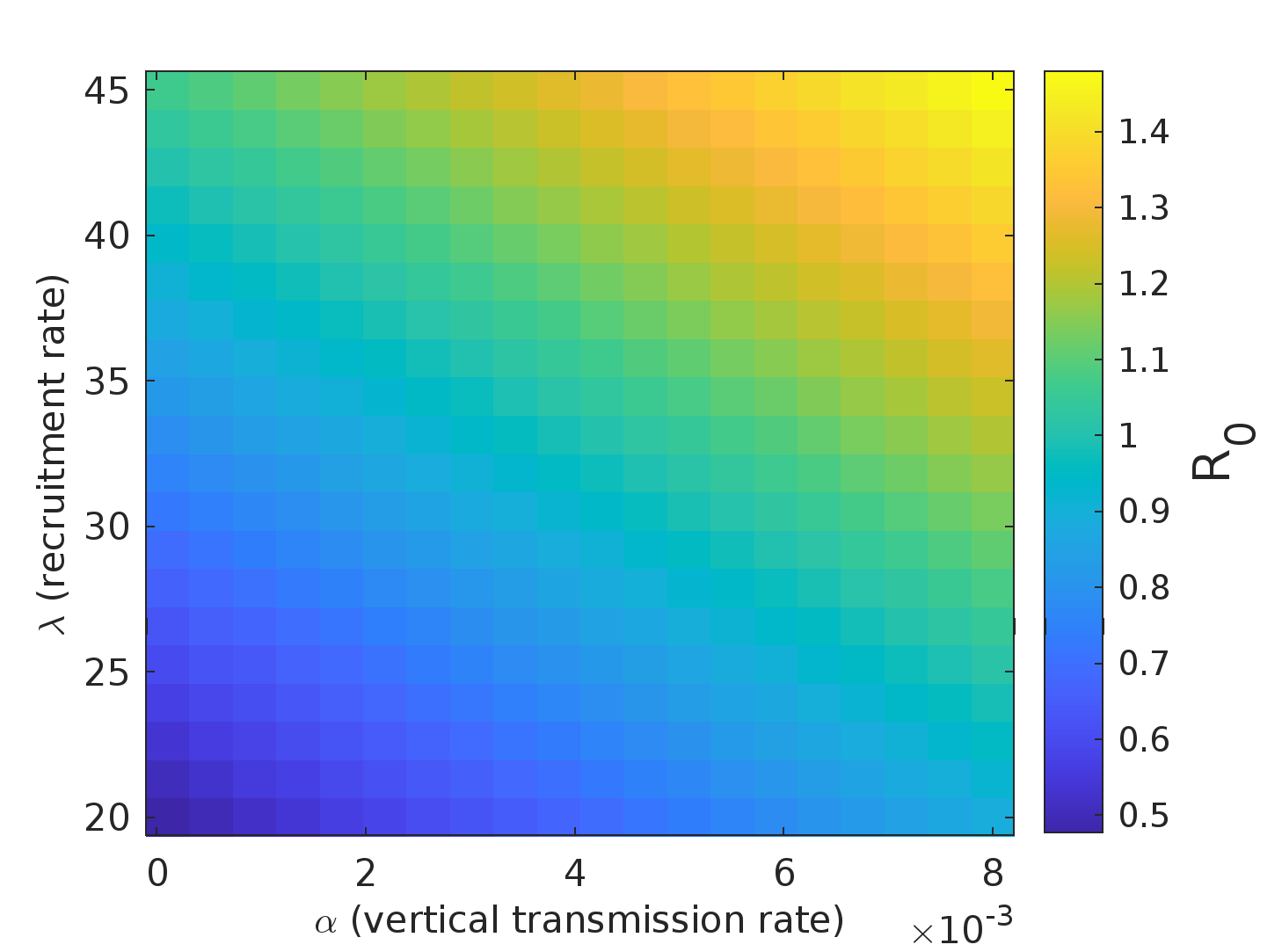}
  \end{minipage}
  \hfill
  \begin{minipage}{0.50\textwidth}
    \includegraphics[width=\linewidth]{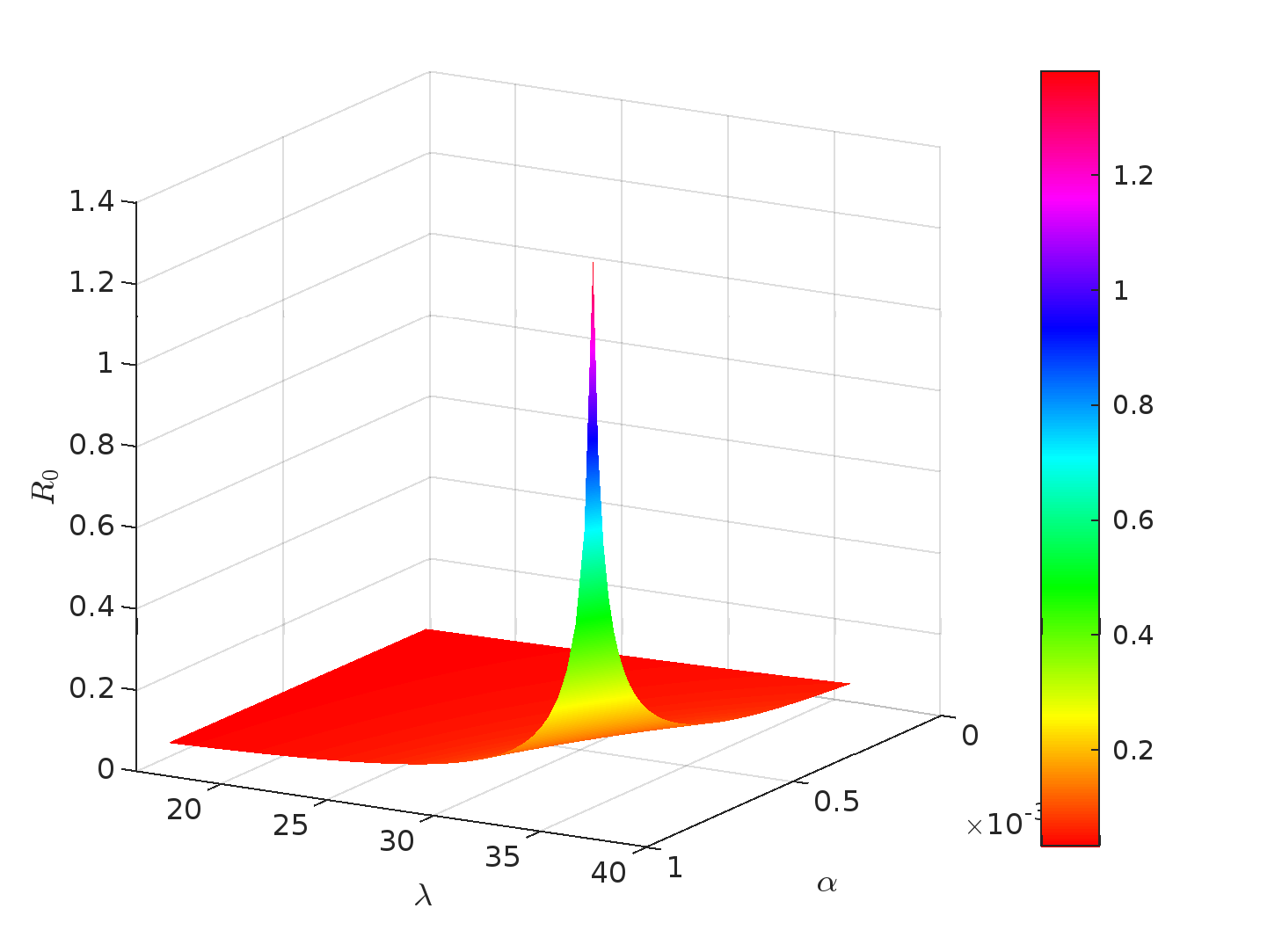}
  \end{minipage}
  \caption{Heatmap and Surface plot of $R_0$ when varying $\alpha$ and $\lambda$.}
  \label{Sens_3}
\end{figure}

\begin{figure}[htbp]
  \centering
  \begin{minipage}{0.40\textwidth}
    \includegraphics[width=\linewidth]{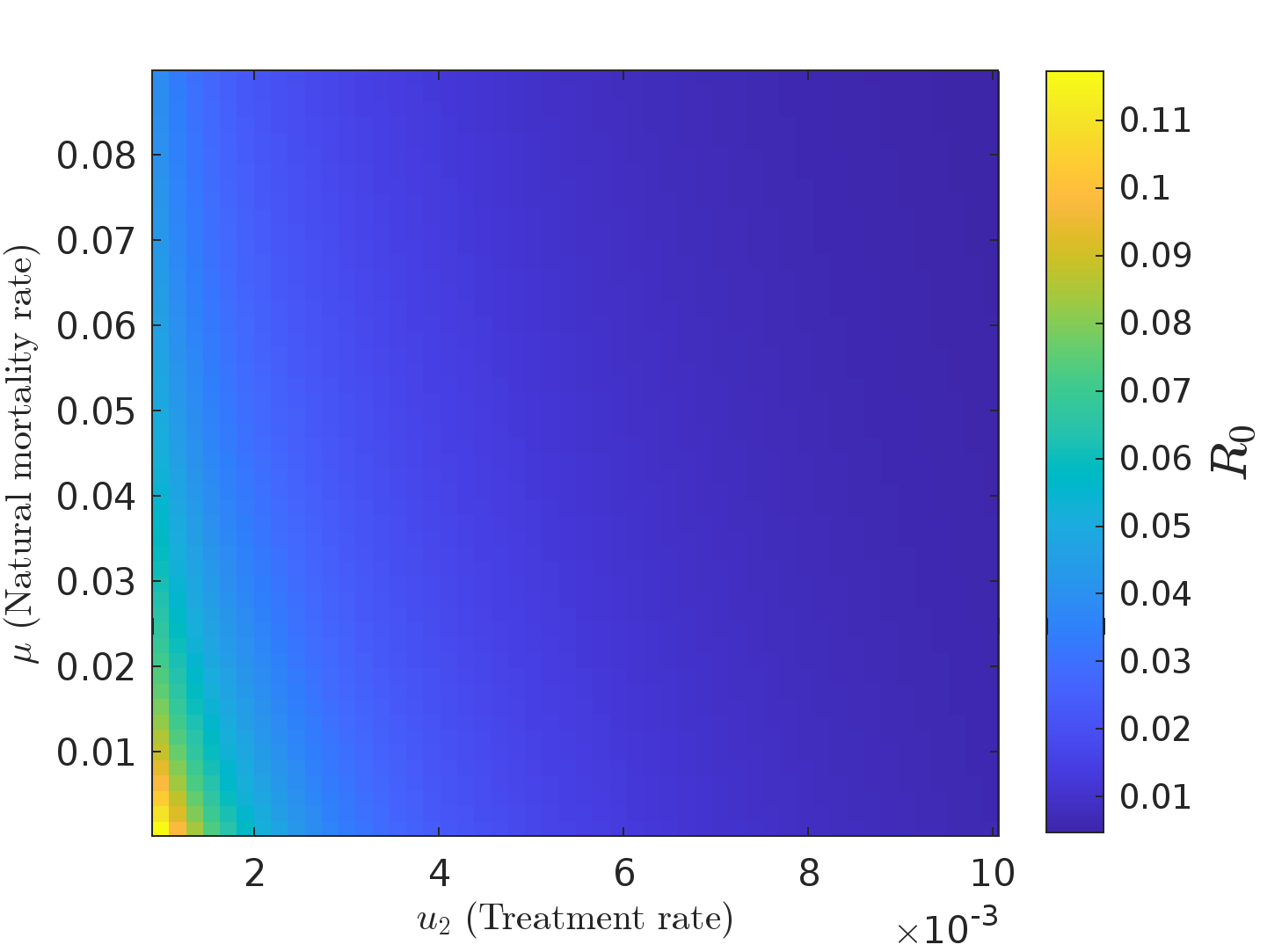}
  \end{minipage}
  \hfill
  \begin{minipage}{0.55\textwidth}
    \includegraphics[width=\linewidth]{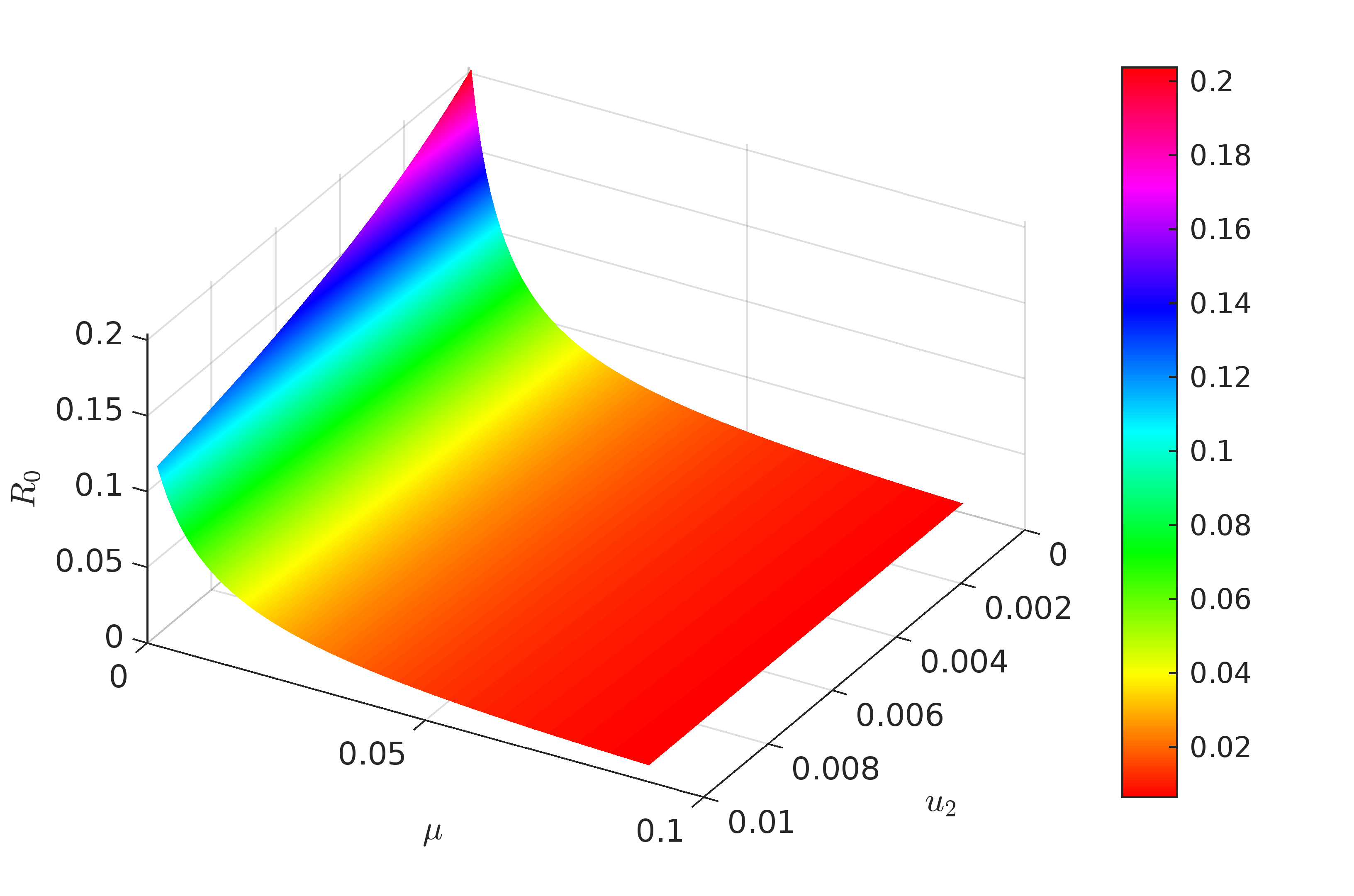}
  \end{minipage}
  \caption{Heatmap and Surface plot of $R_0$ when varying $u_2$ and $\mu$.}
  \label{Sens_4}
\end{figure}

\begin{figure}[htbp]
  \centering
  \begin{minipage}{0.40\textwidth}
    \includegraphics[width=\linewidth]{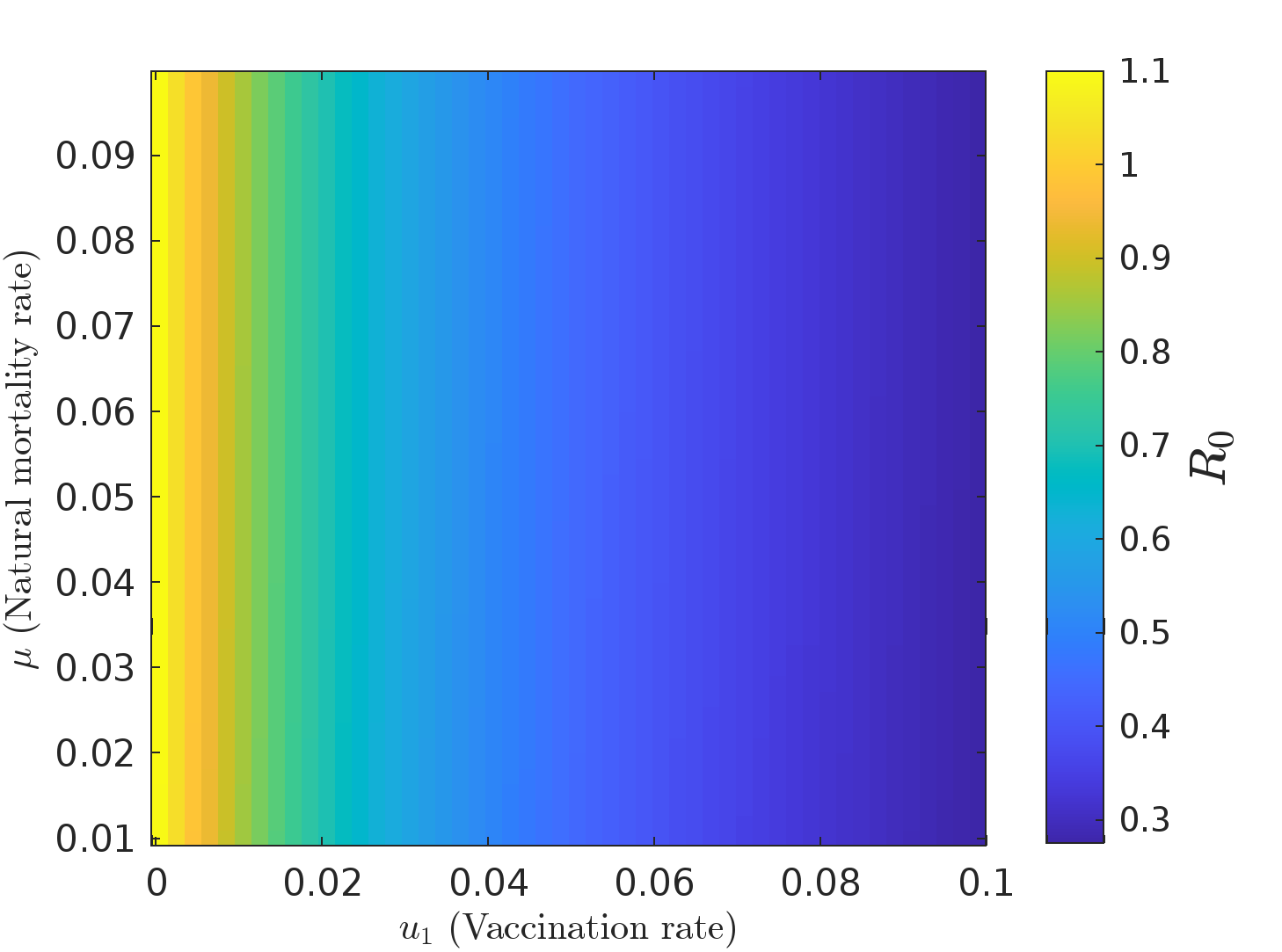}
  \end{minipage}
  \hfill
  \begin{minipage}{0.55\textwidth}
    \includegraphics[width=\linewidth]{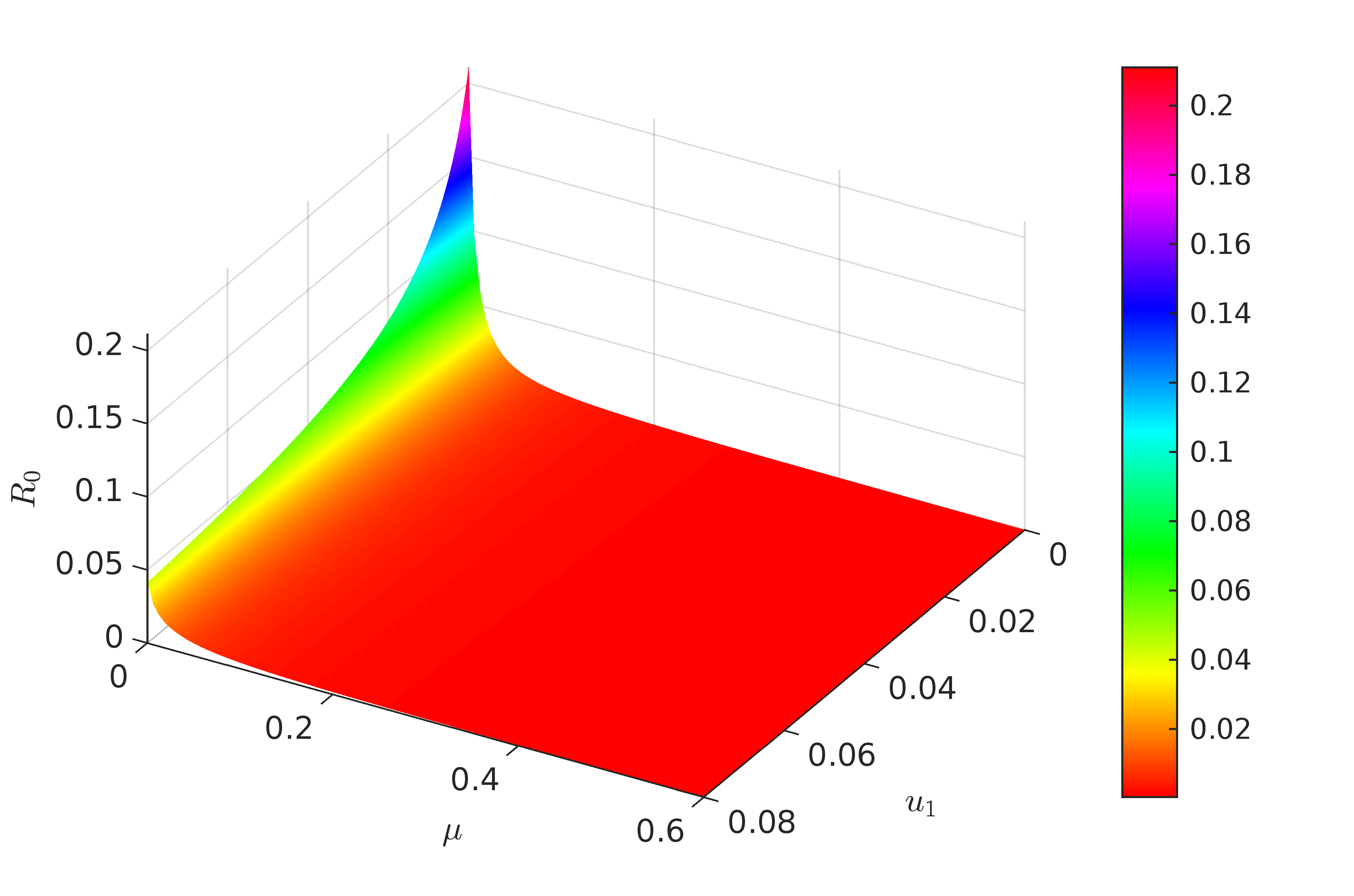}
  \end{minipage}
  \caption{Heatmap and Surface plot of $R_0$ when varying $u_1$ and $\mu$.}
  \label{Sens_5}
\end{figure}

\begin{figure}[htbp]
  \centering
  \begin{minipage}{0.40\textwidth}
    \includegraphics[width=\linewidth]{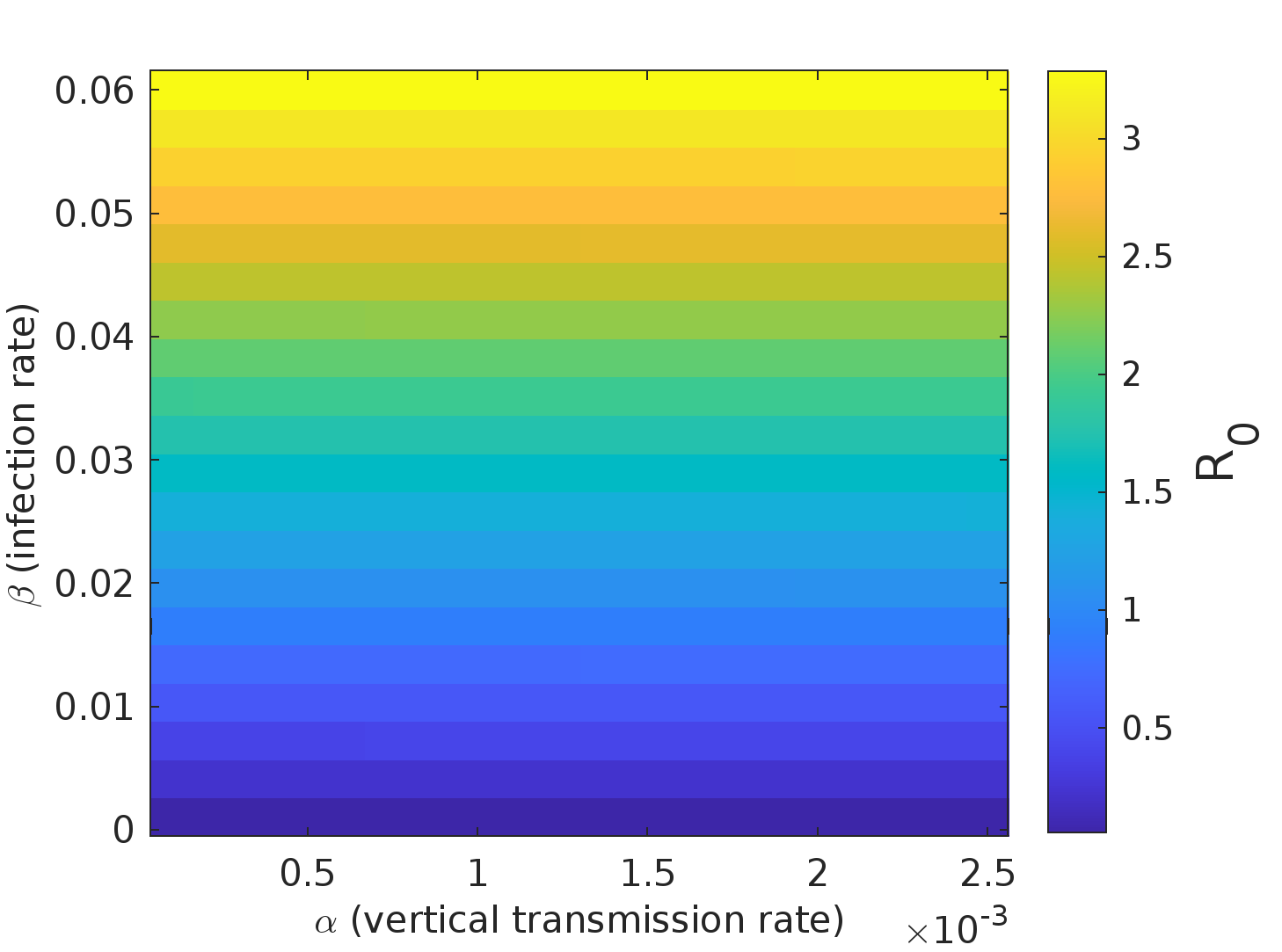}
  \end{minipage}
  \hfill
  \begin{minipage}{0.55\textwidth}
    \includegraphics[width=\linewidth]{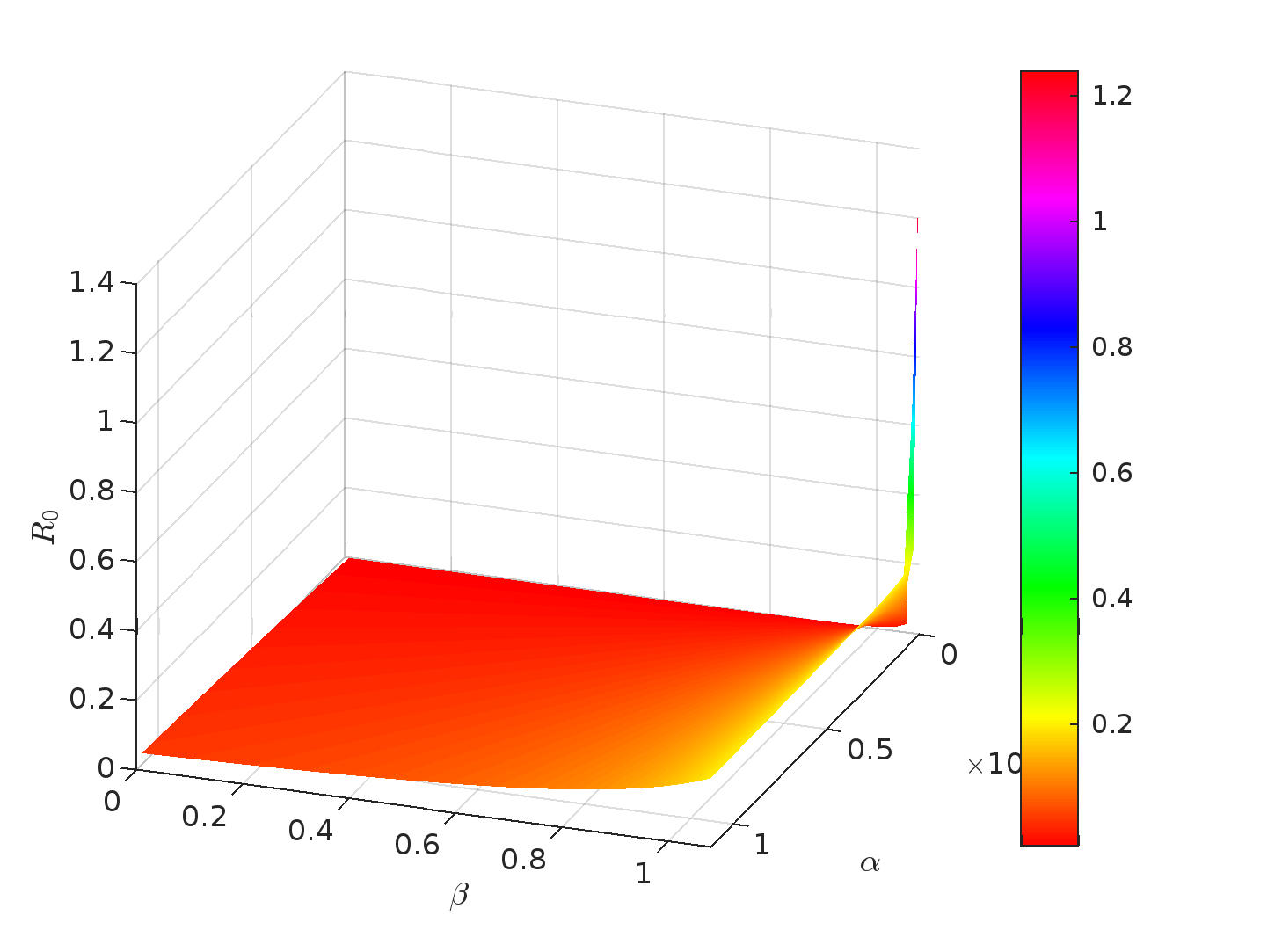}
  \end{minipage}
  \caption{Heatmap and Surface plot of $R_0$ when varying $\alpha$ and $\beta$.}
  \label{Sens_6}
\end{figure}

\begin{figure}[htbp]
  \centering
  \begin{minipage}{0.40\textwidth}
    \includegraphics[width=\linewidth]{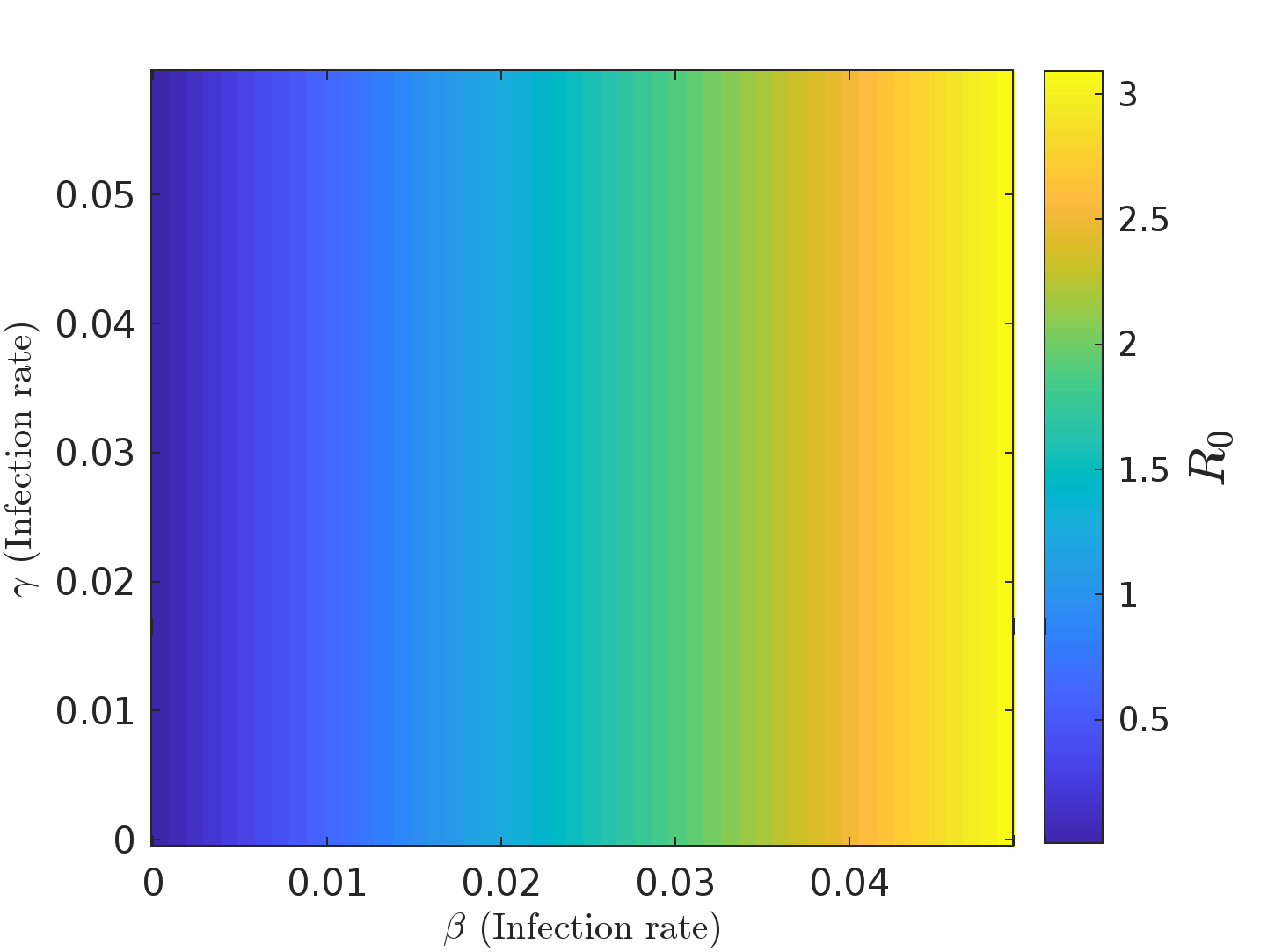}
  \end{minipage}
  \hfill
  \begin{minipage}{0.55\textwidth}
    \includegraphics[width=\linewidth]{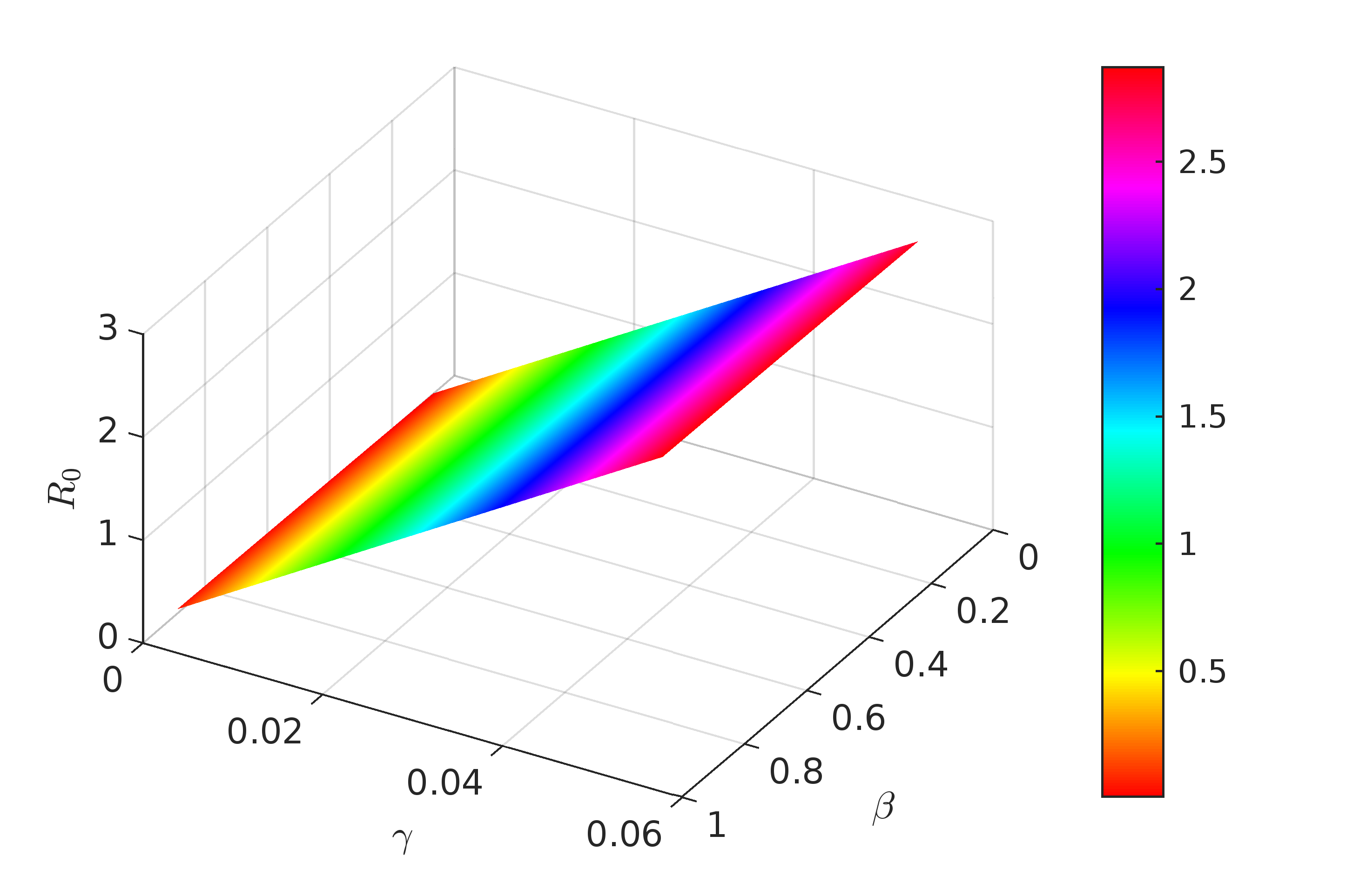}
  \end{minipage}
  \caption{Heatmap and Surface plot of $R_0$ when varying $\beta$ and $\gamma$.}
  \label{Sens_7}
\end{figure}

\begin{figure}[htbp]
  \centering
  \begin{minipage}{0.40\textwidth}
    \includegraphics[width=\linewidth]{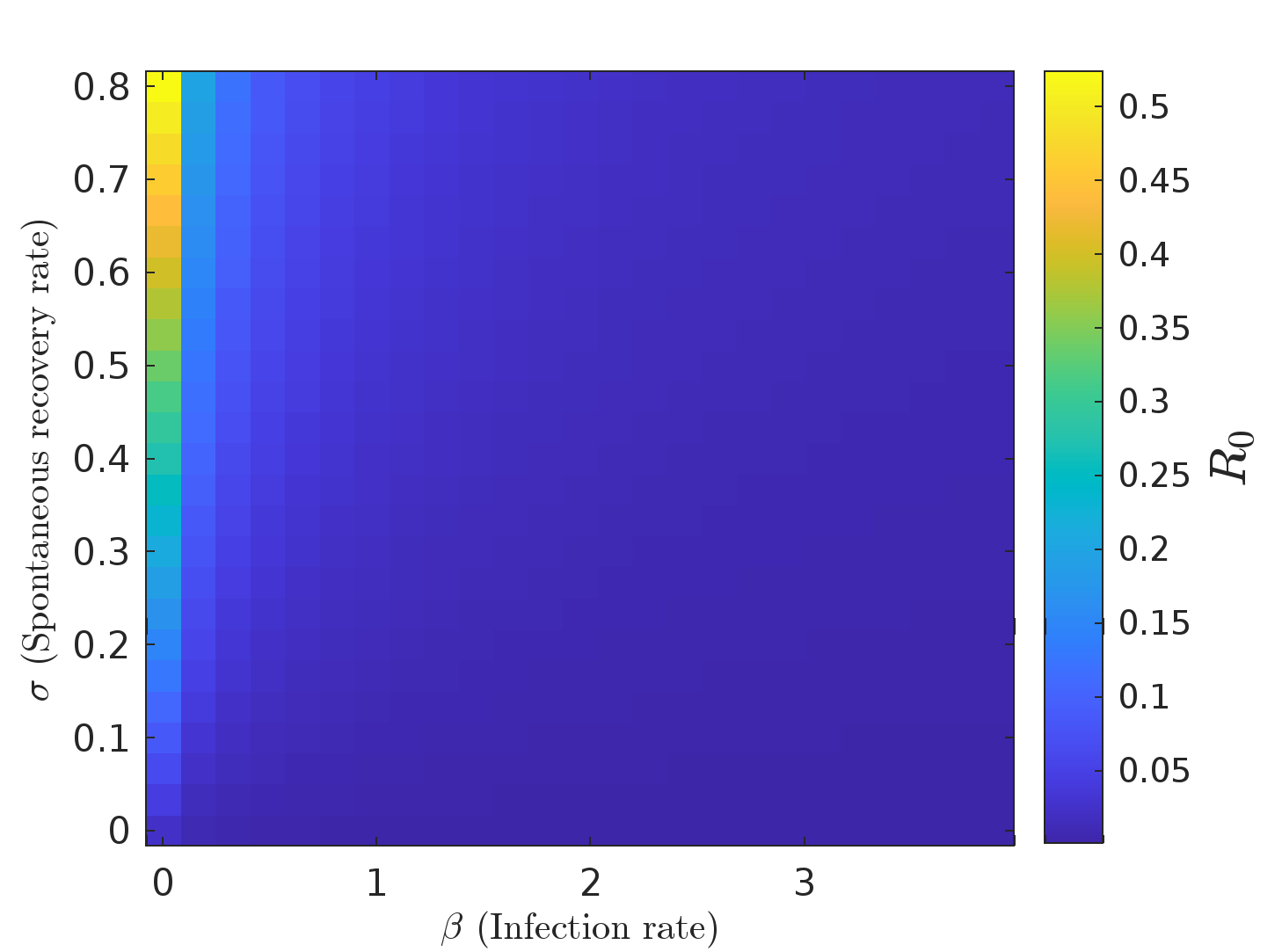}
  \end{minipage}
  \hfill
  \begin{minipage}{0.55\textwidth}
    \includegraphics[width=\linewidth]{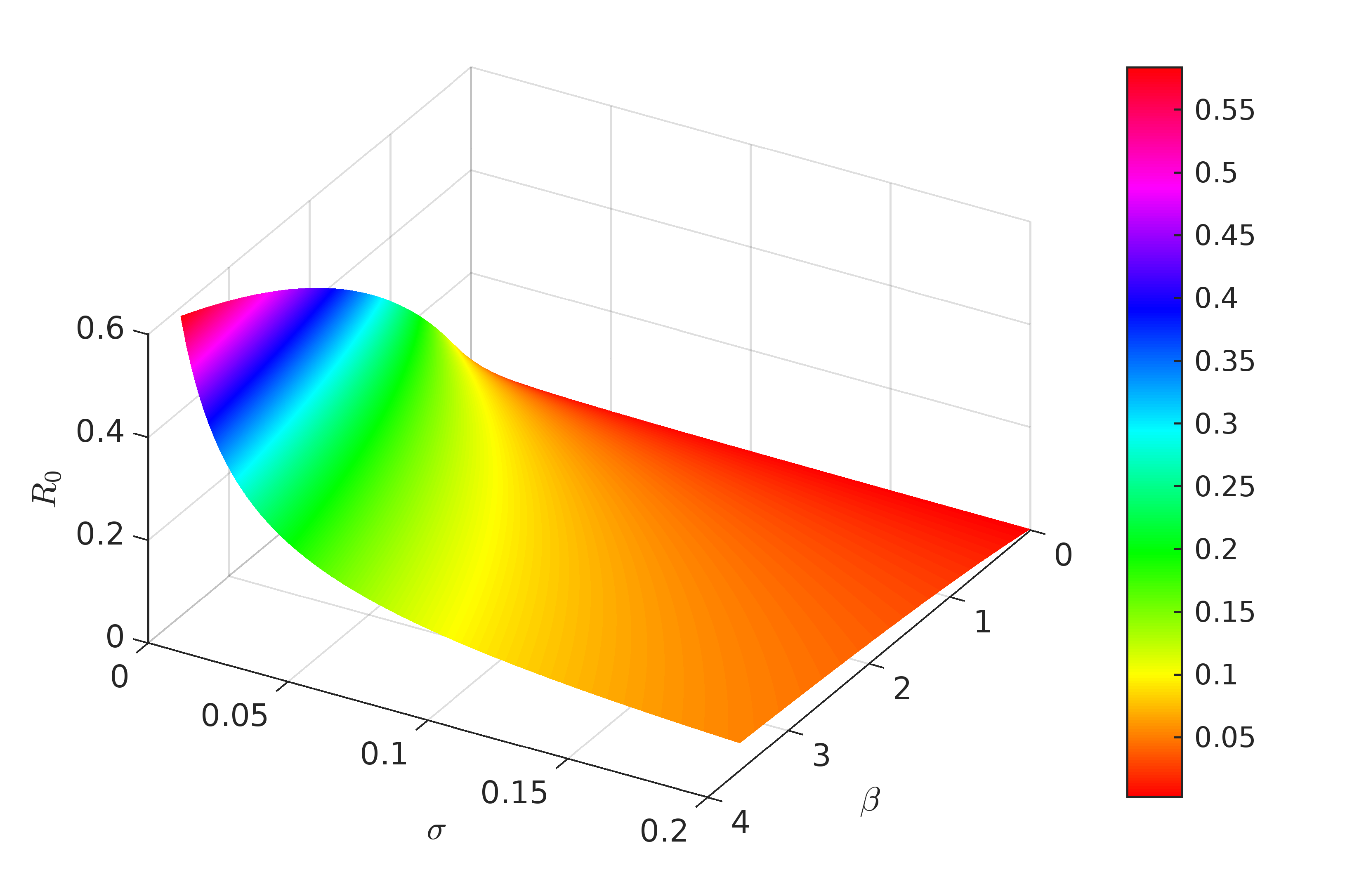}
  \end{minipage}
  \caption{Heatmap and Surface plot of $R_0$ when varying $\beta$ and $\sigma$.}
  \label{Sens_8}
\end{figure}

\begin{figure}[htbp]
  \centering
  \begin{minipage}{0.40\textwidth}
    \includegraphics[width=\linewidth]{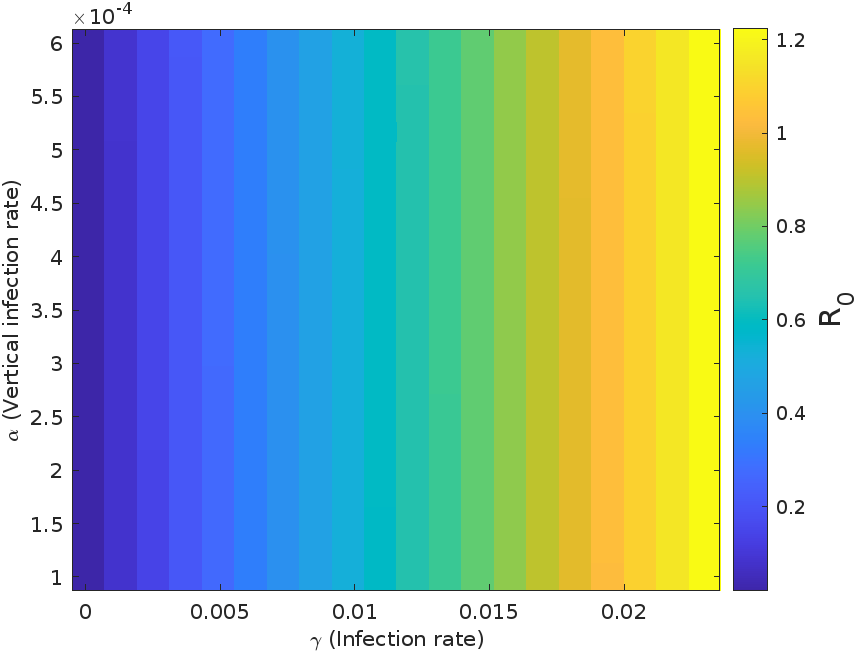}
  \end{minipage}
  \hfill
  \begin{minipage}{0.55\textwidth}
    \includegraphics[width=\linewidth]{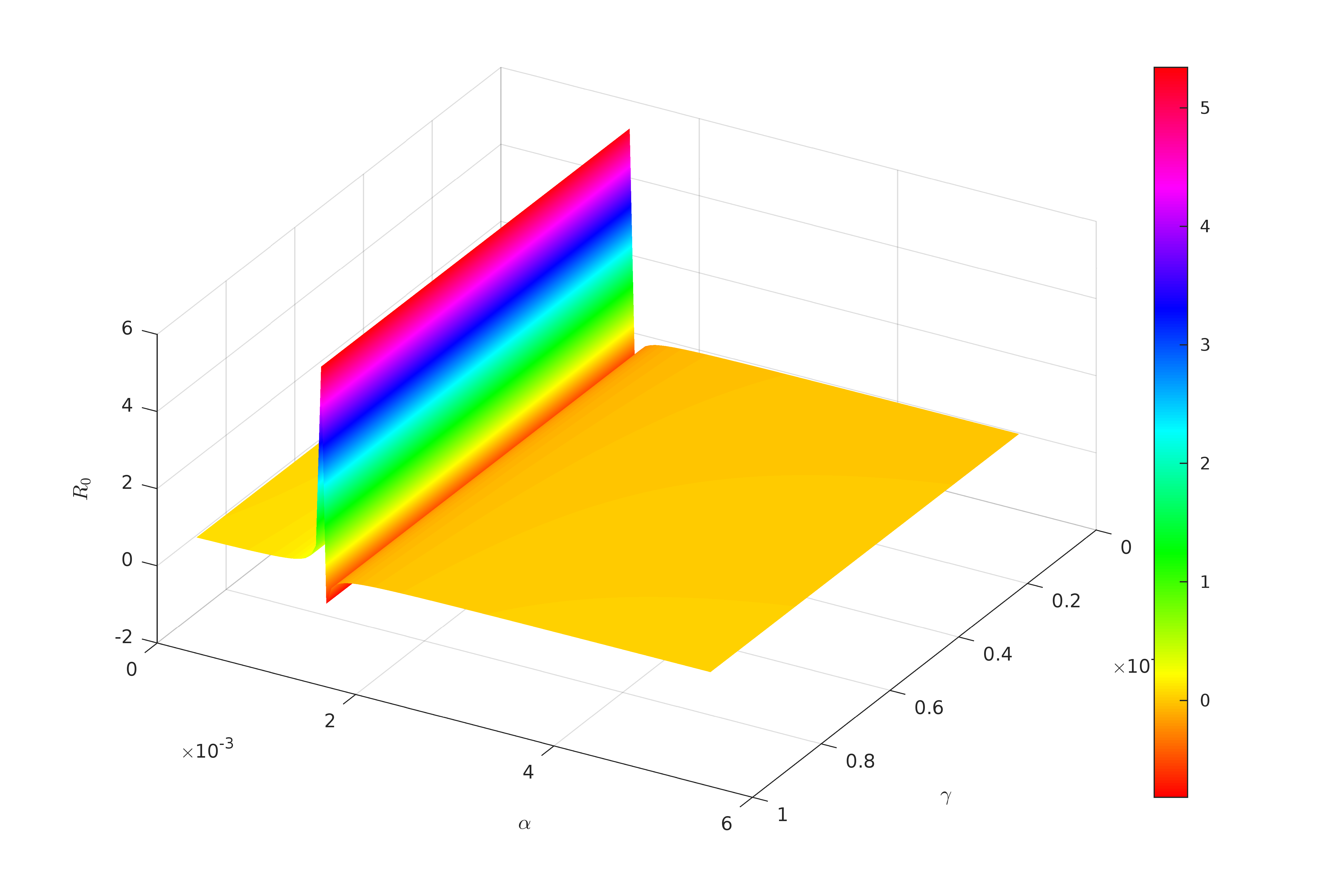}
  \end{minipage}
  \caption{Heatmap and Surface plot of $R_0$ when varying $\gamma$ and $\alpha$.}
  \label{Sens_9}
\end{figure}

\begin{figure}[htbp]
  \centering
  \begin{minipage}{0.40\textwidth}
    \includegraphics[width=\linewidth]{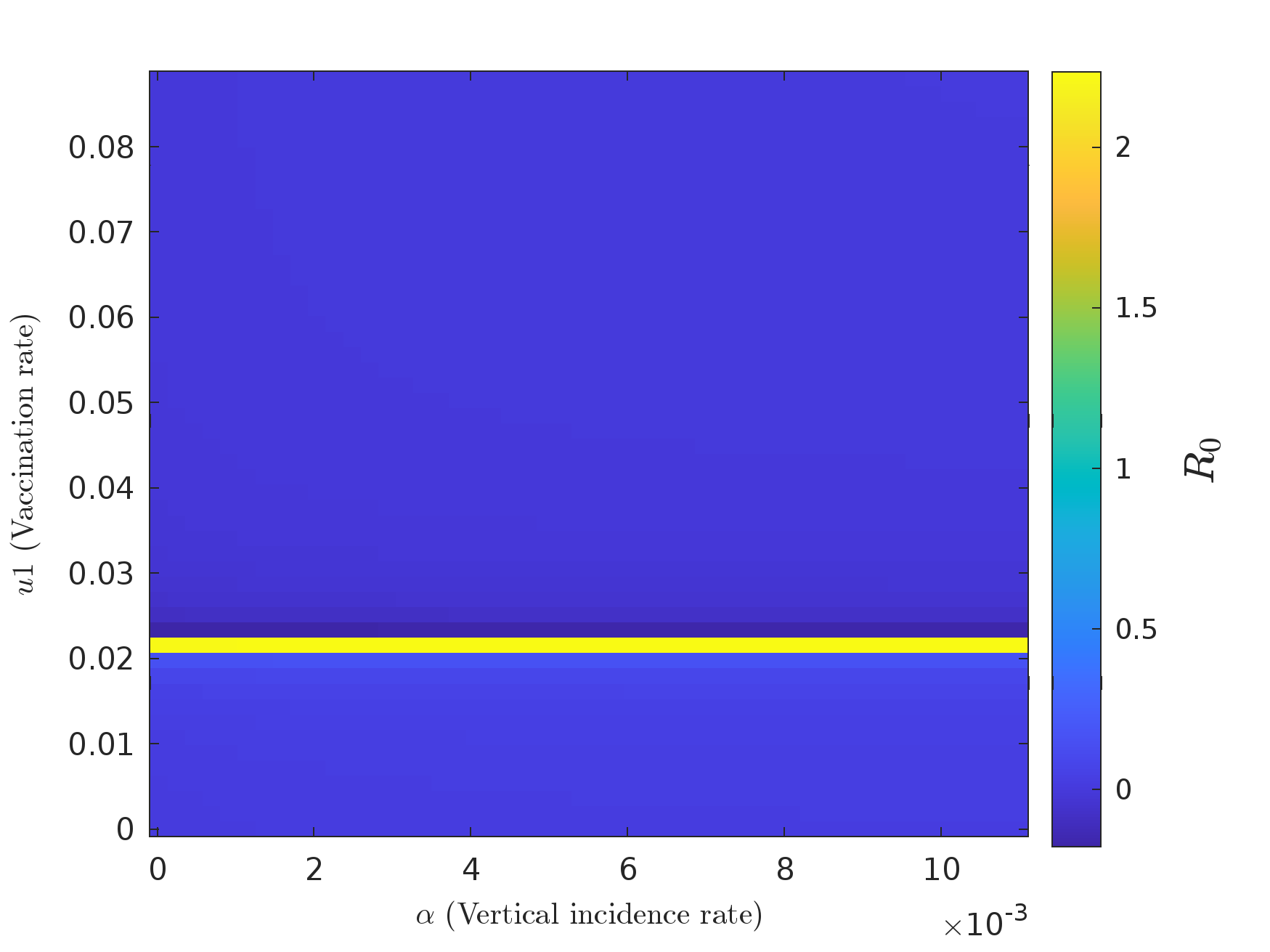}
  \end{minipage}
  \hfill
  \begin{minipage}{0.55\textwidth}
    \includegraphics[width=\linewidth]{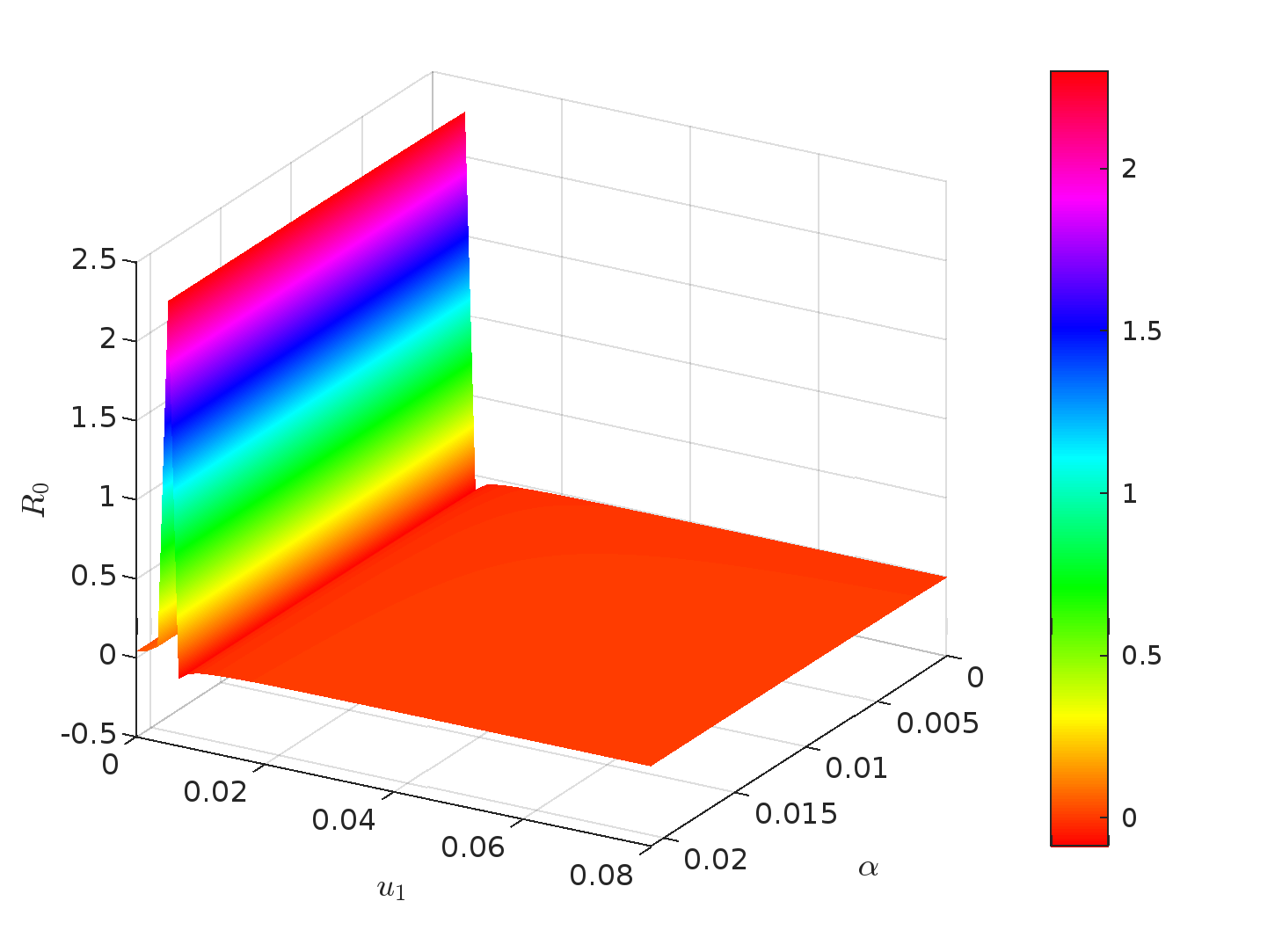}
  \end{minipage}
  \caption{Heatmap and Surface plot of $R_0$ when varying $\alpha$ and $u_1$.}
  \label{Sens_10}
\end{figure}

\begin{figure}[htbp]
  \centering
  \begin{minipage}{0.40\textwidth}
    \includegraphics[width=\linewidth]{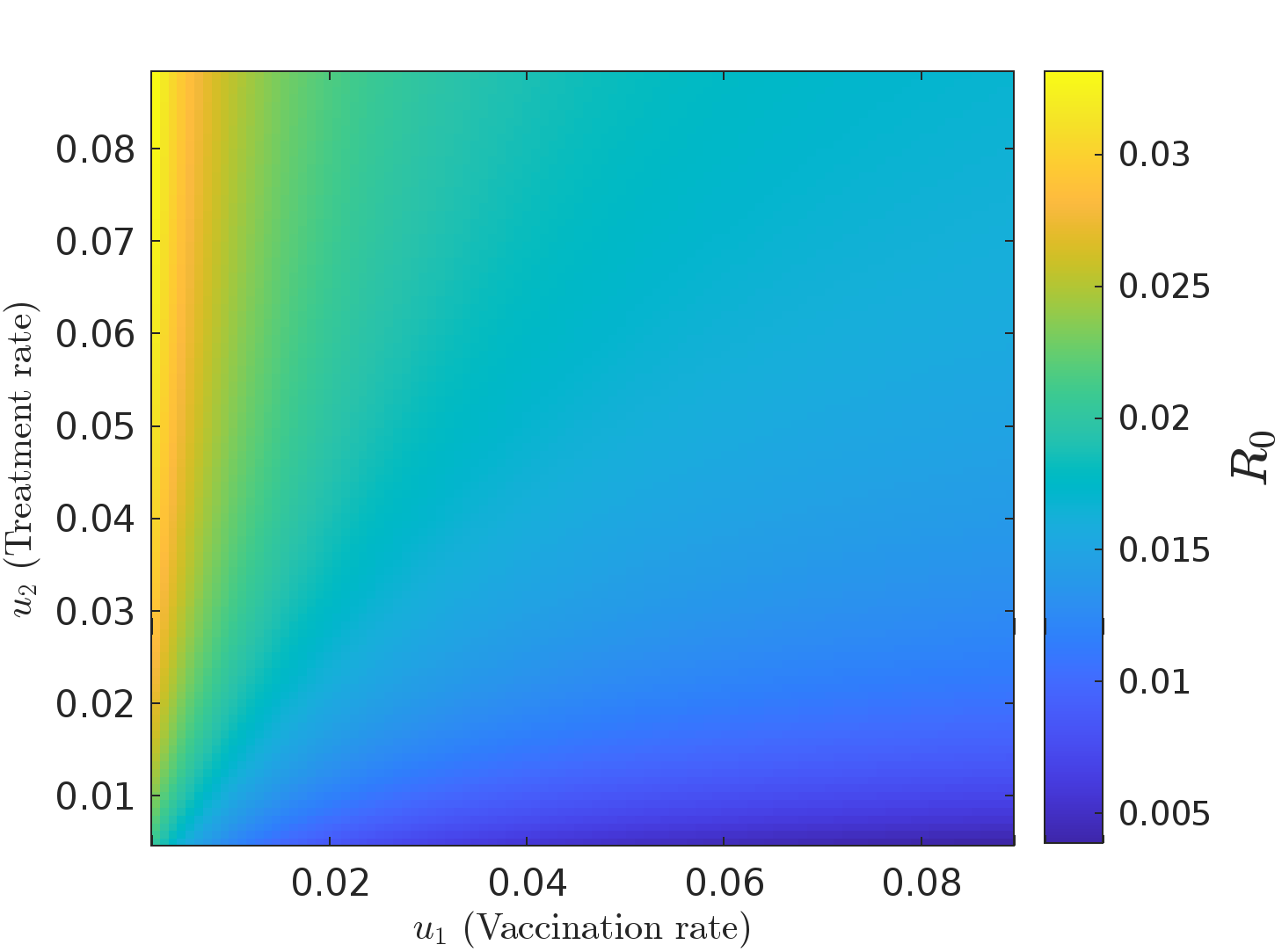}
  \end{minipage}
  \hfill
  \begin{minipage}{0.55\textwidth}
    \includegraphics[width=\linewidth]{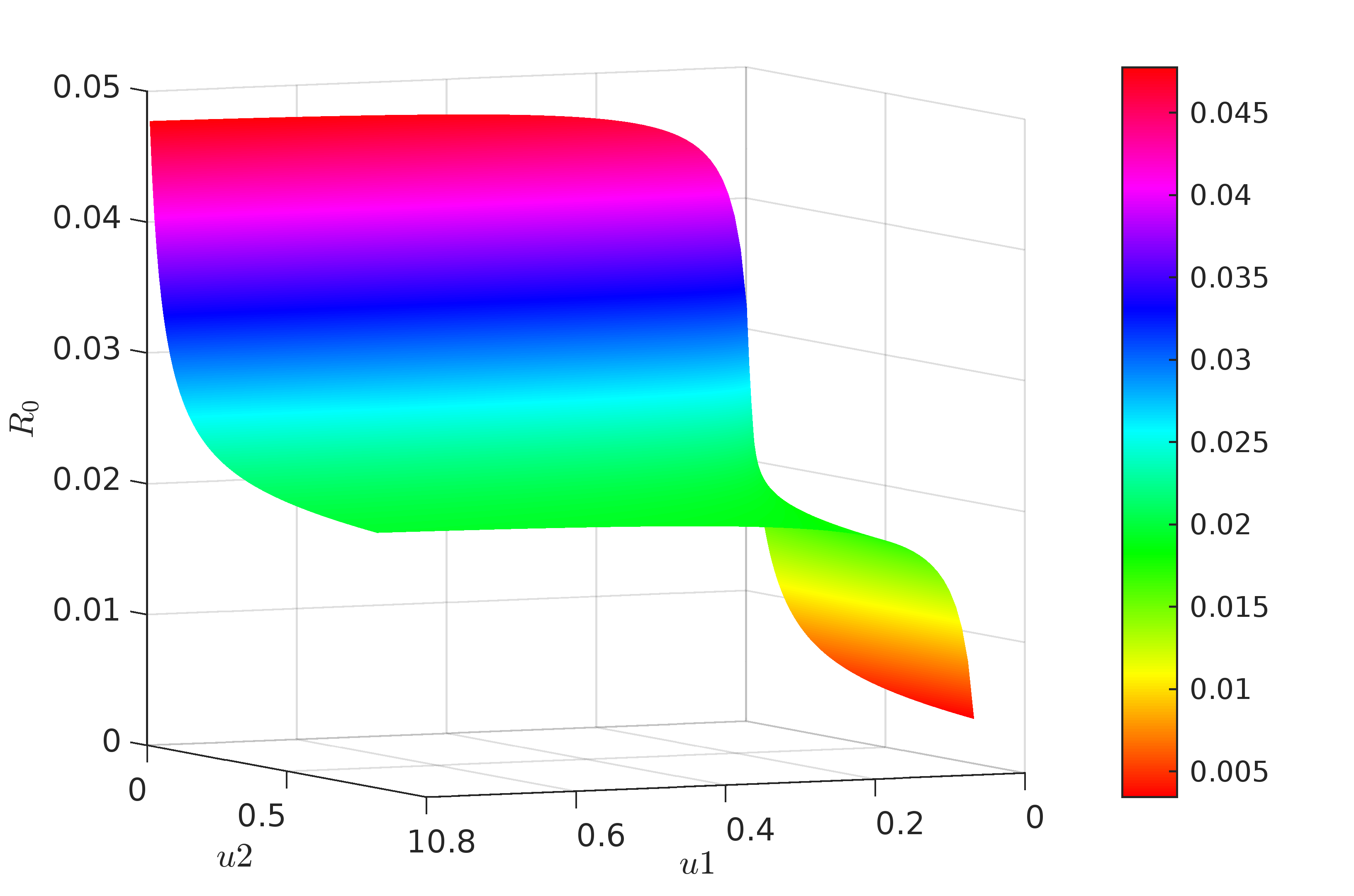}
  \end{minipage}
  \caption{Heatmap and Surface plot of $R_0$ when varying $u_1$ and $u_2$.}
  \label{Sens_11}
\end{figure}

\begin{figure}[htbp]
  \centering
  \begin{minipage}{0.40\textwidth}
    \includegraphics[width=\linewidth]{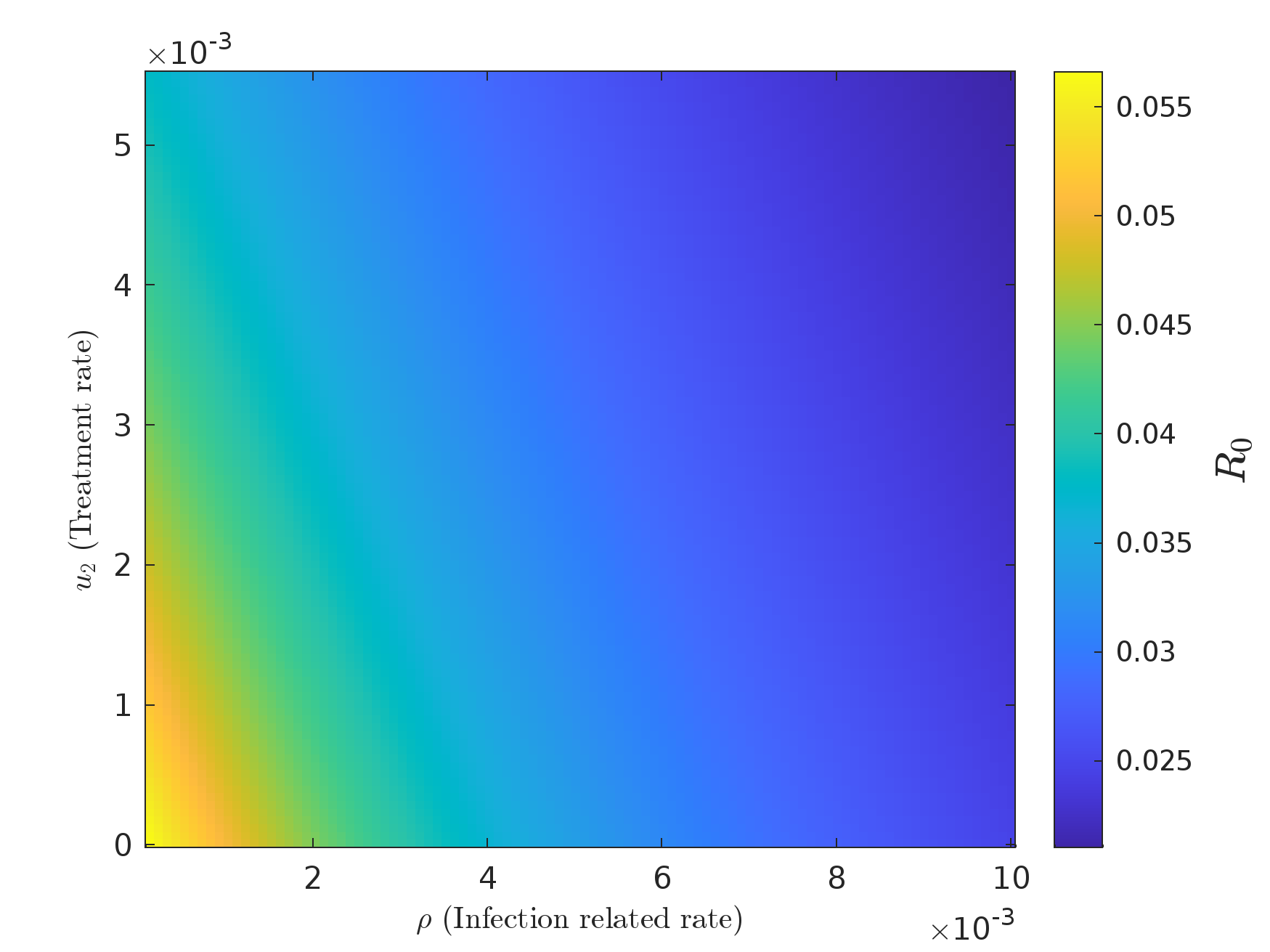}
  \end{minipage}
  \hfill
  \begin{minipage}{0.55\textwidth}
    \includegraphics[width=\linewidth]{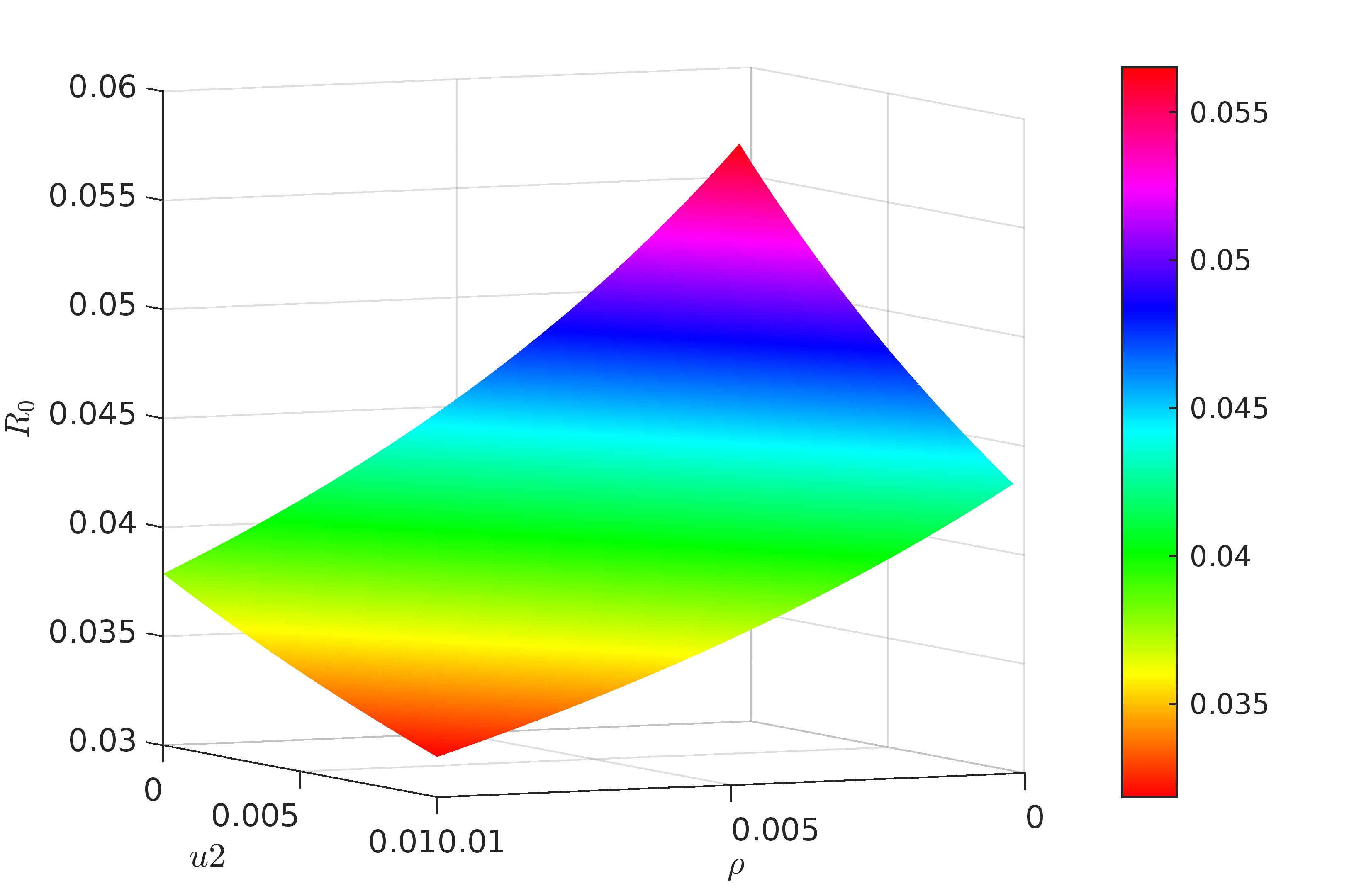}
  \end{minipage}
  \caption{Heatmap and Surface plot of $R_0$ when varying $\rho$ and $u_2$.}
  \label{Sens_12}
\end{figure}

\hfill
\begin{figure}[htbp]
  \centering
  \begin{subfigure}[htbp]{0.65\textwidth}
    \includegraphics[width=\textwidth]{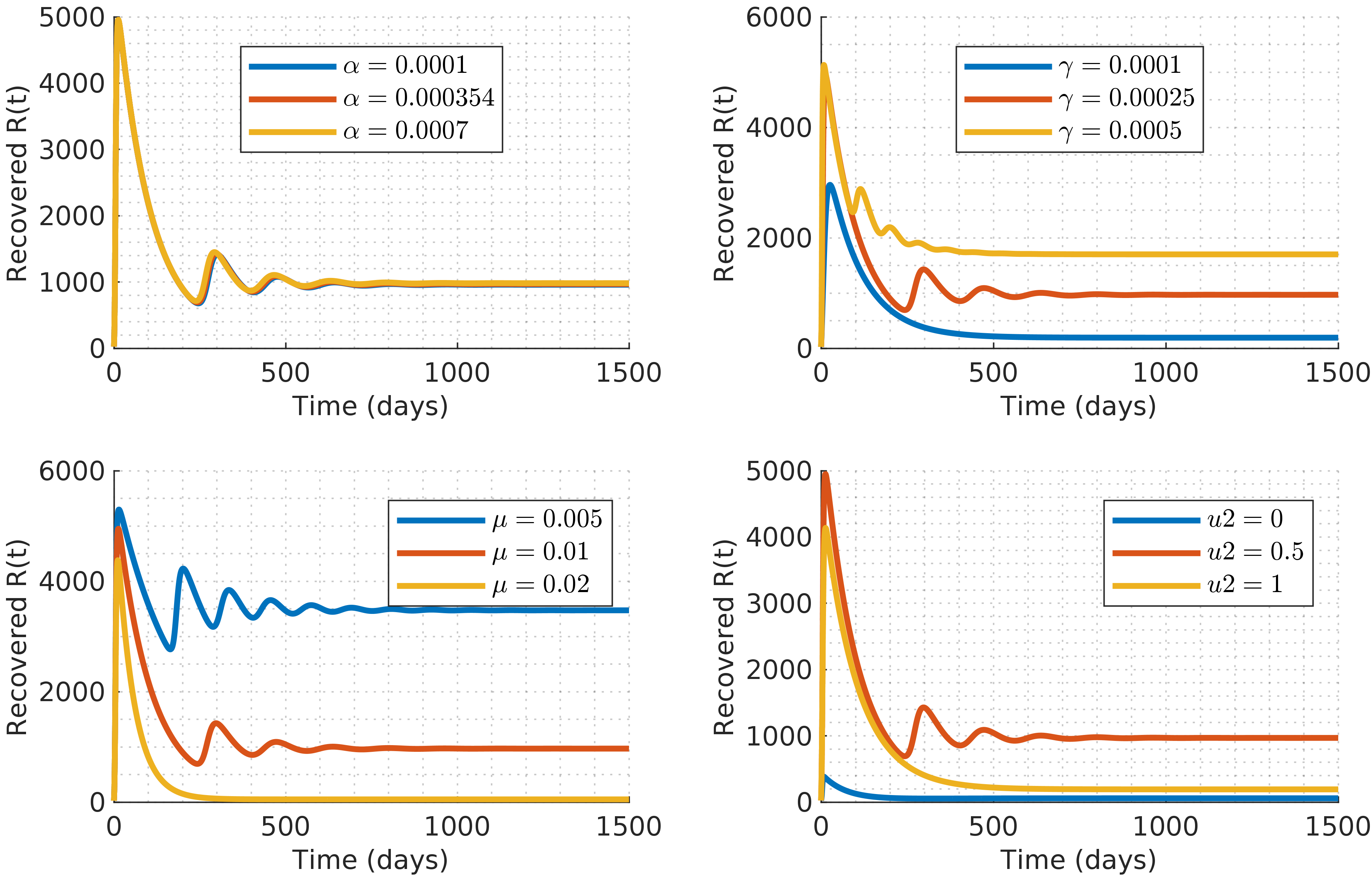}
    \caption{Effect of different values of $\alpha$, $\gamma$, $\mu$, and $u_2$ on R(t)}
    \label{fig:com1}
  \end{subfigure}

  \vspace{1cm}
  \begin{subfigure}[htbp]{0.65\textwidth}
    \includegraphics[width=\textwidth]{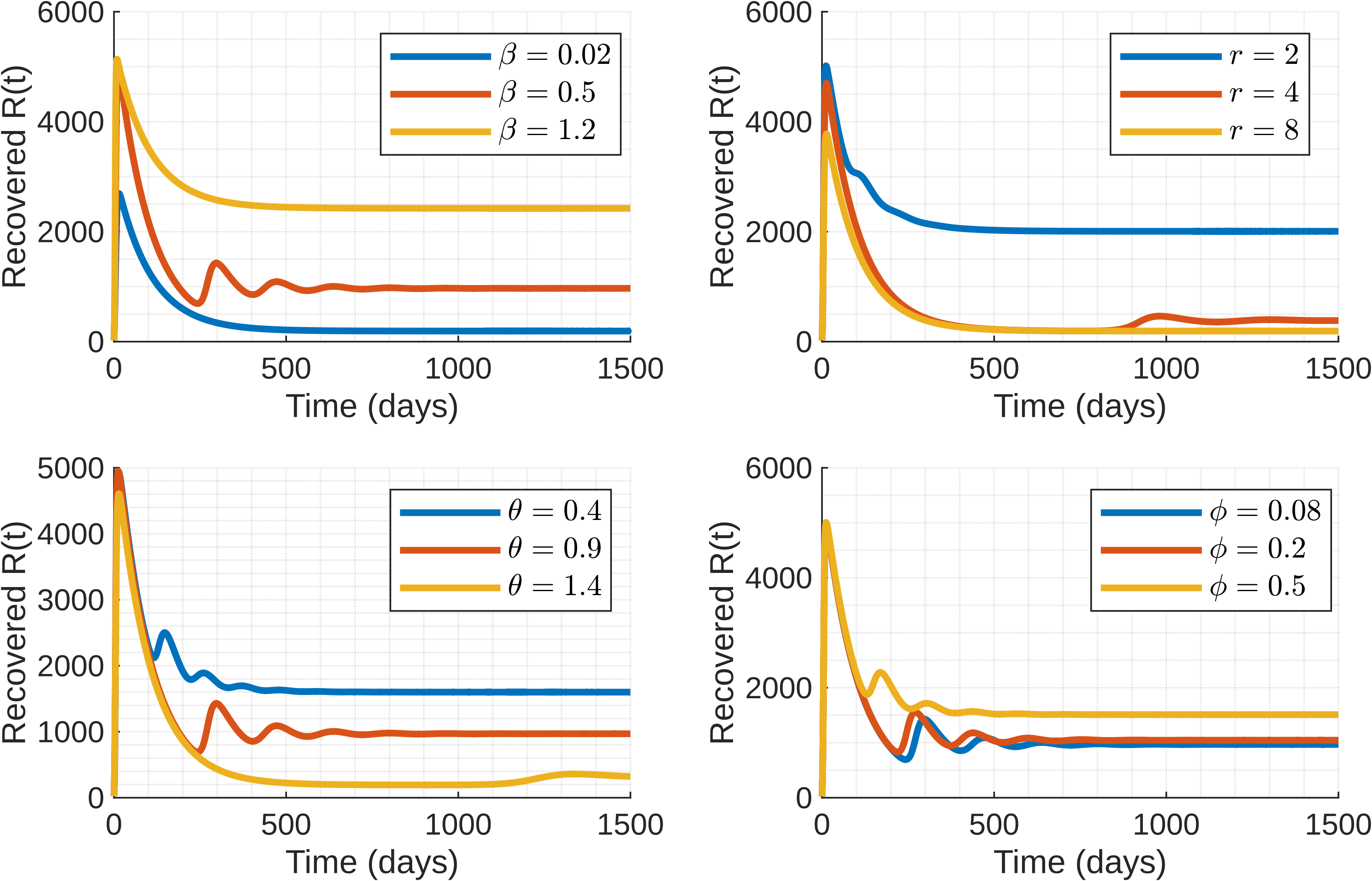}
    \caption{Effect of different values of $\beta$, $r$, $\theta$, and $\phi$ on R(t)}
    \label{fig:com2}
  \end{subfigure}
  \caption{Effect of Parameter Variations on Recovered Population R(t)}
  \label{fig:combined}
\end{figure}

\newpage
\subsection{Dynamics of model on the basis of \texorpdfstring{$\mathcal{R}_0$}{R0}}
According to the theorem~\ref{Theorem1}, the disease-free equilibrium (DFE) of the model~\eqref{model eqs.} is locally asymptotically stable whenever the basic reproduction number satisfies \( \mathcal{R}_0 < 1 \). This implies that, under this condition, the infection cannot invade the population and will eventually die out. But, once \( \mathcal{R}_0 > 1 \), the disease-free equilibrium loses stability and the model allows for the presence of an endemic state. In this case, the infection remains in the population throughout the entire simulation, with the endemic state being locally asymptotically stable, meaning that the disease or infection remains in the population and persists indefinitely. This behavior is illustrated in figures \ref{fig: Convergence 01} and \ref{fig: Convergence 02}. This result aligns with real-world expectations, where diseases tend to persist if transmission outweighs recovery or intervention efforts. While theorem~\ref{Theorem4} suggests that when \( \mathcal{R}_0 > 1 \), the model is globally asymptotically stable at the endemic equilibrium. Under this condition, the infection persists in the population over time, regardless of what we choose for the initial conditions, and aligns with real-world epidemiological observations where sustained transmission occurs if the infection rate exceeds the combined recovery and intervention rates. See fig. \ref{Global convergence}. The simulations were carried out using the parameter values presented in Table \ref{tab:R0_parameters}. The resulting graphs are given below to illustrate the system behavior under the specified conditions.

\begin{table}[htbp]
\centering
\renewcommand{\arraystretch}{1.2}
\begin{tabular}{>{\centering\arraybackslash}p{2.2cm} >{\centering\arraybackslash}p{5.5cm} >{\centering\arraybackslash}p{2.2cm} >{\centering\arraybackslash}p{2.2cm}}
\hline
\textbf{Symbol} & \textbf{Parameter Description} & \textbf{\(\mathcal{R}_0\) at DFE} & \textbf{\(\mathcal{R}_0\) at EE} \\
\hline
\(\lambda\) & Recruitment rate & 25 & 25 \\
\(\alpha\) & Vertical transmission rate & 0.0001 & 0.000354 \\
\(\gamma\) & Infection rate from acute class & 0.00008 & 0.00025 \\
\(\beta\) & Infection rate from chronic class & 0.0006 & 0.02 \\
\(\sigma\) & Spontaneous recovery rate & 0.09 & 0.069 \\
\(u_1\) & Vaccination rate & 0.006 & 0.000835 \\
\(u_2\) & Treatment rate & 0.008 & 0.00995 \\
\(\phi\) & Treatment saturation parameter & 0.002 & 0.08 \\
\(\mu\) & Natural death rate & 0.02 & 0.01 \\
\(\theta\) & Acute-to-chronic transition rate & 0.000002 & 0.9 \\
\(\rho\) & Disease-induced death rate & 0.008 & 0.0028 \\
\(r\) & Treatment recovery coefficient & 5 & 3 \\
\hline
\end{tabular}
\caption{Parameter values used for simulating low and high \(\mathcal{R}_0\) dynamics.}
\label{tab:R0_parameters}
\end{table}

\begin{figure}[htbp]
  \centering
  \begin{subfigure}[htbp]{0.62\textwidth}
    \includegraphics[width=\textwidth]{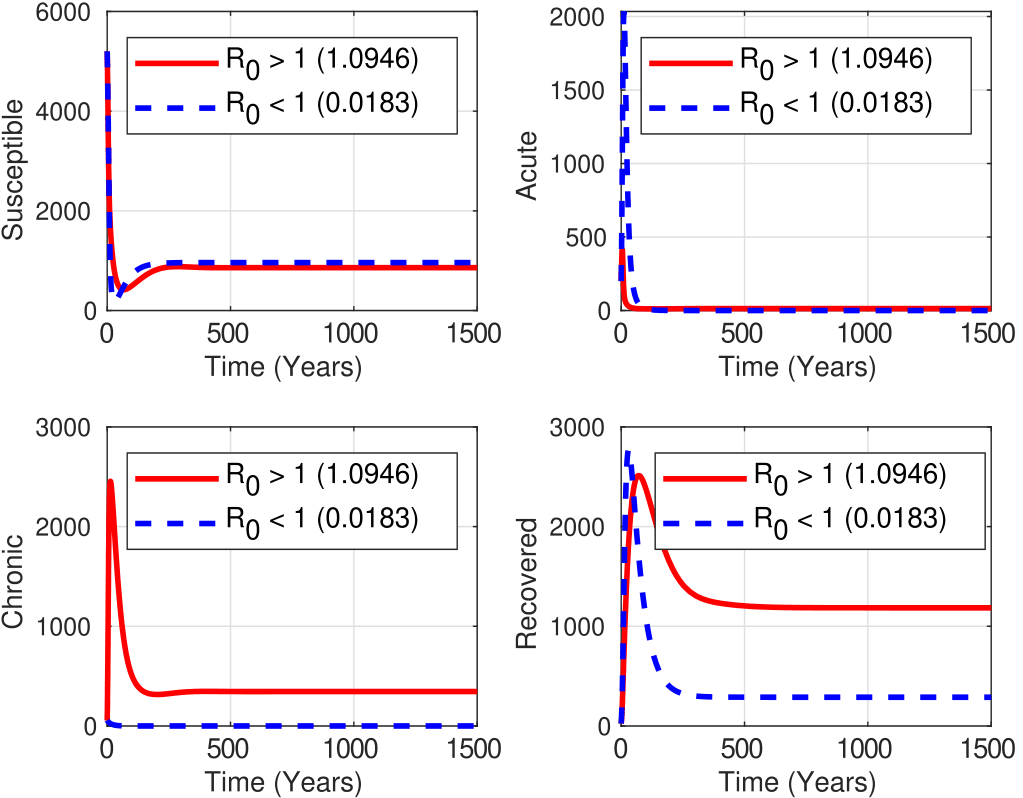}
    \caption{Time series simulation of the model \eqref{model eqs.}, showing the convergence of the solution trajectories.}
    \label{fig: Convergence 01}
  \end{subfigure}

  \begin{subfigure}[htbp]{0.58\textwidth}
    \includegraphics[width=\textwidth]{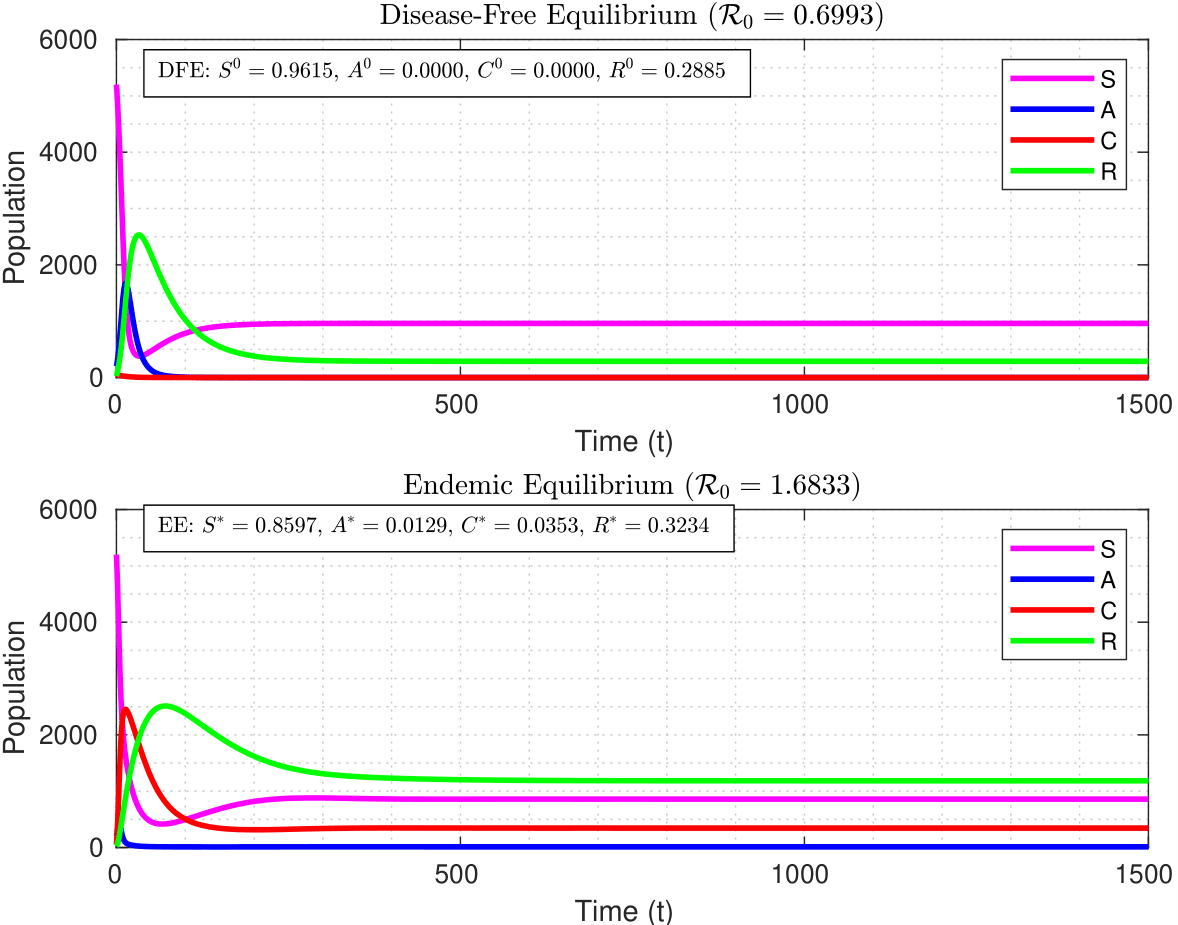}
    \caption{Time series simulation of the model \eqref{model eqs.}, showing the convergence of the solution trajectories.}
    \label{fig: Convergence 02}
  \end{subfigure}
  \caption{Simulation of the model \eqref{model eqs.} illustrating the temporal dynamics of the susceptible (\(S\)), acute infected (\(A\)), chronic infected (\(C\)), and recovered (\(R\)) compartments. The parameters were chosen such that \( \mathcal{R}_0 < 1 \) (blue dotted line) and \( \mathcal{R}_0 > 1 \) (red line). In the first case, the infection gradually dies out, and the system approaches the disease-free equilibrium. In contrast, when \(\mathcal{R}_0 > 1\), the infection persists and the system stabilizes at the endemic equilibrium.} 
\end{figure}

\begin{figure}[htbp]
  \centering
  \begin{minipage}{0.7\textwidth}
    \includegraphics[width=\linewidth]{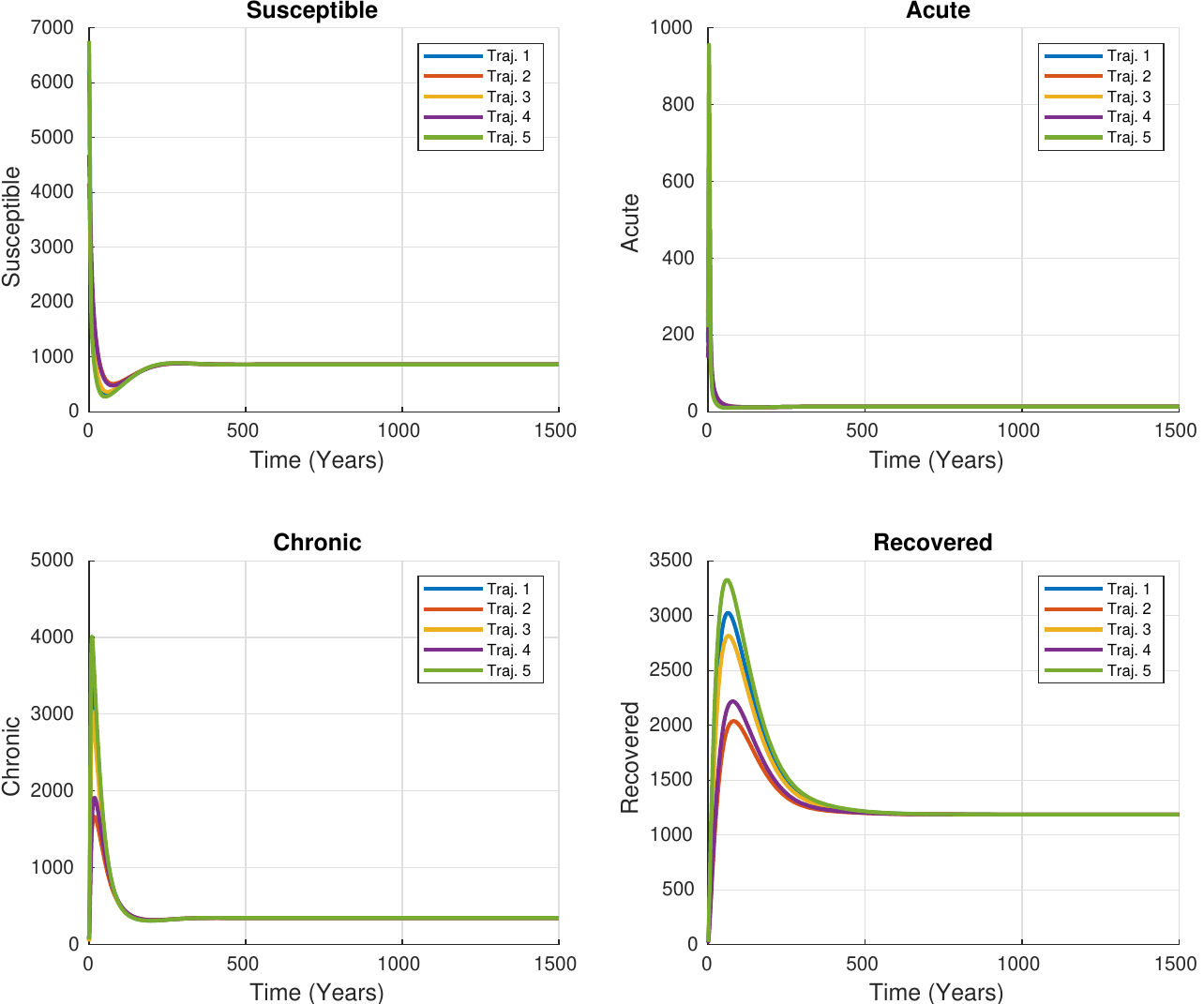}
  \end{minipage}
  \caption{Global stability of the endemic equilibrium when \( \mathcal{R}_0 > 1 \), illustrated for different initial conditions. All solution trajectories converge to the same endemic state, confirming the global asymptotic stability predicted by Theorem \ref{Theorem4}.}
  \label{Global convergence}
\end{figure}

\newpage
\subsection{Comparative Dynamics of Infected Classes Under Varying \texorpdfstring{$\mathcal{R}_0$}{R0}}

Understanding how the basic reproduction number (\(\mathcal{R}_0\)) affects infection dynamics is crucial in infectious disease modeling. A lower \(\mathcal{R}_0\) suggests that each infected individual causes less than one new infection on average, leading to disease elimination. Conversely, a higher \(\mathcal{R}_0\) indicates the potential for sustained transmission and endemicity. Comparison of infected dynamics for both scenarios. Solid lines correspond to the low \(\mathcal{R}_0\) case, and dashed lines correspond to the high \(\mathcal{R}_0\) case. The contrast clearly demonstrates the threshold-driven behavior of disease spread.

The following figure \ref{fig:R0_all} illustrates the time evolution of the acute (\(A\)) and chronic (\(C\)) infected compartments under different values of \(\mathcal{R}_0\).

\begin{figure}[htbp]
    \centering

    \begin{minipage}[htbp]{0.48\textwidth}
        \centering
        \includegraphics[width=\textwidth]{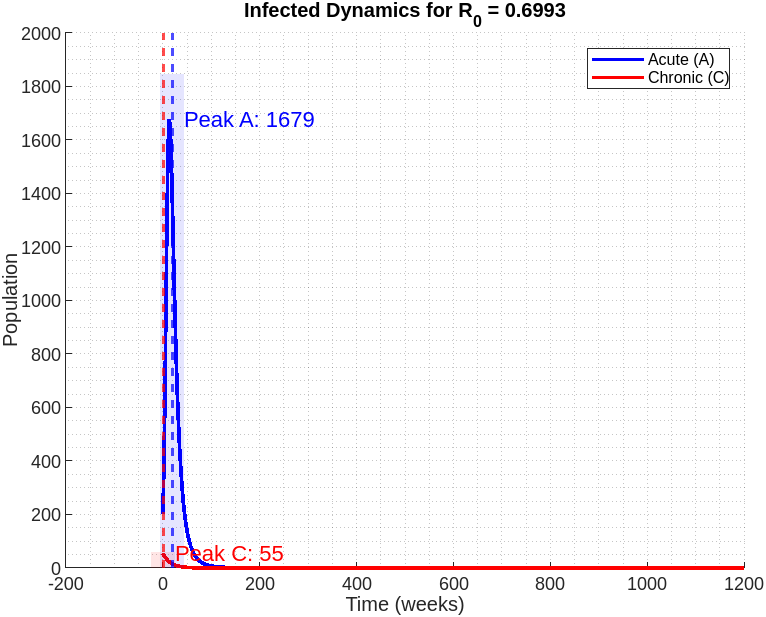}
        \caption*{\textbf{(a)} Infected dynamics under low $\mathcal{R}_0 = 0.6993$. Acute infections rise sharply but decline rapidly, with chronic infections staying low, indicating disease elimination.}
        \label{fig:lowR0_sub}
    \end{minipage}
    \hfill
    \begin{minipage}[htbp]{0.48\textwidth}
        \centering
        \includegraphics[width=\textwidth]{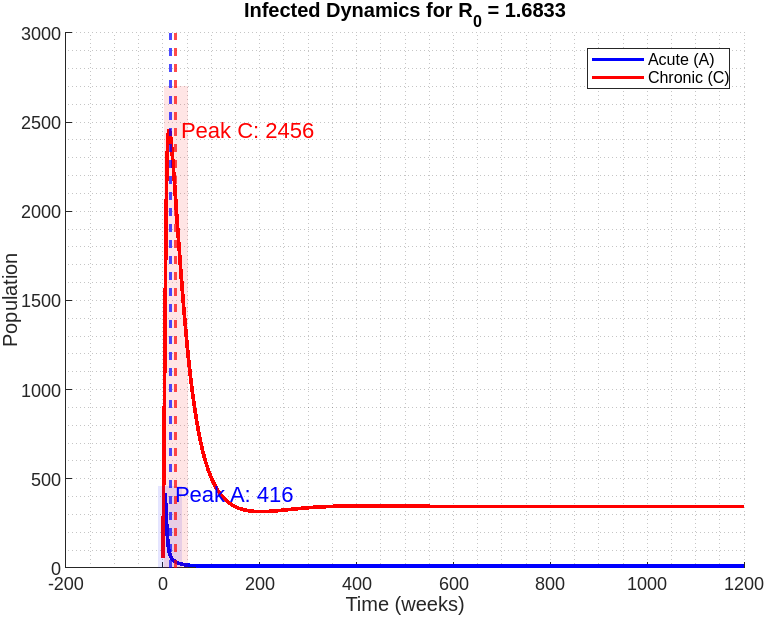}
        \caption*{\textbf{(b)} Infected dynamics under high $\mathcal{R}_0 = 1.6833$. Both acute and chronic infections persist at elevated levels, with chronic infection stabilizing at an endemic state.}
        \label{fig:highR0_sub}
    \end{minipage}

    \caption{Comparison of infected population dynamics under different basic reproduction numbers. Subfigure (a) demonstrates disease fade-out for $\mathcal{R}_0 < 1$, whereas (b) shows sustained transmission and endemic persistence for $\mathcal{R}_0 > 1$.}
    \label{fig:R0_all}
\end{figure}

\newpage
\subsection{Graphical illustration of optimal control analysis}
In this section, we analyses the impact of vaccination \(u_1\) and treatment \(u_2\) on the susceptible populations as well as on the infected populations in the case of optimal condition of preventive measures \(u_1\) and \(u_2\). To obtain these optimal conditions, we used the Pontryagin's maximum principle. In this method, we defined an objective functional called Hamiltonian H, subject to constraints are the equations of the model \eqref{model eqs.}. Here, cost weights are fixed within the objective functional \ref{Hamiltonian} and the simulation period is T = 1500 days. Then, using the adjoint equations and optimality condition, we obtained the optimal values of control measures \(u_1\) and \(u_2\) as \({u_1}^*\) and \({u_2}^*\). We used an iterative algorithm, the "forward-backward sweep method" to apply this principle. Based on this analysis, we present three strategies based on either single or combined control interventions.

\subsubsection{Optimal Strategy 1: Usage of vaccination (\texorpdfstring{$u_1$}{u1}) as preventive measure}
In this strategy, vaccination is applied to the susceptible population as an optimized control. From the graphs in Figure \ref{fig:control01}, it is evident that vaccination significantly reduces the number of acute and chronic infections, while boosting the recovered population. The control profile in Figure \ref{fig:profile01} shows that vaccination is most intense when infections are high and tapers off as the infected population declines over time.

\begin{figure}[htbp]
    \centering
    \includegraphics[width=0.8\linewidth]{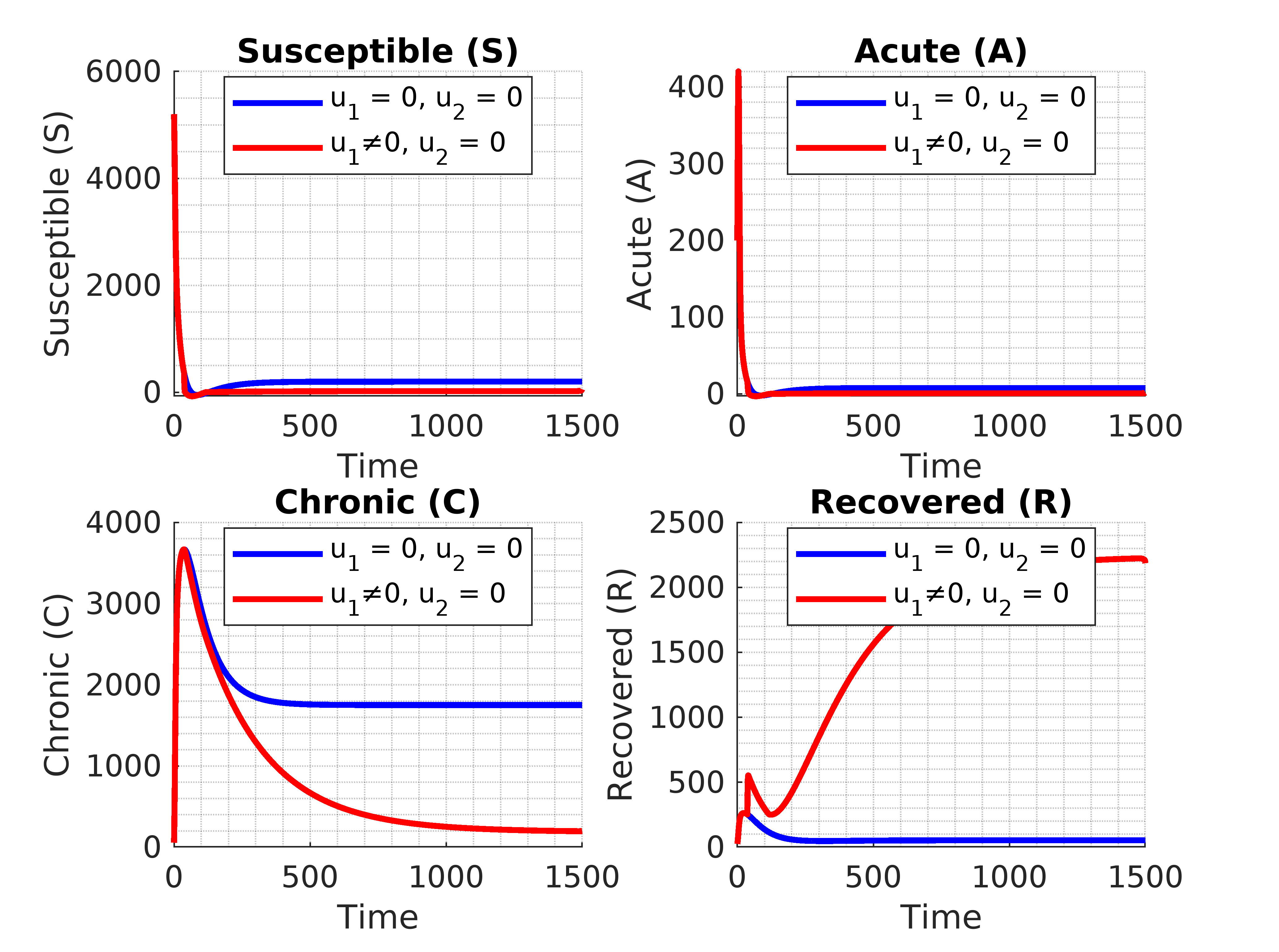}
    \caption{Impact of vaccination on disease dynamics. Blue curves: uncontrolled dynamics. Red curves: reduced acute and chronic cases with increased recovery under optimal vaccination.}
    \label{fig:control01}
\end{figure}

\begin{figure}[htbp]
    \centering
    \includegraphics[width=0.6\linewidth]{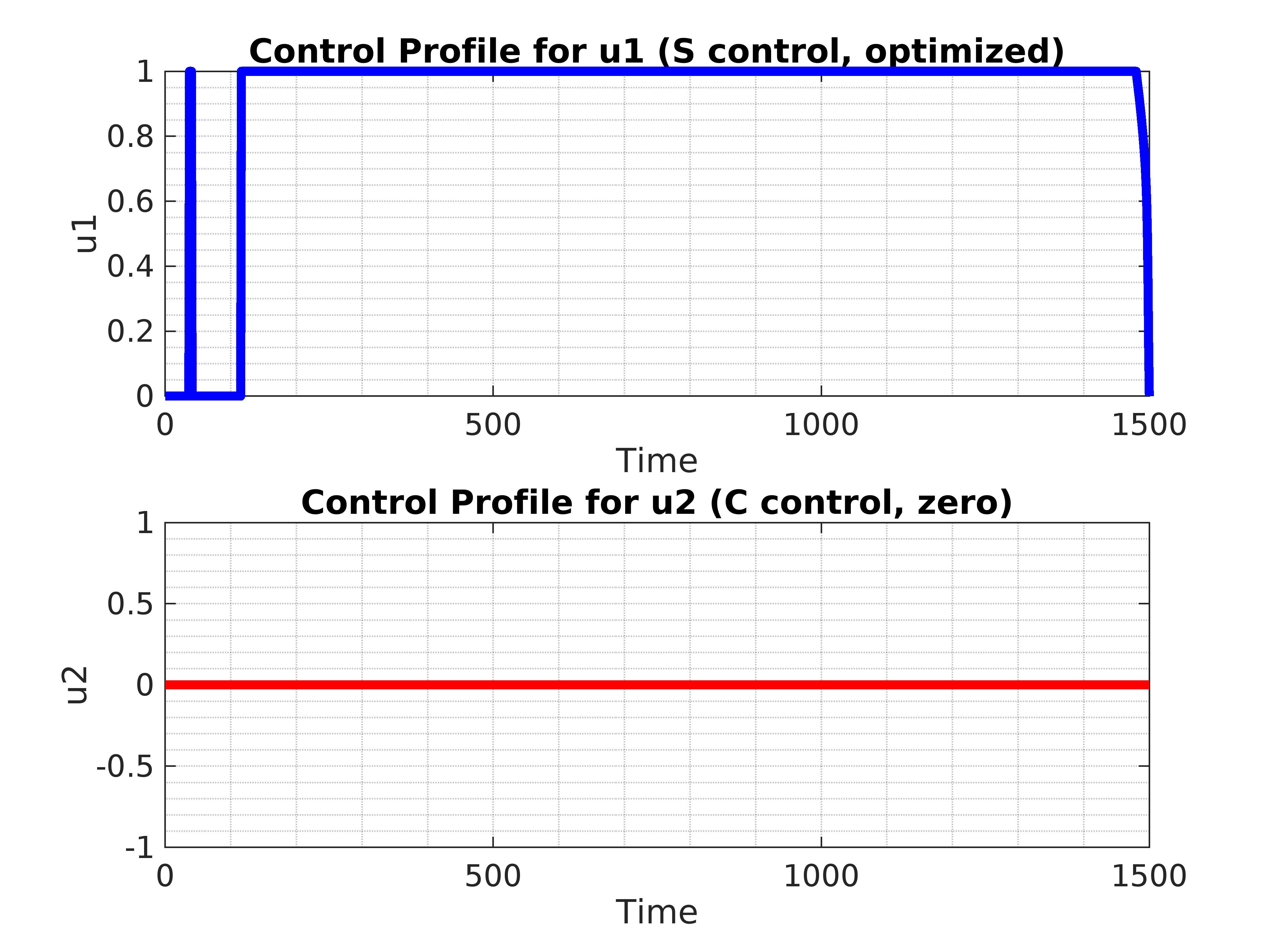}
    \caption{Control profile for \(u_1(t)\) (vaccination). Control effort is high initially and decreases as the number of infected individuals falls.}
    \label{fig:profile01}
\end{figure}

\subsubsection{Optimal Strategy 2: Usage of drugs treatment (\texorpdfstring{$u_2$}{u2}) as prevention of disease infection}
In this strategy, treatment is applied to the chronically infected population as an optimised control. From the graphs illustrated below in the fig.\ref{fig:control02}, we can observe that the treatment alone can reduce infections and promote recovery, although the effect is less pronounced than combined strategies. Fig. \ref{fig:profile02} shows the usage of vaccination over time, when it is needed most, and declines gradually after some time when there are very less number of infected people in the community.

\begin{figure}[htbp]
    \centering
    \includegraphics[width=0.8\linewidth]{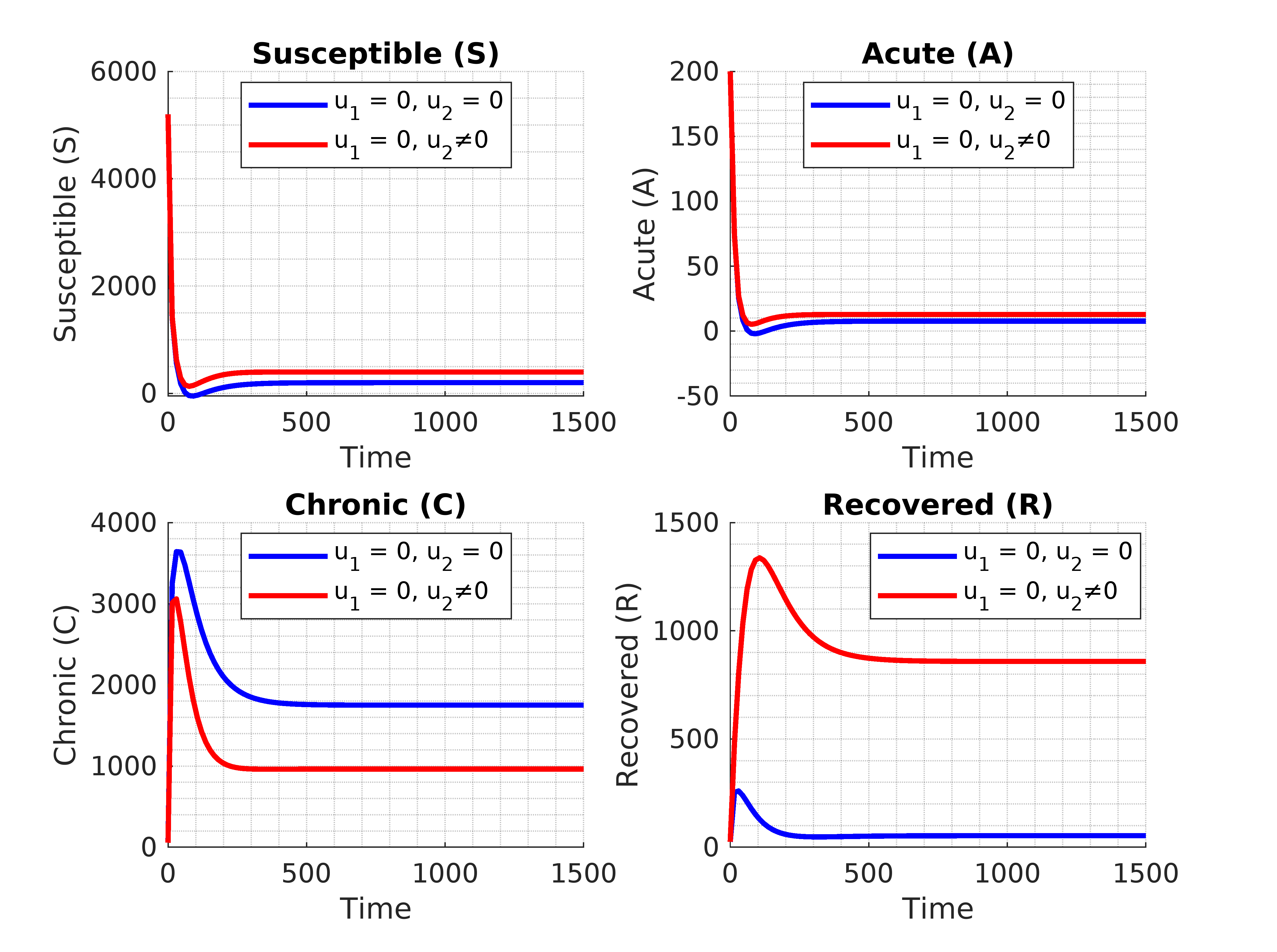}
    \caption{Impact of treatment on disease dynamics. Blue curves: no control. Red curves: improved recovery and reduced chronic cases under optimal treatment.}
    \label{fig:control02}
\end{figure}

\begin{figure}[htbp]
    \centering
    \includegraphics[width=0.6\linewidth]{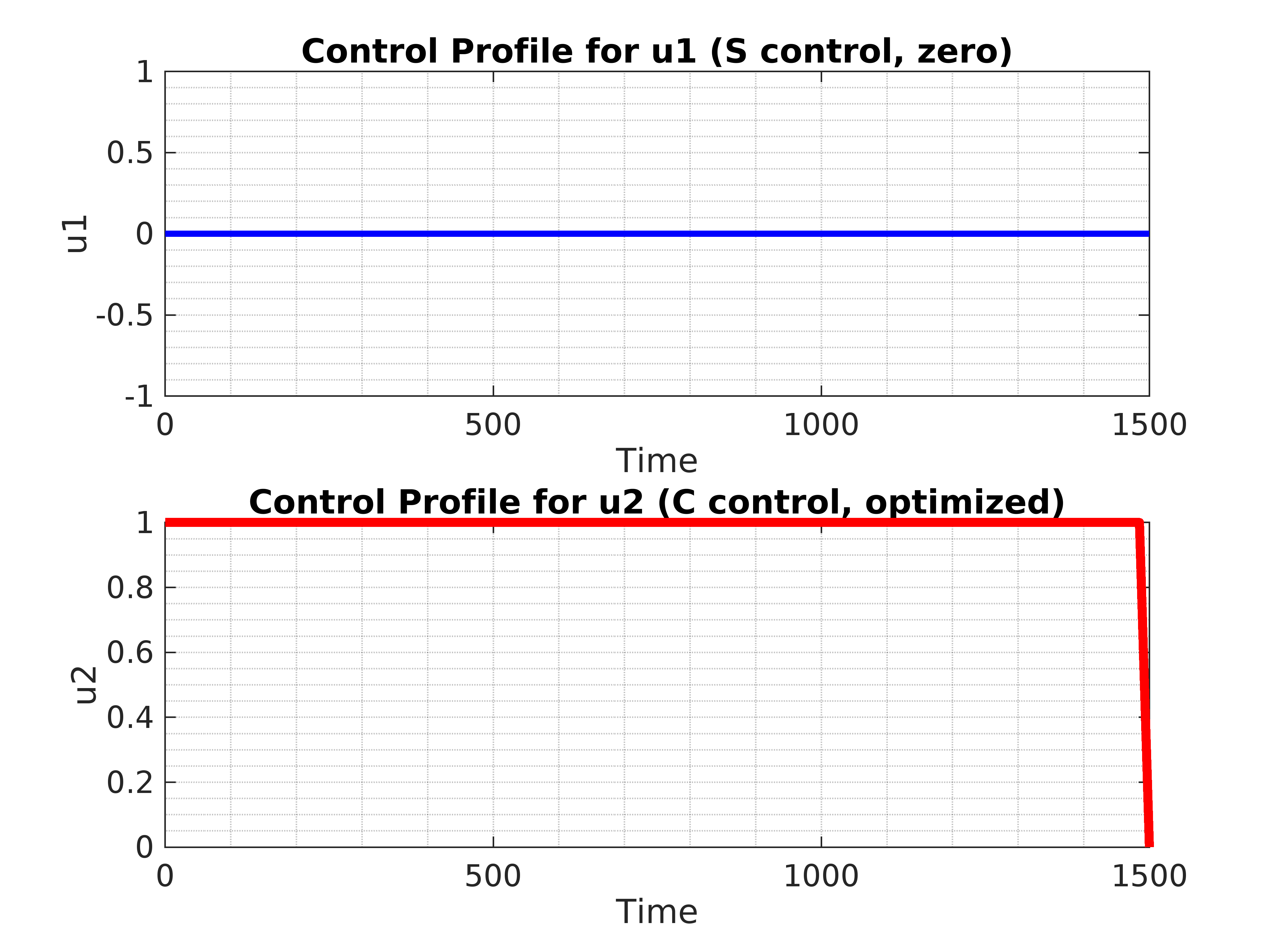}
    \caption{Control profile for \(u_2(t)\) (treatment).}
    \label{fig:profile02}
\end{figure}

\subsubsection{Optimal Strategy 3: Usage of vaccination (\texorpdfstring{$u_1$}{u1}) and drugs treatment (\texorpdfstring{$u_2$}{u2}) as combined prevention of disease infection}
In this strategy, we used vaccination on the susceptible compartment as an optimised control along with treatment. From the graphs illustrated below in the fig. \ref{fig:control12}, we can observe that the impact of vaccination on the population is significant and caused a steep decrement in the infected populations and increased the number of recovered populations. Fig. \ref{fig:profile12} shows the usage of vaccination over time, when it was needed most, and declines gradually after some time when there are very few numbers of infected individuals in the community. Comparing the optimal strategies for each single control, which comprises only vaccination and only treatment with multiple control measures, vaccination, and treatment both simultaneously, we got superior outcomes than the former strategies. By applying single preventive measures, we observe that there is a significant reduction in some compartments, like when we employ vaccination or treatment, we achieve considerable success in reducing susceptibility to infection or reducing the infection in the population. However, when we employed both control strategies simultaneously, infections were suppressed more effectively, and the recovered class grew significantly, indicating synergistic benefits of the combined intervention.

\begin{figure}[ht]
    \centering
    \includegraphics[width=0.8\linewidth]{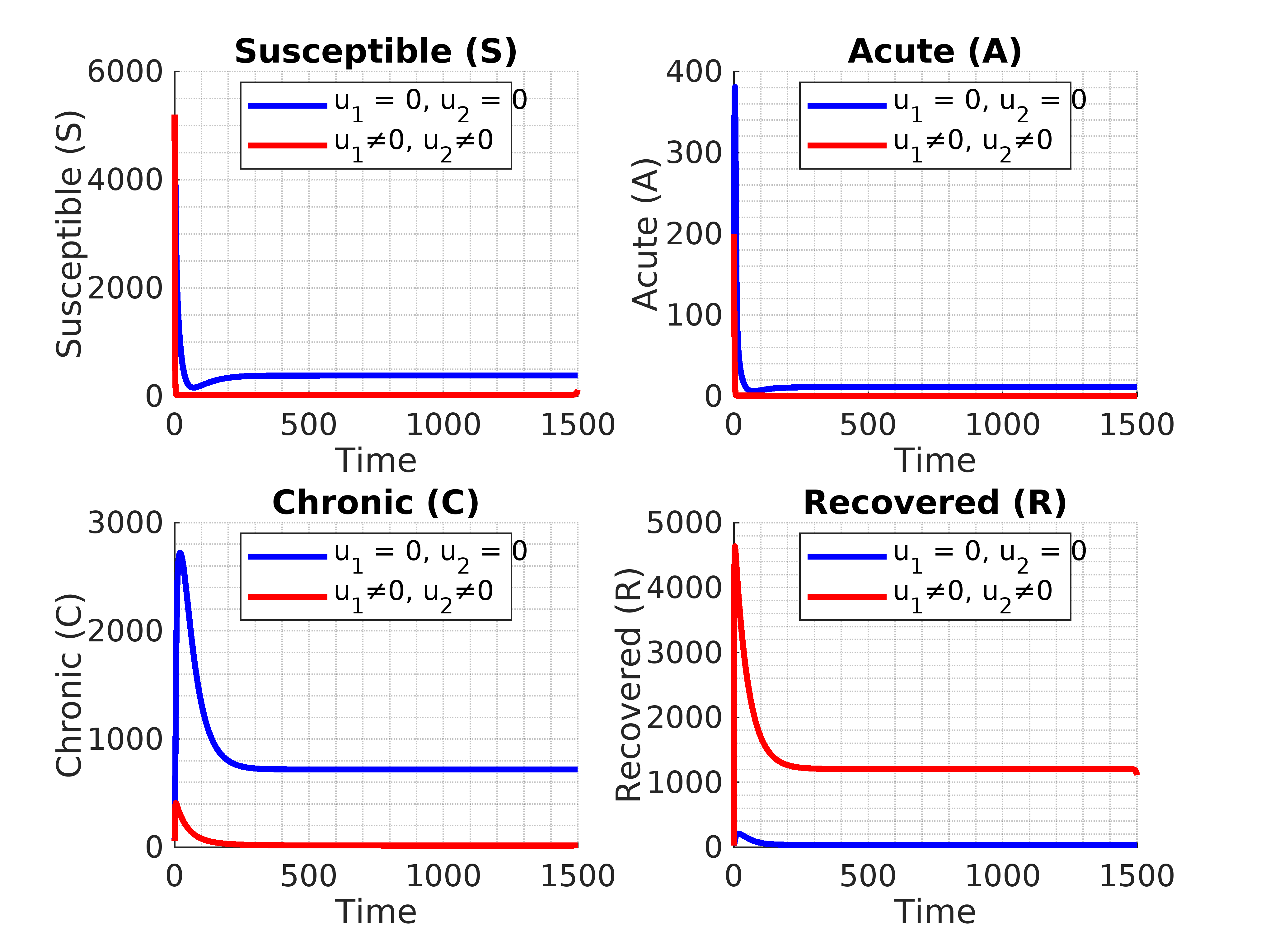}
    \caption{Impact of optimal usage of combined controls (vaccination and treatment) on the system. Blue curves correspond to outbreak dynamics without any controls, while red curves indicate a substantial reduction in the infected compartments (acute and chronic), with strong recovery.}
    \label{fig:control12}
\end{figure}

\begin{figure}[htbp]
    \centering
    \includegraphics[width=0.6\linewidth]{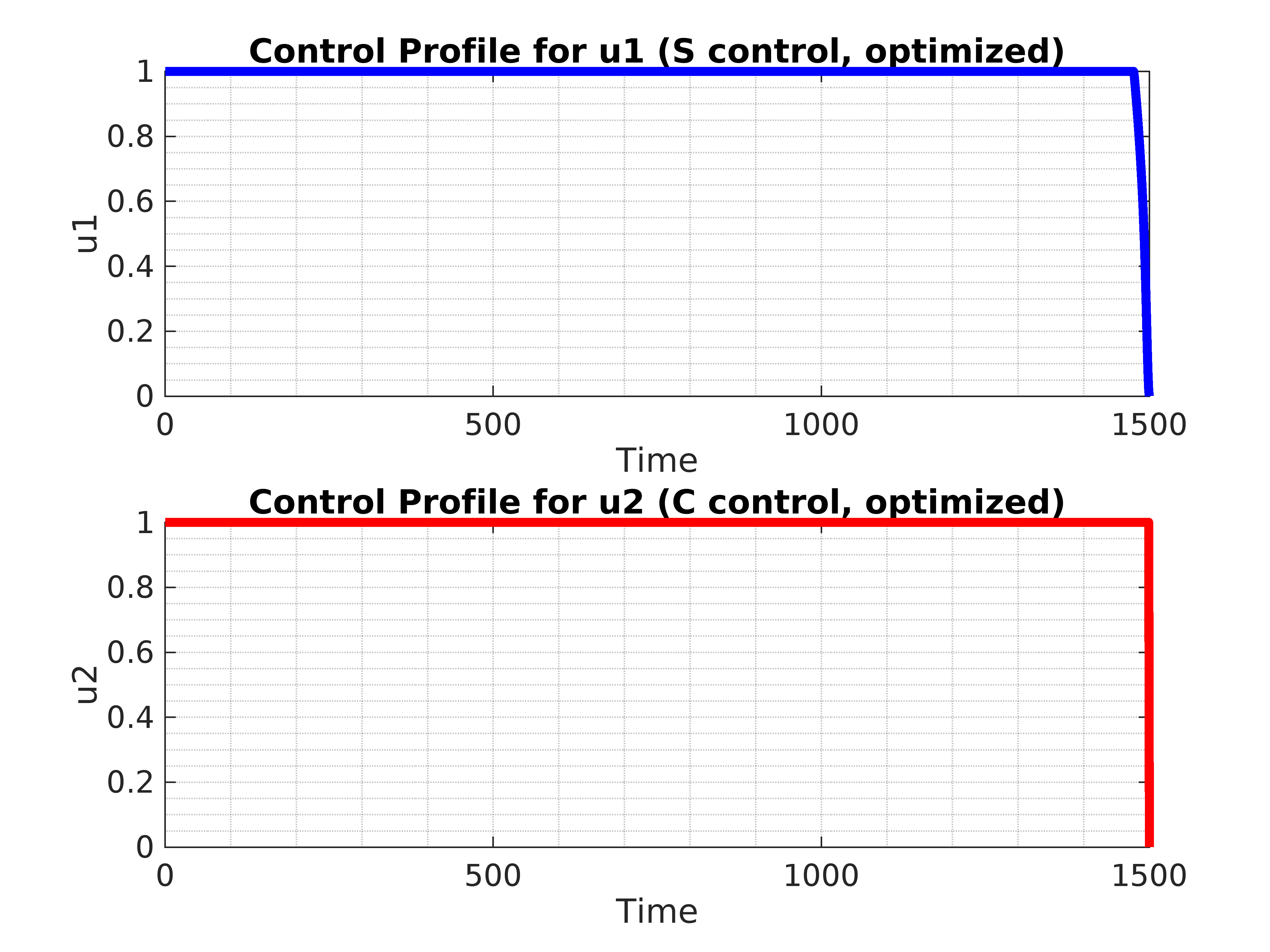}
    \caption{Control profiles for both \(u_1(t)\) (vaccination) and \(u_2(t)\) (treatment).}
    \label{fig:profile12}
\end{figure}

\pagebreak
\section{Conclusion}
Hepatitis B virus (HBV) infection remains endemic in numerous regions across the globe and poses a significant public health issue. The different infection transmission rates in the host population and vertical transmission, especially perinatal transmission from carrier mothers, are the two mainstays of HBV transmission dynamics. This trend of transmission is a major factor in the chronicity of HBV, which is particularly high in endemic areas where infantile exposure to the virus greatly enhances the risk of chronic infection. The persistence of HBV in the population, therefore, cannot be fully understood without accounting for these different kinds of transmission of disease and vertical transmission, both of which emphasize the need for targeted intervention strategies such as vaccination and drug treatment. Addressing these facets, we developed a novel mathematical model framework aimed at evaluating the spread of infection and the optimal control of HBV disease. The model's dynamics are rigorously analyzed through the computation of the basic reproduction number $R_0$ using the next-generation matrix method, followed by a global stability analysis of both the disease-free and endemic equilibrium points. To optimize the implementation of intervention strategies, we apply Pontryagin’s Maximum Principle, designing multiple control policies to minimize both the disease burden and economic cost. This work contributes a more biologically accurate and analytically rich model to the literature, offering new insights into the strategic control of HBV infection, particularly in regions where vaccination and treatment resources are limited.

\section*{Acknowledgment} The first author acknowledges financial support from the Maulana Azad National Fellowship (MANF), University Grants Commission (UGC), Government of India (Award No. F.No. PLNG/13/2023-PLANNING-NMDFC; Ref. No. 211610001674).

\section*{Authors contributions} All authors have contributed equally.

\section*{Data availability} This manuscript has no associated data.

\section*{Declarations}
\section*{Conflict of interests} There is no conflict of interest regarding this work.

\bibliographystyle{ieeetr}
\bibliography{biblog}

@article{kermack1927contribution,
  title={A contribution to the mathematical theory of epidemics},
  author={Kermack, William Ogilvy and McKendrick, Anderson G},
  journal={Proceedings of the royal society of london. Series A, Containing papers of a mathematical and physical character},
  volume={115},
  number={772},
  pages={700--721},
  year={1927},
  publisher={The Royal Society London}
}

@book{anderson1991infectious,
  title={Infectious diseases of humans: dynamics and control},
  author={Anderson, Roy M and May, Robert M},
  year={1991},
  publisher={Oxford university press}
}

@misc{who2024global,
  title={Global Hepatitis Report 2024: Action for Access in Low-and Middle-Income Countries},
  author={WHO},
  year={2024},
  publisher={World Health Organization Geneva, Switzerland}
}

@article{easterbrook20242024,
  title={WHO 2024 hepatitis B guidelines: an opportunity to transform care},
  author={Easterbrook, Philippa J and Luhmann, Niklas and Bajis, Sahar and Min, Myat Sandi and Newman, Morkor and Lesi, Olufunmilayo and Doherty, Meg C},
  journal={The lancet Gastroenterology \& hepatology},
  volume={9},
  number={6},
  pages={493--495},
  year={2024},
  publisher={Elsevier}
}

@article{seeger2000hepatitis,
  title={Hepatitis B virus biology},
  author={Seeger, Christoph and Mason, William S},
  journal={Microbiology and molecular biology reviews},
  volume={64},
  number={1},
  pages={51--68},
  year={2000},
  publisher={American Society for Microbiology}
}

@article{cui2023global,
  title={Global reporting of progress towards elimination of hepatitis B and hepatitis C},
  author={Cui, Fuqiang and Blach, Sarah and Mingiedi, Casimir Manzengo and Gonzalez, Monica Alonso and Alaama, Ahmed Sabry and Mozalevskis, Antons and S{\'e}guy, Nicole and Rewari, Bharat Bhushan and Chan, Po-Lin and Le, Linh-vi and others},
  journal={The lancet Gastroenterology \& hepatology},
  volume={8},
  number={4},
  pages={332--342},
  year={2023},
  publisher={Elsevier}
}

@article{world2024guidelines,
  title={Guidelines for the prevention, diagnosis, care and treatment for people with chronic hepatitis B infection (text extract): executive summary},
  author={World Health Organization},
  journal={Infectious Diseases \& Immunity},
  volume={4},
  number={03},
  pages={103--105},
  year={2024},
  publisher={Chinese Medical Association Publishing House Co., Ltd 42 Dongsi Xidajie~…}
}

@article{worldint,
  title={www. who. int/news-room/fact-sheets/detail/hepatitis-b},
  author={World Health Organization},
  year={2024},

}

@inproceedings{beard2024combined,
  title={Combined “test and treat” campaigns for human immunodeficiency virus, hepatitis B, and hepatitis C: a systematic review to provide evidence to support World Health Organization treatment guidelines},
  author={Beard, Natasha and Hill, Andrew},
  booktitle={Open Forum Infectious Diseases},
  volume={11},
  number={2},
  pages={ofad666},
  year={2024},
  organization={Oxford University Press US}
}

@article{mann2011modelling,
  title={Modelling the epidemiology of hepatitis B in New Zealand},
  author={Mann, Joanne and Roberts, Mick},
  journal={Journal of theoretical biology},
  volume={269},
  number={1},
  pages={266--272},
  year={2011},
  publisher={Elsevier}
}

@article{hethcote2000mathematics,
  title={The mathematics of infectious diseases},
  author={Hethcote, Herbert W},
  journal={SIAM review},
  volume={42},
  number={4},
  pages={599--653},
  year={2000},
  publisher={SIAM}
}

@article{castillo2002computation,
  title={On the computation of r. And its role on global stability carlos castillo-chavez∗, zhilan feng, and wenzhang huang},
  author={Castillo-Chavez, Carlos},
  journal={Mathematical approaches for emerging and reemerging infectious diseases: an introduction},
  volume={1},
  pages={229},
  year={2002}
}

@article{diekmann1990definition,
  title={On the definition and the computation of the basic reproduction ratio R 0 in models for infectious diseases in heterogeneous populations},
  author={Diekmann, Odo and Heesterbeek, Johan Andre Peter and Metz, Johan Anton Jacob},
  journal={Journal of mathematical biology},
  volume={28},
  pages={365--382},
  year={1990},
  publisher={Springer}
}

@book{diekmann2000mathematical,
  title={Mathematical epidemiology of infectious diseases: model building, analysis and interpretation},
  author={Diekmann, Odo and Heesterbeek, Johan Andre Peter},
  volume={5},
  year={2000},
  publisher={John Wiley \& Sons}
}

@article{diekmann2010construction,
  title={The construction of next-generation matrices for compartmental epidemic models},
  author={Diekmann, Odo and Heesterbeek, Johan Andre Peter and Roberts, Michael G},
  journal={Journal of the royal society interface},
  volume={7},
  number={47},
  pages={873--885},
  year={2010},
  publisher={The Royal Society}
}

@article{van2002reproduction,
  title={Reproduction numbers and sub-threshold endemic equilibria for compartmental models of disease transmission},
  author={Van den Driessche, Pauline and Watmough, James},
  journal={Mathematical biosciences},
  volume={180},
  number={1-2},
  pages={29--48},
  year={2002},
  publisher={Elsevier}
}

@book{castillo2002mathematical,
  title={Mathematical approaches for emerging and reemerging infectious diseases: models, methods, and theory},
  author={Castillo-Chavez, Carlos and Blower, Sally and van den Driessche, Pauline and Kirschner, Denise and Yakubu, Abdul-Aziz},
  volume={126},
  year={2002},
  publisher={Springer Science \& Business Media}
}

@article{raezah2023exploring,
  title={Exploring the effectiveness of control measures and long-term behavior in Hepatitis B: An analysis of an endemic model with horizontal and vertical transmission},
  author={Raezah, Aeshah A and Raouf, Abdur and Zarin, Rahat and Khan, Amir},
  journal={Results in Physics},
  volume={53},
  pages={106966},
  year={2023},
  publisher={Elsevier}
}

@article{khan2023modelling,
  title={Modelling the dynamics of acute and chronic hepatitis B with optimal control},
  author={Khan, Tahir and Rihan, Fathalla A and Ahmad, Hijaz},
  journal={Scientific reports},
  volume={13},
  number={1},
  pages={14980},
  year={2023},
  publisher={Nature Publishing Group UK London}
}

@article{gaff2009optimal,
  title={Optimal control applied to vaccination and treatmentstrategies for various epidemiological models},
  author={Gaff, Holly and Schaefer, Elsa},
  journal={Mathematical biosciences \& engineering},
  volume={6},
  number={3},
  pages={469--492},
  year={2009},
  publisher={Mathematical Biosciences \& Engineering}
}

@inproceedings{lasalle1976stability,
  title={The stability of dynamical systems, society for industrial and applied mathematics, Philadelphia, pa., 1976},
  author={LaSalle, JP},
  booktitle={With an appendix:“Limiting equations and stability of nonautonomous ordinary differential equations” by Z. Artstein, Regional Conference Series in Applied Mathematics},
  year={1976}
}

@article{khan2016classification,
  title={Classification of different hepatitis B infected individuals with saturated incidence rate},
  author={Khan, Tahir and Zaman, Gul},
  journal={SpringerPlus},
  volume={5},
  number={1},
  pages={1082},
  year={2016},
  publisher={Springer}
}

@article{khan2017transmission,
  title={The transmission dynamic and optimal control of acute and chronic hepatitis B},
  author={Khan, Tahir and Zaman, Gul and Chohan, M Ikhlaq},
  journal={Journal of biological dynamics},
  volume={11},
  number={1},
  pages={172--189},
  year={2017},
  publisher={Taylor \& Francis}
}

@incollection{bhadauria2022epidemic,
  title={Epidemic theory: Studying the effective and basic reproduction numbers, epidemic thresholds and techniques for the analysis of infectious diseases with particular emphasis on tuberculosis},
  author={Bhadauria, Archana Singh and Dhungana, Hom Nath},
  booktitle={Methods of Mathematical Modeling},
  pages={1--21},
  year={2022},
  publisher={Elsevier}
}

@book{la1976stability,
  title={The stability of dynamical systems},
  author={La Salle, Joseph P},
  year={1976},
  publisher={SIAM}
}

@article{bhadauria2021siq,
  title={A SIQ mathematical model on COVID-19 investigating the lockdown effect},
  author={Bhadauria, Archana Singh and Pathak, Rachana and Chaudhary, Manisha},
  journal={Infectious Disease Modelling},
  volume={6},
  pages={244--257},
  year={2021},
  publisher={Elsevier}
}

@article{GUMEL2003409,
title = {A qualitative study of a vaccination model with non-linear incidence},
journal = {Applied Mathematics and Computation},
volume = {143},
number = {2},
pages = {409-419},
year = {2003},
issn = {0096-3003},
doi = {https://doi.org/10.1016/S0096-3003(02)00372-7},
url = {https://www.sciencedirect.com/science/article/pii/S0096300302003727},
author = {A.B. Gumel and S.M. Moghadas}
}

@incollection{yadav2025comparative,
  title={A comparative study of different disease-incidence functions in the mathematical modeling of infectious diseases},
  author={Yadav, Sudha and Pandey, Aditya and Bhadauria, Archana Singh},
  booktitle={Mathematical Methods in Medical and Biological Sciences},
  pages={141--157},
  year={2025},
  publisher={Elsevier}
}

@article{zarin2021fractional,
  title={Fractional modeling and optimal control analysis of rabies virus under the convex incidence rate},
  author={Zarin, Rahat and Ahmed, Iftikhar and Kumam, Poom and Zeb, Anwar and Din, Anwarud},
  journal={Results in Physics},
  volume={28},
  pages={104665},
  year={2021},
  publisher={Elsevier}
}

@article{thornley2008hepatitis,
  title={Hepatitis B in a high prevalence New Zealand population: a mathematical model applied to infection control policy},
  author={Thornley, Simon and Bullen, Chris and Roberts, Mick},
  journal={Journal of Theoretical Biology},
  volume={254},
  number={3},
  pages={599--603},
  year={2008},
  publisher={Elsevier}
}

@article{zou2010modeling,
  title={Modeling the transmission dynamics and control of hepatitis B virus in China},
  author={Zou, Lan and Zhang, Weinian and Ruan, Shigui},
  journal={Journal of theoretical biology},
  volume={262},
  number={2},
  pages={330--338},
  year={2010},
  publisher={Elsevier}
}

@article{yadav2024modelling,
  title={Modelling the Effect of Multiple Interventions to Balance Healthcare Demand for Controlling Infectious Disease Outbreaks},
  author={Yadav, Sudha and Bhadauria, Archana Singh},
  journal={International Journal of Biomathematics},
  year={2024},
  publisher={World Scientific}
}

@article{khan2021modeling,
  title={Modeling and sensitivity analysis of HBV epidemic model with convex incidence rate},
  author={Khan, Amir and Zarin, Rahat and Hussain, Ghulam and Usman, Auwalu Hamisu and Humphries, Usa Wannasingha and Gomez-Aguilar, JF},
  journal={Results in Physics},
  volume={22},
  pages={103836},
  year={2021},
  publisher={Elsevier}
}

@article{pandey2025impact,
  title={Impact of COVID-19 vaccination and self-protection strategies for intervention of co-infection of COVID-19 and influenza},
  author={Pandey, Aditya and Yadav, Sudha and Verma, Vijai Shanker and Bhadauria, Archana Singh},
  journal={MESA},
  volume={16},
  number={2},
  pages={593--616},
  year={2025}
}

@article{lashari2016optimal,
  title={Optimal control of an SIR epidemic model with a saturated treatment},
  author={Lashari, Abid Ali},
  journal={Applied Mathematics \& Information Sciences},
  volume={10},
  number={1},
  pages={185},
  year={2016},
  publisher={Natural Sciences Publishing Corp}
}

@article{alrabaiah2020optimal,
  title={Optimal control analysis of hepatitis B virus with treatment and vaccination},
  author={Alrabaiah, Hussam and Safi, Mohammad A and DarAssi, Mahmoud H and Al-Hdaibat, Bashir and Ullah, Saif and Khan, Muhammad Altaf and Shah, Syed Azhar Ali},
  journal={Results in Physics},
  volume={19},
  pages={103599},
  year={2020},
  publisher={Elsevier}
}

@article{wang2019optimal,
  title={Optimal vaccination strategy of a constrained time-varying SEIR epidemic model},
  author={Wang, Xinwei and Peng, Haijun and Shi, Boyang and Jiang, Dianheng and Zhang, Sheng and Chen, Biaosong},
  journal={Communications in Nonlinear Science and Numerical Simulation},
  volume={67},
  pages={37--48},
  year={2019},
  publisher={Elsevier}
}

@article{freddi2020optimal,
  title={Optimal control of the transmission rate in compartmental epidemics},
  author={Freddi, Lorenzo},
  journal={arXiv preprint arXiv:2007.00318},
  year={2020}
}

@book{martcheva2015introduction,
  title={An introduction to mathematical epidemiology},
  author={Martcheva, Maia},
  volume={61},
  year={2015},
  publisher={Springer}
}

@article{molina2022optimal,
  title={An optimal feedback control that minimizes the epidemic peak in the SIR model under a budget constraint},
  author={Molina, Emilio and Rapaport, Alain},
  journal={Automatica},
  volume={146},
  pages={110596},
  year={2022},
  publisher={Elsevier}
}

@book{lenhart2007optimal,
  title={Optimal control applied to biological models},
  author={Lenhart, Suzanne and Workman, John T},
  year={2007},
  publisher={Chapman and Hall/CRC}
}

@book{pontryagin2018mathematical,
  title={Mathematical theory of optimal processes},
  author={Pontryagin, Lev Semenovich},
  year={2018},
  publisher={Routledge}
}

@article{qian2020sensitivity,
  title={Sensitivity analysis methods in the biomedical sciences},
  author={Qian, George and Mahdi, Adam},
  journal={Mathematical biosciences},
  volume={323},
  pages={108306},
  year={2020},
  publisher={Elsevier}
}

@article{gomero2012latin,
  title={Latin hypercube sampling and partial rank correlation coefficient analysis applied to an optimal control problem},
  author={Gomero, Boloye},
  year={2012}
}

@article{butt2023computational,
  title={Computational analysis of control of hepatitis B virus disease through vaccination and treatment strategies},
  author={Butt, Azhar Iqbal Kashif and Imran, Muhammad and Aslam, Javeria and Batool, Saira and Batool, Saira},
  journal={Plos one},
  volume={18},
  number={10},
  pages={e0288024},
  year={2023},
  publisher={Public Library of Science San Francisco, CA USA}
}

@article{korobeinikov2006lyapunov,
  title={Lyapunov functions and global stability for SIR and SIRS epidemiological models with non-linear transmission},
  author={Korobeinikov, Andrei},
  journal={Bulletin of Mathematical biology},
  volume={68},
  pages={615--626},
  year={2006},
  publisher={Springer}
}

@article{dejesus1987routh,
  title={Routh-Hurwitz criterion in the examination of eigenvalues of a system of nonlinear ordinary differential equations},
  author={DeJesus, Edmund X and Kaufman, Charles},
  journal={Physical Review A},
  volume={35},
  number={12},
  pages={5288},
  year={1987},
  publisher={APS}
}

@article{patil2021routh,
  title={Routh-Hurwitz criterion for stability: an overview and its implementation on characteristic equation vectors using MATLAB},
  author={Patil, Aseem},
  journal={Emerging Technologies in Data Mining and Information Security: Proceedings of IEMIS 2020, Volume 1},
  pages={319--329},
  year={2021},
  publisher={Springer}
}

@article{hurwitz1964conditions,
  title={On the conditions under which an equation has only roots with negative real parts},
  author={Hurwitz, Adolf and others},
  journal={Selected papers on mathematical trends in control theory},
  volume={65},
  pages={273--284},
  year={1964},
  publisher={Dover New York}
}

@book{lancaster1985theory,
  title={The theory of matrices: with applications},
  author={Lancaster, Peter and Tismenetsky, Miron},
  year={1985},
  publisher={Elsevier}
}

\end{document}